\documentclass[draftclsnofoot,onecolumn]{IEEEtran}
\usepackage{def}
\begin{document}

\title{Accelerated Decentralized Constraint-Coupled Optimization: A Dual$^2$ Approach}

\author{Jingwang Li and Vincent Lau, \IEEEmembership{Fellow, IEEE}
  \thanks{This work was supported by the Research Grants Council under the Areas of Excellence scheme grants AoE/E-601/22-R and AoE/E-101/23-N. (\textit{Corresponding author: Vincent Lau.})}
  \thanks{The authors are with the Department of Electronic and Computer Engineering, The Hong Kong University of Science and Technology, Hong Kong 99077, China (e-mail: jingwangli@outlook.com; eeknlau@ust.hk).}
}

\maketitle

\begin{abstract}
  In this paper, we focus on a class of decentralized constraint-coupled optimization problem: $\min_{x_i \in \mathbb{R}^{d_i}, i \in \mathcal{I}; y \in \mathbb{R}^p}$ $\sum_{i=1}^n\left(f_i(x_i) + g_i(x_i)\right) + h(y) \ \text{s.t.} \ \sum_{i=1}^{n}A_ix_i = y$, over an undirected and connected network of $n$ agents. Here, $f_i$, $g_i$, and $A_i$ represent private information of agent $i \in \mathcal{I} = \{1, \cdots, n\}$, while $h$ is public for all agents. Building on a novel dual$^2$ approach, we develop two accelerated algorithms to solve this problem: the inexact Dual$^2$ Accelerated (iD2A) gradient method and the Multi-consensus inexact Dual$^2$ Accelerated (MiD2A) gradient method. We demonstrate that both iD2A and MiD2A can guarantee asymptotic convergence under a milder condition on $h$ compared to existing algorithms. Furthermore, under additional assumptions, we establish linear convergence rates and derive significantly lower communication and computational complexity bounds than those of existing algorithms. Several numerical experiments validate our theoretical analysis and demonstrate the practical superiority of the proposed algorithms.
\end{abstract}

\begin{IEEEkeywords}
  Decentralized optimization, coupled constraints, dual$^2$ approach, accelerated algorithms, inexact methods, linear convergence.
\end{IEEEkeywords}

\IEEEpeerreviewmaketitle

\section{Introduction}
In this paper, we focus on the following decentralized optimization problem:
\begin{equation} \label{main_pro} \tag{P1}
  \begin{aligned}
    \min_{x_i \in \R^{d_i}, i \in \sI; y \in \R^p} \  & \sum_{i=1}^n\pa{f_i(x_i) + g_i(x_i)} + h(y) \\
    \text{s.t.} \                                     & \sum_{i=1}^{n}A_ix_i = y,
  \end{aligned}
\end{equation}
where $f_i: \R^{d_i} \rightarrow \R$ is a smooth and convex function, $g_i:\R^{d_i} \rightarrow \exs$ and $h:\R^p \rightarrow \exs$ are convex but possibly nonsmooth functions, and $A_i \in \R^{p \times d_i}$, $i \in \sI$. Without loss of generality, we assume that the (optimal) solution of \cref{main_pro} exists. We should mention that, \cref{main_pro} may appear in some literature \cite{alghunaim2021dual,li2024npga} in another equivalent form:
\begin{equation} \label{main_pro2}
  \begin{aligned}
    \min_{x_i \in \R^{d_i}, i \in \sI} \ \sum_{i=1}^n \pa{f_i(x_i) + g_i(x_i)} + h\pa{\sum_{i=1}^{n}A_ix_i}.
  \end{aligned}
\end{equation}
Under the decentralized setting, we are required to solve \cref{main_pro} over an undirected and connected network of $n$ agents, where $f_i$, $g_i$, and $A_i$ are private information of agent $i \in \sI$ and $h$ is public for all agents. The network topology is represented by an undirected graph $\mathcal{G}=(\sI, \mathcal{E})$, where $\sI = \{1,\cdots, n\}$ and $\mathcal{E} \subset \sI \times \sI$ denote the sets of nodes and edges, respectively. Specifically, $(i, j) \in \mathcal{E}$ if agents $i$ and $j$ can communicate with each other. Since the decision variables of all agents are coupled by an equality constraint, \cref{main_pro} is referred to as a decentralized/distributed\footnote{In control community, ``decentralized'' and ``distributed'' are often interchangeable, both describing systems without a central coordinator. However, in fields like computer science, ``distributed'' may still imply centralized coordination. To avoid ambiguity, this paper consistently uses ``decentralized''.} constraint-coupled optimization problem in many references \cite{falsone2020tracking,li2022implicit,li2024npga}. \cref{main_pro} has extensive applications, including decentralized resource allocation \cite{doostmohammadian2023distributed}, decentralized model predictive control \cite{yfantis2024distributed}, and decentralized learning \cite{liu2024vertical}; see \cref{appendix_applications} for further details. The primary goal of this work is to develop accelerated decentralized algorithms that provide both theoretical and experimental guarantees for solving \cref{main_pro} under mild conditions, leveraging Nesterov's acceleration technique \cite{nesterov2018lectures}.

\begin{table*}[t!]
  \begin{threeparttable}[b]
    \caption{Conditions required by different algorithms to guarantee convergence\tnote{1}.}
    \label{2024}
    \scriptsize
    \tabcolsep=10pt
    \renewcommand{\arraystretch}{0.5}
    \centering
    \begin{tabularx}{\textwidth}{cCCCC}
      \toprule
                                                                               & $f_i$        & $g_i$                          & $h$               \\
      \midrule
      \\
      IDC-ADMM \cite{chang2014multi}                                           & C1A\tnote{2} & C1B\tnote{3}                   & $\iota_{\set{b}}$ \\
      \midrule
      Tracking-ADMM \cite{falsone2020tracking}, NECPD \cite{su2021distributed} & C2\tnote{4}  & $\iota_{\sX_i}$, $\sX_i$ is CP & $\iota_{\set{b}}$ \\
      \midrule
      Proj-IDEA \cite{li2022implicit}                                          & C3\tnote{5}  & $\iota_{\sX_i}$, $\sX_i$ is CC & $\iota_{\set{b}}$ \\
      \midrule
      DCDA \cite{alghunaim2019proximal}, NPGA \cite{li2023proximal}            & SC, SM       & CPC                            & $\iota_{\set{b}}$ \\
      \midrule
      \rowcolor{bgcolor}
      Proposed Algorithms: iD2A, MiD2A\tnote{6}                                                     & SC, SM       & CPC                            & CPC               \\
      \bottomrule
    \end{tabularx}
    \begin{tablenotes}
      \item C: convex; SC: strongly convex; SM: smooth; CPC: closed proper convex; CC: closed and convex; CP: closed and compact.
      \item [1] All algorithms presented below exhibit asymptotic convergence, except for NECPD which achieves a sublinear convergence rate of $\bO{\frac{1}{k}}$. It should be noted, however, that NECPD is a two-loop algorithm (when the closed-form solution of the subproblem for each iteration is unavailable), and the convergence rate applies specifically to its outer loop.
      \item [2] C1A: $f_i(x_i) = f'_i(E_ix_i)$, where $f'_i: \R^p \rightarrow \R$ is strongly convex and smooth, and $E_i \in \R^{p \times m}$ with $m = d_1 = \cdots = d_n$. This condition implies that all the $x_i$ and all the $A_i$ has the same dimension.
      \item [3] C1B: $g_i$ is CPC and $\partial g_i(x_i) \neq \emptyset$ for any $x_i \in \dom{g_i}$.
      \item [4] C2: $\dom{f_i} = \R^{d_i}$ and $f_i$ is convex.
      \item [5] C3: $\dom{f_i} = \R^{d_i}$, $f_i$ is convex on $\sX_i$ and differentiable on $\tilde{\sX}_i$, and $\nabla f_i$ is locally Lipschitz on $\tilde{\sX}_i$, where $\tilde{\sX}_i$ is an open set containing $\sX_i$.
      \item [6] \textbf{Throughout this paper, we use color in tables to distinguish our algorithms from existing ones.}
    \end{tablenotes}
  \end{threeparttable}
\end{table*}

\begin{table*}[t]
  \begin{threeparttable}[b]
  \caption{Communication and oracle complexities of various algorithms across different scenarios (full version available in \cref{appendix_full_table}).}
  \label{tab:complexity_h}
  \renewcommand{\arraystretch}{0.5}
  \centering
  \scriptsize
    \begin{tabularx}{\textwidth}{>{\centering\arraybackslash}p{3cm}ccC}
      \toprule
                                                            & \makecell{Additional                                                                                                                                                                                                                                        \\ assumptions}                                                                                                       & Communication complexity                                                                                                         & Oracle complexity\tnote{1} \\
      \midrule
      \multicolumn{4}{c}{\textbf{Case 1: \cref{G,convex,strong_duality,hs_smooth}, $h^*$ is $\mu_{h^*}$-strongly convex}}                                                                                                                                                                                                 \\
      \midrule
      DCPA \cite{alghunaim2021dual}, NPGA \cite{li2024npga} & \textcircled{1}             & $\bO{\max\pa{\kappa_f, \kappa_{pd}\kappa_C} \log \pa{\frac{1}{\epsilon}}}$                                       & $\bO{\max\pa{\kappa_f, \kappa_{pd}\kappa_C} \log \pa{\frac{1}{\epsilon}}}$                                 \\
      \midrule
      \rowcolor{bgcolor}
      MiD2A-iDAPG, $\rho>0$\tnote{2}               & \textcircled{1}\tnote{3}             & $\tilde{\mO}\pa{\sqrt{\kappa_{pd}\kappa_C} \log \pa{\frac{1}{\epsilon}}}$                                        & $\mathcal{A}$:\tnote{4} $\tilde{\mO}\pa{\sqrt{\kappa_f\kappa_{pd}} \log \pa{\frac{1}{\epsilon}}}$, $\mathcal{B}$:  $\tilde{\mO}\pa{\sqrt{\kappa_{pd}} \log \pa{\frac{1}{\epsilon}}}$ \\
      \midrule
      \multicolumn{4}{c}{\textbf{Case 2: \cref{G,convex,strong_duality,hs_smooth}, $g_i = 0$ and $A_i$ has full row rank, $i \in  \sI$}}                                                                                                                                                                                  \\
      \midrule
      (80) in \cite{nedic2018improved}                      & $A_i = \I$, \textcircled{2} & $\bO{\sqrt{\kappa_f\kappa_C}\log \pa{\frac{1}{\epsilon}}}$                                                       & $\bO{\sqrt{\kappa_f\kappa_C}\log \pa{\frac{1}{\epsilon}}}$                                                 \\
      \midrule
      DCDA \cite{alghunaim2019proximal}                     & \textcircled{2}             & $\mO\lt(\max\lt(\kappa_f\kappa_{\A}, \kappa_C\rt) \log \pa{\frac{1}{\epsilon}}\rt)$                              & $\mO\lt(\max\lt(\kappa_f\kappa_{\A}, \kappa_C\rt) \log \pa{\frac{1}{\epsilon}}\rt)$                        \\
      \midrule
      \rowcolor{bgcolor}
      MiD2A-iDAPG, $\rho>0$ & \textcircled{1}                  & $\tilde{\mO}\pa{\sqrt{\kappa_{f}\kappa_{\A}\kappa_C} \log \pa{\frac{1}{\epsilon}}}$ & $\mathcal{A}$: $\tilde{\mO}\pa{\kappa_{f}\sqrt{\kappa_{\A}} \log \pa{\frac{1}{\epsilon}}}$, $\mathcal{B}$: $\tilde{\mO}\pa{\sqrt{\kappa_{f}\kappa_{\A}} \log \pa{\frac{1}{\epsilon}}}$ \\      
      \midrule
      \multicolumn{4}{c}{\textbf{Case 3: \cref{G,convex,strong_duality,hs_smooth}, $g_i = 0$ and $A = [A_1, \cdots, A_n]$ has full row rank}}                                                                                                                                                                             \\
      \midrule
      DCPA \cite{alghunaim2021dual}                         & \textcircled{1}             & $\mO\lt(\kappa_f\kappa_{\A_{\rho}}\kappa_C \log \pa{\frac{1}{\epsilon}}\rt)$                                     & $\mO\lt(\kappa_f\kappa_{\A_{\rho}}\kappa_C \log \pa{\frac{1}{\epsilon}}\rt)$                               \\
      \midrule
      NPGA \cite{li2024npga}                                & \textcircled{1}             & $\mO\lt(\max\lt(\kappa_f\kappa_{\A_{\rho}}, \kappa_C\rt) \log \pa{\frac{1}{\epsilon}}\rt)$                       & $\mO\lt(\max\lt(\kappa_f\kappa_{\A_{\rho}}, \kappa_C\rt) \log \pa{\frac{1}{\epsilon}}\rt)$                 \\
      \midrule
      \rowcolor{bgcolor}
      MiD2A-iDAPG, $\rho>0$ & \textcircled{1}                 & $\tilde{\mO}\pa{\sqrt{\kappa_f\kappa_{\A'_{\rho}}\kappa_C} \log \pa{\frac{1}{\epsilon}}}$ & $\mathcal{A}$: $\tilde{\mO}\pa{\kappa_f\sqrt{\kappa_{\A'_{\rho}}} \log \pa{\frac{1}{\epsilon}}}$, $\mathcal{B}$: $\tilde{\mO}\pa{\sqrt{\kappa_f\kappa_{\A'_{\rho}}} \log \pa{\frac{1}{\epsilon}}}$ \\
      \bottomrule
    \end{tabularx}
    \begin{tablenotes}
      \item $\mathcal{A}$: $\nabla f$ and $\prox_g$; $\mathcal{B}$: $\A$, $\A\T$, and $\nabla \mh^*$/$\prox_{\mh^*}$ ($\prox_{\mh^*}$/$\nabla \mh^*$ denotes $\prox_{\mh^*}$ if $h^*$ is proximal-friendly; otherwise, it denotes $\nabla \mh^*$). The oracle complexity of $\A$ and $\A\T$ represents the number of matrix-vector multiplications involving $\A$ and $\A\T$. \textcircled{1}: $h^*$ is proximal-friendly. \textcircled{2}: $h = \iota_{\set{b}}$, i.e., $h^*$ is a linear function. $\kappa_C = \frac{\ove{C}}{\ue{C}}$, $\kappa_f = \frac{L_f}{\mu_f}$, $\kappa_{pd} = \frac{\os^2(\A)}{\mu_f\mu_{h^*}/n}$, $\kappa_{\A}  = \frac{\os^2(\A)}{\mins^2(\A)}$, $\kappa_{\A_{\rho}} = \frac{\os^2(\A)}{\mine{\A\A\T + \rho L_f\mC}}$, $\kappa_{\A'_{\rho}} = \frac{\os^2(\A)}{\mine{\A\A\T + \rho L_fP_K(\mC)}}$.
      \item [1] When only one oracle complexity is provided, it indicates that the oracle complexities of $\mathcal{A}$ and $\mathcal{B}$ are the same.
      \item [2] By selecting different subproblem solvers and varying the value of $\rho$ (a key parameter), we can obtain numerous versions of iD2A and MiD2A, each with potentially significant differences in complexity. Due to space constraints, we only present the complexities for a specific version of MiD2A (MiD2A-iDAPG ($\rho>0$), which employs iDAPG \cite{li2025inexact} as its subproblem solver with $\rho>0$) here. The complete table is available in \cref{appendix_full_table}, which includes the complexities for many other versions of iD2A and MiD2A. \textbf{It is important to note that there are other versions of iD2A and MiD2A that exhibit lower communication complexity or oracle complexity of $\mathcal{A}$ for Cases 1 and 2.}
      \item [3] Unlike existing algorithms, the convergence of iD2A and MiD2A does not depend on \textcircled{1} (which includes \textcircled{2} as a special case). For the complexities of iD2A and MiD2A without \textcircled{1}, refer to \cref{appendix_full_table}.
      \item [4] $\tilde{\mO}$ hides a logarithmic factor that depends on the problem parameters; refer to \cref{iD2A_complexity_convex} for further details.
    \end{tablenotes}
  \end{threeparttable}
\end{table*}

\subsection{Related Works}
The presence of coupled constraints renders \cref{main_pro} substantially more complex to solve than classical decentralized unconstrained optimization problems (DUOPs). Two popular approaches have emerged for designing decentralized algorithms to address \cref{main_pro} or its special cases:
\begin{enumerate}
  \item \textbf{Decentralizing existing centralized algorithms for solving \cref{main_pro}:} In the centralized setting, two classical algorithms immediately come to mind for solving \cref{main_pro}: the alternating direction method of multipliers (ADMM) \cite{boyd2011distributed} and the primal-dual gradient method (also referred to as the primal-dual hybrid gradient method in some literature) \cite{arrow1958studies,chambolle2011first}. However, these algorithms cannot be directly implemented in a decentralized manner. To decentralize these methods, additional structures are needed. For example, a dynamic average consensus algorithm \cite{zhu2010discrete} can be used to track the time-varying violation of the coupled equality constraint in a decentralized manner \cite{falsone2020tracking}. Alternatively, the implicit tracking approach \cite{li2022implicit} can be employed for the same purpose.
  \item \textbf{Applying existing decentralized algorithms for solving DUOP to address the dual of \cref{main_pro}:} A key insight regarding \cref{main_pro} is that its dual problem exhibits structural equivalence to DUOP. This property emerges from the particular form of the coupled constraint, which enables decomposition of the dual function into $n$ agent-specific local dual functions, each dependent solely on local parameters ($f_i$, $g_i$, and $A_i$), see \cref{dual_pro_eq} for the mathematical formulation. This structural similarity immediately suggests the possibility of leveraging existing DUOP algorithms to solve the dual problem of \cref{main_pro}, effectively yielding decentralized solutions for the original constrained problem. The main computational challenge in this approach lies in the computation of the dual gradient. One straightforward option is to use the exact dual gradient \cite{falsone2020tracking,zhang2020distributed}; however, this may be impractical due to the requirement for exact optimization at each iteration. An effective alternative is to employ the inexact dual gradient \cite{alghunaim2021dual,li2024npga,li2023inexact}, which avoids this issue.
\end{enumerate}

Building on these two main approaches, numerous decentralized algorithms have been proposed to solve \cref{main_pro}. These algorithms mainly aim to achieve various convergence rates under milder conditions, including asymptotic/sublinear convergence and linear convergence:
\begin{enumerate}
  \item \textbf{Asymptotic/Sublinear Convergence:} Unlike earlier algorithms, Tracking-ADMM \cite{falsone2020tracking}, NECPD \cite{su2021distributed}, and Proj-IDEA \cite{li2022implicit} achieve convergence without requiring $f_i$ or $g_i$ to be strongly convex, as all these methods are essentially augmented Lagrangian methods (ALMs) applied to the dual of \cref{main_pro}. Moreover, Proj-IDEA, unlike Tracking-ADMM and NECPD, does not require the local constrained sets $\sX_i$ to be bounded. In contrast, algorithms such as DCDA \cite{alghunaim2019proximal} and NPGA \cite{li2023proximal} rely on the strong convexity of $f_i$ to guarantee convergence but can handle more general forms of $g_i$.
  \item \textbf{Linear Convergence:} Early algorithms could only achieve linear convergence when $g_i = 0$ and $A_i = \I$ \cite{nedic2018improved}, or $A_i$ has full row rank \cite{chang2014multi,alghunaim2019proximal}. However, DCPA \cite{alghunaim2021dual}, IDEA \cite{li2022implicit}, and NPGA \cite{li2024npga} relax the full row rank condition to that of $A = [A_1, \cdots, A_n]$. For the special case where $h = \iota_{\set{b}}$, iD2GT \cite{li2023inexact} achieves linear convergence without any full row rank condition. Furthermore, when $h^*$ is strongly convex, DCPA and NPGA can linearly converge without imposing any other conditions.
\end{enumerate}

Despite these advances, none of the aforementioned algorithms are accelerated methods. Here, ``accelerated'' refers to the generalized mechanism developed by Nesterov \cite{nesterov2018lectures}, which improves the convergence rate of the classical gradient method for smooth unconstrained optimization \cite{lin2015universal}. For DUOP, accelerated methods have been shown to significantly outperform unaccelerated ones in terms of convergence rates \cite{scaman2017optimal,arjevani2020ideal}. This observation is expected to hold for \cref{main_pro}: accelerated methods converge much faster than unaccelerated ones.
To the best of our knowledge, only one accelerated algorithm has been developed for a special case of \cref{main_pro}: $g_i = 0$, $A_i = \I$, and $h = \iota_{\set{b}}$ \cite{nedic2018improved}. However, this problem is overly specific compared to \cref{main_pro}, severely constraining its practical utility. Moreover, the algorithm's design critically depends on the problem's simplified structure, rendering extensions to more general settings improbable. Another limitation involves the requirement for each agent to share its local gradient with neighbors at every iteration, which introduces potential privacy risks \cite{abadi2016deep}.
Given these limitations, our objective is to develop accelerated algorithms that can effectively solve more general cases of \cref{main_pro}.

\subsection{Our Contributions}
The contributions of this work are as follows:
\begin{enumerate}
  \item \textbf{A Dual$^2$ Approach to Achieve Acceleration and the Resulting Accelerated Algorithms for \cref{main_pro}:} Directly designing accelerated decentralized algorithms for \cref{main_pro} is highly challenging. Instead, we propose a novel dual$^2$ approach that enables acceleration by decomposing the original problem into two tractable components: a convex unconstrained optimization problem \cref{final_pro} and a saddle-point problem \cref{saddle_pro}. We prove that \cref{final_pro} is well-defined and equivalent to \cref{main_pro}: it is smooth everywhere with gradients computable via \cref{saddle_pro}, and its solution exists and implies a solution to \cref{main_pro}. Leveraging these properties, we incorporate Nesterov's acceleration technique \cite{nesterov2018lectures} into \cref{final_pro}, deriving two accelerated algorithms for \cref{main_pro}: iD2A and its multi-consensus variant, MiD2A.
  \item \textbf{Theoretical and Experimental Superiority over Existing Algorithms:} We first prove the asymptotic convergence of iD2A and MiD2A under the mild condition that $h$ is a closed proper convex function. This significantly generalizes existing algorithms that only handle special cases of $h$; see \cref{2024}. Furthermore, we prove linear convergence rates for three scenarios: (1) $h^*$ is strongly convex, (2) $A_i$ has full row rank, and (3) $A = [A_1, \cdots, A_n]$ has full row rank. We then analyze the communication and computational complexity bounds, which depend on the choice of saddle-point subproblem solver (e.g., iDAPG method proposed in \cite{li2025inexact}). Crucially, when an appropriate subproblem solver is selected, both iD2A and MiD2A achieve significantly lower communication and computational complexities than state-of-the-art (SOTA) algorithms in all scenarios; see \cref{tab:complexity_h}. Several numerical experiments further confirm these theoretical advantages.
\end{enumerate}

The paper is organized as follows: \cref{algo_develop} outlines the development of iD2A and MiD2A. The convergence rates for both algorithms are established in \cref{converge_rate}, while their complexities are analyzed in \cref{complexity}. \cref{experiments} presents the numerical experiments, and \cref{conclusion} draws the conclusions.

\section{Preliminaries}
\subsection{Notations}
Throughout this paper, unless specified otherwise, we use the standard inner product $\dotprod{\cdot, \cdot}$ and the standard Euclidean norm $\norm{\cdot}$ for vectors, and the standard spectral norm $\norm{\cdot}$ for matrices. For $x \in \R$, $\lfloor x \rfloor$ denotes the largest integer not greater than $x$. The set $\R^d_+$ represents $\set{x \in \R^d | x \geq \0}$. The vectors $\1_n$ and $\0_n$ represent the vectors of $n$ ones and zeros, respectively. For vectors $x$ and $y$, $\coll{x, y} = [x\T, y\T]\T$, and $\Span{x}$ denotes the span of $x$. The matrix $\I_n$ denotes the $n \times n$ identity matrix, and $\0_{m \times n}$ denotes the $m \times n$ zero matrix.
For a matrix $B \in \R^{m \times n}$, let $\mins(B)$, $\us(B)$, and $\os(B)$ represent the smallest, smallest nonzero, and largest singular values of $B$, respectively. Additionally, let $\Range{B}$ and $\Null{B}$ denote its column space and null space. For a symmetric matrix $A \in \R^{n \times n}$, $\mine{A}$, $\ue{A}$, and $\ove{A}$ denote the smallest, smallest nonzero, and largest eigenvalues of $A$, respectively. We use $A > 0$ (or $A \geq 0$) indicates that $A$ is positive definite (or positive semi-definite).
For a function $f: \R^n \rightarrow \exs$, let $\dom{f} = \set{x \in \R^n | f(x) < +\infty}$ denote its (effective) domain. $S_f(x)$ denotes one of its subgradients at $x$, and $\partial f(x)$ denotes its subdifferential at $x$, which includes all subgradients at that point. The proximal operator of $f$ is defined as $\prox_{\alpha f}(x) = \arg\min_{y} f(y) + \frac{1}{2\alpha}\norm{y-x}^2$ for stepsize $\alpha$, and the Fenchel conjugate of $f$ is defined as $f^*(y) = \sup_{x \in \R^n}y\T x-f(x)$. The smoothness parameter, strong convexity parameter, and condition number are denoted by $\mu_f$, $L_f$, and $\kappa_f=\frac{L_f}{\mu_f}$, respectively. Additionally, the Moreau envelope (or Moreau-Yosida regularization) of $f$ is given by $M_{\gamma f}(x) = \inf_{y}\set{f(y) + \frac{1}{2\gamma}\norm{x-y}^2}$, where $\gamma>0$. For a set $\sX \subseteq \R^n$, $\ri{\sX}$ denotes its relative interior, and $\dist{y, \sX} = \min_{x \in \sX}\norm{y-x}$ denotes the distance from a point $y \in \R^n$ to the set. The indicator function of $\sX$ is defined as $\iota_{\sX}(x) = \lt\{
  \begin{aligned}
    0,      \   & x \in \sX,        \\
    +\infty, \  & \text{otherwise}.
  \end{aligned}\right.$

\subsection{Convex Analysis} \label{convex_analysis}
At the beginning, we should point out that most definitions appeared in this subsection are borrowed from \cite{hiriart2004fundamentals}. An extended-valued function $f: \R^n \rightarrow \exs$, which may take the value $+\infty$ and includes the finite-valued function $g: \R^n \rightarrow \R$ as a special case, is said to be convex on a set $\sX \subset \R^n$ if $\sX$ is convex, and
\eqe{ \label{1358}
  f(\theta x + (1-\theta x')) \leq \theta f(x) + (1-\theta)f(x'), \ \forall \theta \in [0, 1]
}
holds for any $x, x' \in \sX$, where \cref{1358} is considered as an inequality in $\exs$. Particularly, $f$ is said to be convex if \cref{1358} holds for any $x, x' \in \R^n$, and $f$ is said to be $\mu$-strongly convex (on $\sX$) if $f-\frac{\mu}{2}\norm{\cdot}^2$ is convex (on $\sX$). An obvious fact is that, $\dom{f}$ is convex and $f$ is convex (strongly convex) on $\dom{f}$ is sufficient and necessary for $f$ to be convex (strongly convex).
We say $f$ is closed if it is lower semicontinuous on $\R^n$ or if its epigraph is closed. Besides, $f$ is said to be $L$-smooth if it is differentiable on $\R^n$ and $\nabla f$ is $L$-Lipschitz continuous on $\R^n$, i.e.,
\eqe{
  \norm{\nabla f(x)-\nabla f(x')} \leq L\norm{x-x'}, \ \forall x, x' \in \R^n.
}

The following two lemmas reveal the duality between the strong convexity of a function $f$ and the smoothness of its Fenchel conjugate $f^*$.
\begin{lemma} \cite[Lecture 5]{vandenberghe2022om} \label{conjugate_smooth}
  Assume that $f: \R^n \rightarrow \exs$ is closed proper\footnote{We say a convex function $f$ is proper if $f(x)>-\infty$ for all $x$ and $\dom{f} \neq \emptyset$.} and $\mu$-strongly convex with $\mu>0$, then (1) $\dom{f^*} = \R^n$; (2) $f^*$ is differentiable on $\R^n$ with $\nabla f^*(y) = \arg\max_{x \in \dom{f}} y\T x - f(x)$; (3) $f^*$ is convex and $\frac{1}{\mu}$-smooth.
\end{lemma}

\begin{lemma} \cite[Theorem E.4.2.2]{hiriart2004fundamentals} \label{conjugate_strong_convex}
  Assume that $f: \R^n \rightarrow \R$ is convex and $L$-smooth with $L>0$, then $f^*$ is $\frac{1}{L}$-strongly convex on every convex set $\mathcal{Y} \subseteq \dom{\partial f^*}$, where $\dom{\partial f^*} = \set{y \in \R^n | \partial f^*(y) \neq \emptyset}$.
\end{lemma}

The following lemma demonstrates the relationship between the subdifferentials of a function and its conjugate.
\begin{lemma} \label{conjugate_subdiff} \cite[Corollary E.1.4.4]{hiriart2004fundamentals}
  Assume that $f: \R^n \rightarrow \exs$ is closed proper convex, we have
  \eqe{
    y \in \partial f(x) \iff x \in \partial f^*(y), \ \forall x, y \in \R^n.
  }
\end{lemma}

The following lemma pertains to subdifferential calculus, which can be easily derived from \cite[Theorems 6.2 and 23.8]{rockafellar1970convex}.
\begin{lemma}*\footnote{Throughout this paper, ``*'' following a lemma indicates that its proof is omitted due to its straightforward nature.} \label{subdiff}
  Assume that $f_1: \R^n \rightarrow \exs$ is proper convex and $f_2: \R^n \rightarrow \R$ is convex, then $\partial (f_1 + f_2) = \partial f_1 + \partial f_2$.
\end{lemma}

\section{Algorithm Development} \label{algo_develop}
\subsection{Basic Assumptions}
Let $\mx = \col{x} \in \R^d$, where $\sum_{i=1}^nd_i = d$. Define
\eqe{
&f(\mx) = \sum_{i=1}^n f_i(x_i), \ g(\mx) = \sum_{i=1}^n g_i(x_i), \\
&A = [A_1, \cdots, A_n] \in \R^{p \times d}, \ \A = \diag{A_1, \cdots, A_n} \in \R^{np \times d}.
\nonumber
}
Using the above notation, \cref{main_pro} can be rewritten in the following compact form:
\eqe{
  \min_{\mx \in \R^d} f(\mx) + g(\mx) + h(A\mx).
}

Throughout this paper, we assume that \cref{strong_duality,G,convex} hold.
\begin{assumption} \label{strong_duality}
  Slater's condition holds for \cref{main_pro}, i.e., $A\ri{\dom{g}} \cap \ri{\dom{h}} \neq \emptyset$.
\end{assumption}

\begin{assumption} \label{convex}
  For $i \in \sI$, $f_i$, $g_i$, and $h$ satisfy
  \begin{enumerate}
    \item $f_i$ is $\mu_i$-strongly convex\footnote{For properties of a function, such as convexity or continuity, if we do not specify the region where the property holds, the readers should always assume that it is valid on the function's entire domain.} and $L_i$-smooth for $L_i \geq \mu_i > 0$;
    \item $h$ and $g_i$ are proper convex and lower semicontinuous;
    \item $g_i$ is proximal-friendly \footnote{We say a function $f$ is proximal-friendly if $\prox_{\alpha f}(x)$ can be easily computed for any $x$. For example, this is the case when $\prox_{\alpha f}(x)$ can be calculated by a closed-form expression of $x$.}.
  \end{enumerate}
\end{assumption}

\begin{assumption} \label{G}
  The gossip matrix $C = [c_{ij}] \in \R^{n \times n}$ satisfies
  \begin{enumerate}
    \item $c_{ij}>0$ if $i = j$ or $(i,j) \in \mathcal{E}$, otherwise $c_{ij} = 0$, where $\mathcal{E}$;
    \item $C$ is symmetric;
    \item $C$ is positive semidefinite;
    \item $\Null{C} = \Span{\1_n}$.
  \end{enumerate}
\end{assumption}

Under \cref{convex}, $f_i + g_i$ and $h$ are closed proper convex, while $f$ is $\mu_f$-strongly convex and $L_f$-smooth, where $\mu_f = \min_{i \in \sI}\mu_i$ and $L_f = \max_{i \in \sI}L_i$. Additionally,  \cref{convex,strong_duality} imply two important facts: (1) strong duality holds for \cref{main_pro}, and (2) the solution to the dual of \cref{main_pro} exists \cite[Corollary 31.2.1]{rockafellar1970convex}.

\begin{remark}
  Since the network topology $\mathcal{G}$ is undirected and connected, a gossip matrix satisfying \cref{G} can be readily constructed. A straightforward choice for the gossip matrix is the Laplacian or weighted Laplacian matrix of $\mathcal{G}$. Alternatively, we can define $C = \I - W$, where $W \in \R^{n \times n}$ is a doubly stochastic matrix associated with $\mathcal{G}$. For detailed methods of constructing $W$, refer to \cite{shi2015extra}.
\end{remark}

\subsection{A Dual$^2$ Approach to Achieve Acceleration} \label{2241}
\begin{algorithm}[tb]
  \caption{\small{Inexact Dual$^2$ Accelerated Gradient Method (iD2A)}}
  \label{alg:iD2A}
  \begin{algorithmic}[1]
    \small
    \Require
    $K>0$, $\rho \geq 0$, $C$, $L_{\fr}$, $\mu_{\fr}$
    \Ensure
    $\mx^K$
    \State $\mx^0 = \ml^0 = \0$, $\mz^0 = \mw^0 = \0$
    \State Set $\beta_k = \frac{\sqrt{\kappa_{\fr}}-1}{\sqrt{\kappa_{\fr}}+1}$ if $\mu_{\fr} > 0$, where $\kappa_{\fr} = \frac{L_{\fr}}{\mu_{\fr}}$; otherwise set $\beta_k = \frac{k}{k+3}$.
    \For {$k=0,\dots, K-1$}
    \State Solve
    \begin{equation} \label{inner_pro} \tag{SP}
      \begin{aligned}
      \min_{\mx \in \R^d}\max_{\ml \in \R^{np}} f(\mx) &+ g(\mx) + \ml\T\A\mx \\
      &- \pa{\mh^*(\ml) + \frac{\rho}{2}\ml\T\mC\ml + \ml\T\mz^k}
      \end{aligned}
    \end{equation}
    to obtain an inexact solution $\pa{\ml\+, \mx\+}$.
    \State $\mw\+ = \mz^k + \frac{1}{L_{\fr}}\mC\ml\+$
    \State $\mz\+ = \mw\+ + \beta_k\pa{\mw\+-\mw^k}$
    \EndFor
  \end{algorithmic}
\end{algorithm}

\begin{algorithm}[tb]
  \caption{\small{Multi-Consensus Inexact Dual$^2$ Accelerated Gradient Method (MiD2A)}}
  \label{alg:MiD2A}
  \begin{algorithmic}[1]
    \small
    \Require
    $T>0$, $K = \lfloor \sqrt{\kappa_C} \rfloor$, $\rho \geq 0$, $C$, $L_{\fr}$, $\mu_{\fr}$
    \Ensure
    $\mx^T$
    \State $\mx^0 = \ml^0 = \0$, $\mz^0 = \mw^0 = 0$
    \State Set $\beta_k = \frac{\sqrt{\kappa_{\fr}}-1}{\sqrt{\kappa_{\fr}}+1}$ if $\mu_{\fr} > 0$, where $\kappa_{\fr} = \frac{L_{\fr}}{\mu_{\fr}}$; otherwise set $\beta_k = \frac{k}{k+3}$.
    \For {$k=0,\dots, T-1$ }
    \State Solve
    \eqe{ \label{inner_pro_m}
      &\min_{\mx \in \R^d}\max_{\ml \in \R^{np}} f(\mx) + g(\mx) + \ml\T\A\mx - \\
      &\pa{\mh^*(\ml) + \ml\T\mz^k + \frac{\rho}{2}\dotprod{\ml, \text{AcceleratedGossip}\protect\footnotemark(\ml, \mC, K)}}
      \nonumber
    }
    to obtain an inexact solution $\pa{\ml\+, \mx\+}$.
    \State $\mw\+ = \mz^k + \frac{1}{L_{\fr}}\text{AcceleratedGossip}(\ml\+, \mC, K)$
    \State $\mz\+ = \mw\+ + \beta_k\pa{\mw\+-\mw^k}$
    \EndFor
  \end{algorithmic}
\end{algorithm}
\footnotetext{AcceleratedGossip is an accelerated consensus algorithm proposed in \cite{scaman2017optimal}, which can be found in \cref{appendix_AcceleratedGossip}.}

The following lemma introduces three optimization problems related to \cref{main_pro}.
  \begin{lemma} \label{1552}
    Assume that \cref{convex,strong_duality,G} hold. Then,
    \begin{enumerate}
      \item The Lagrangian dual problem of \cref{main_pro} is formulated as
            \eqe{ \label{dual_pro_eq}
              \min_{\lambda \in \R^p} \sum_{i=1}^n \phi_i(\lambda) + h^*(\lambda),
            }
            where $\phi_i(\lambda) = -\inf_{x_i}\pa{f_i(x_i) + g_i(x_i) + \lambda\T A_ix_i}$, $i \in \sI$.
      \item Problem \cref{dual_pro_eq} is equivalent to
            \eqe{ \label{dual_pro_consensus}
              \min_{\ml \in \R^{np}} &\mph(\ml) + \mh^*(\ml) \\
              \text{s.t.} \ &\sqrt{\mC}\ml = \0,
            }
            where $\ml = \col{\lambda}$, $\sqrt{\mC} = \sqrt{C} \otimes \I_p$ ($\sqrt{C}$ satisfies $C = \sqrt{C}\sqrt{C}$), $\mph(\ml) = \sum_{i=1}^n \phi_i(\lambda_i)$, and $\mh^*(\ml) = \frac{1}{n} \sum_{i=1}^n h^*(\lambda_i)$.
      \item The augmented Lagrangian dual problem of \cref{dual_pro_consensus} is given by
            \eqe{ \label{final_pro}
              \min_{\my \in \R^{np}} \fr(\my),
            }
            where $\fr(\my) = H^*\pa{-\sqrt{\mC}\my}$, and $H(\ml) = \mph(\ml) + \mh^*(\ml) + \frac{\rho}{2}\ml\T\mC\ml$.
    \end{enumerate}
    Additionally, the solutions to \cref{dual_pro_eq}, \cref{dual_pro_consensus}, and \cref{final_pro} exist.
  \end{lemma}

\begin{proof}
  See \cref{appendix_1552}.
\end{proof}

According to \cref{1552}, we can observe that \cref{final_pro} is essentially an equivalent problem of \cref{main_pro}.
A natural question arising from \cref{1552} is whether the solution to \cref{main_pro} can be obtained based on that of \cref{final_pro}. Before addressing this question, we first investigate the properties of $\fr$, such as convexity and smoothness.
Based on the proof of \cref{1552} and \cite[Proposition B.2.1.4]{hiriart2004fundamentals}, it is straightforward to verify that $\fr$ is closed proper convex under \cref{G,convex}. The following two lemmas build upon the smoothness of the augmented Lagrangian dual for general linearly constrained optimization problems we established in \cite{li2025smooth}. They show that $\fr$ is $\frac{1}{\rho}$-smooth when $\rho>0$ and demonstrate how to calculate the gradient of $\fr$.
\begin{lemma} \label{2355}
  Assume that \cref{G,convex} hold, and $\rho>0$. Then: (1) $\dom{\fr} = \R^{np}$; (2) $\fr$ is convex and $\frac{1}{\rho}$-smooth.
\end{lemma}

\begin{proof}
  See \cref{appendix_2355}.
\end{proof}

\begin{lemma} \label{23551}
  Assume that \cref{G,convex} hold, and $\fr$ is differentiable everywhere. Define
  \eqe{
    \cT(\mz, \mx, \ml) =& f(\mx) + g(\mx) + \ml\T\A\mx \\
    &- \pa{\mh^*(\ml) + \frac{\rho}{2}\ml\T\mC\ml + \ml\T\mz}.
  }
  Then, for any $\my \in \R^{np}$, there exists at least one solution $(\mx^+, \ml^+) \in \R^d \times \R^{np}$ of the following saddle-point problem:
  \eqe{ \label{saddle_pro}
    \min_{\mx \in \R^d}\max_{\ml \in \R^{np}} \cT(\mz, \mx, \ml),
  }
  where $\mz = \sqrt{\mC}\my$. Furthermore, $\nabla \fr(\my) = -\sqrt{\mC}\ml^+$ holds for any solution $(\mx^+, \ml^+)$.
\end{lemma}

\begin{proof}
  See \cref{appendix_23551}.
\end{proof}

Using \cref{23551}, we can explore the relationship between the solutions of \cref{main_pro} and \cref{final_pro}.
\begin{lemma} \label{relation_solutions}
  Assume that \cref{G,strong_duality,convex} holds, $\fr$ is differentiable everywhere, and $\my^*$ is a solution of \cref{final_pro}. Then, there exists a solution $(\mx^*, \ml^*) \in \R^d \times \R^{np}$ of \cref{saddle_pro} with $\mz = \sqrt{\mC}\my^*$, where $\mx^*$ is the solution of \cref{main_pro}.
\end{lemma}
\begin{proof}
  See \cref{appendix_relation_solutions}.
\end{proof}

According to \cref{2355,relation_solutions}, two key facts about \cref{final_pro} are: (1) it is a smooth convex optimization problem when $\rho > 0$, and (2) solving \cref{final_pro} allows us to readily obtain the solution to \cref{main_pro} by solving the saddle-point problem \cref{saddle_pro}. Therefore, a natural approach to solving \cref{main_pro} is to address the unconstrained smooth convex optimization problem \cref{final_pro}. To solve the latter, Nesterov's accelerated gradient descent (AGD) \cite{nesterov2018lectures} can be employed to achieve the optimal convergence rate (provided that only first-order information is available), potentially leading to an accelerated algorithm for the original problem \cref{main_pro}.
Since \cref{final_pro} is the dual problem of an equivalent form of the dual problem of the original problem \cref{main_pro}, we refer to this approach as the dual$^2$ approach.

\subsection{Inexact Dual$^2$ Accelerated Gradient Method} \label{1442}
Building on the dual$^2$ approach, we now design accelerated algorithms to solve \cref{main_pro}.
As previously mentioned, our approach involves applying AGD to solve \cref{final_pro}, which probably lead to an accelerated algorithm for the original problem \cref{main_pro}.
In fact, AGD has many variants, one of which takes the following form when applied to \cref{final_pro}:
\eqe{ \label{2044}
  \my\+ &= \mv^k - \frac{1}{L_{\fr}}\nabla \fr(\mv^k), \\
  \mv\+ &= \my\+ + \beta_k\pa{\my\+-\my^k},
}
where $\beta_k$ is the variable or constant stepsize designed to maximize the convergence rate. For the convex case, various choices of $\beta_k$ exists; here, we adopt the one proposed in \cite{tseng2008accelerated}: $\beta_k = \frac{k}{k+3}, k \geq 0$ \footnote{This stepsize scheme corresponds to Equations (19) and (27) in \cite{tseng2008accelerated}, which is also used in \cite{schmidt2011convergence}. Note that the original stepsize scheme in \cite{schmidt2011convergence} is $\beta_k = \frac{k-1}{k+2}, k \geq 1$, which is equivalent to ours.}. When $\fr$ is strongly convex, a classical choice is $\beta_k = \frac{\sqrt{\kappa_{\fr}}-1}{\sqrt{\kappa_{\fr}}+1}$, as described in \cite[\S2.2.1, Constant Step scheme III]{nesterov2018lectures}.
By \cref{23551}, we can write \cref{2044} as
\eqe{ \label{2018}
  \mx^*(\mv^k), \ml^*(\mv^k) &= \arg\min_{\mx \in \R^d}\max_{\ml \in \R^{np}} \cT\pa{\sqrt{\mC}\mv^k, \mx, \ml}, \\
  \my\+ &= \mv^k + \frac{1}{L_{\fr}}\sqrt{\mC}\ml^*(\mv^k), \\
  \mv\+ &= \my\+ + \beta_k\pa{\my\+-\my^k}.
}
However, implementing \cref{2018} in practice is highly challenging, as it requires computing the exact solution to the saddle-point problem \cref{saddle_pro} at each iteration, which can be prohibitively expensive or even infeasible. To overcome this limitation, we replace the exact solution with an inexact solution to \cref{saddle_pro}, resulting in the following algorithm:
\eqe{ \label{analyze_form}
  \mx\+, \ml\+ &\approx \arg\min_{\mx \in \R^d}\max_{\ml \in \R^{np}} \cT\pa{\sqrt{\mC}\mv^k, \mx, \ml}, \\
  \my\+ &= \mv^k + \frac{1}{L_{\fr}}\sqrt{\mC}\ml\+, \\
  \mv\+ &= \my\+ + \beta_k\pa{\my\+-\my^k}.
}
However, the presence of $\sqrt{\mC}$ prevents the decentralized implementation of \cref{analyze_form}. Fortunately, this issue can be resolved by introducing two new variables: $\mw^k = \sqrt{\mC}\my^k$ and $\mz^k = \sqrt{\mC}\mv^k$. This leads to the final formulation of iD2A:
\eqe{ \label{iD2A}
  \mx\+, \ml\+ &\approx \arg\min_{\mx \in \R^d}\max_{\ml \in \R^{np}} \cT\pa{\mz^k, \mx, \ml}, \\
  \mw\+ &= \mz^k + \frac{1}{L_{\fr}}\mC\ml\+, \\
  \mz\+ &= \mw\+ + \beta_k\pa{\mw\+-\mw^k}.
}
We can readily verify that \cref{iD2A} is fully decentralized, owing to the decomposability of $f$, $g$, $\A$ and $\mh^*$, as well as the decentralized structure of $\mC$. The formal version of iD2A is presented in \cref{alg:iD2A}, and its decentralized implementation is detailed in \cref{appendix_decen_iD2A}.

\begin{remark}
  An important feature of the subproblem in iD2A, $\min_{\mx \in \R^d}\max_{\ml \in \R^{np}} \cT\pa{\mz^k, \mx, \ml}$, is that its first-order information, including $\nabla f$, $\prox_g$, $\A$, $\nabla \mh^*$/$\prox_{\mh^*}$, $\mC$, and $\mz^k$, is fully decentralized. This implies that solving $\min_{\mx \in \R^d}\max_{\ml \in \R^{np}} \cT\pa{\mz^k, \mx, \ml}$ in a decentralized manner can by achieved by applying any centralized first-order algorithm\footnote{Here, first-order algorithms refer to those that utilize only first-order information of $\cT\pa{\mz^k, \mx, \ml}$, such as those listed in \cref{appendix_oracle_complexity}.} designed for the general saddle-point problem \cref{general_saddle_pro}. The resulting subproblem-solving procedure is inherently decentralized.
\end{remark}

\subsection{Multi-Consensus Inexact Dual$^2$ Accelerated Gradient Method}
Assume that $\fr$ is $\mu_{\fr}$-strongly convex with $\mu_{\fr} > 0$. We will prove in \cref{converge_rate} that the outer iteration complexity of iD2A depends on $\kappa_{\fr}$, which in turn relies on $\kappa_C = \frac{\ove{C}}{\ue{C}}$. Consequently, reducing $\kappa_C$ can decrease the outer iteration complexity of iD2A.
One potential approach to achieve this is to construct a new gossip matrix with a smaller condition number than $C$. This can be accomplished using the Chebyshev acceleration technique \cite{auzinger2017iterative,arioli2014chebyshev}.

Consider the Chebyshev polynomials $T_k$:
\eqe{
&T_0(x) = 1, \ T_1(x) = x, \\
&T_{k+1}(x) = 2xT_k(x) - T_{k-1}(x), \ k \geq 1,
}
then define the following polynomial:
\eqe{
  P_K(x) = 1 - \frac{T_K(c_2(1-c_3x))}{T_K(c_2)},
}
where $K \geq 0$, $c_2 = \frac{\kappa_C+1}{\kappa_C-1}$, and $c_3 = \frac{2}{(1+1/\kappa_C)\ove{C}}$. Let $K = \lfloor \sqrt{\kappa_C} \rfloor$. Then, the matrix $P_K(C)$ satisfies $\kappa_{P_K(C)} = \frac{\ove{P_K(C)}}{\ue{P_K(C)}} \leq 4$, according to the proof of Theorem 4 in \cite{scaman2017optimal}.
Based on the definition of $P_K$ and \cref{G}, we can easily verify that $P_K(C)$ is symmetric, positive semidefinite, and shares the same null space as $C$. However, $P_K(C)$ lacks a decentralized nature and thus does not meet the first requirement of \cref{G}. Despite this, matrix-vector multiplication involving $P_K(C)$ can still be implemented in a decentralized manner using AcceleratedGossip proposed in \cite{scaman2017optimal}; see \cref{appendix_AcceleratedGossip}. Consequently, we obtain a new gossip matrix with a condition number of $\bO{1}$. Define
\eqe{
  \cT_{P_K(\mC)}(\mz, \mx, &\ml) = f(\mx) + g(\mx) + \ml\T\A\mx \\
  &- \pa{\mh^*(\ml) + \frac{\rho}{2}\ml\T P_K(\mC)\ml + \ml\T\mz}.
}
By the definition of $P_K$ and the properties of Kronecker product, we can easily verify that $P_K(C) \otimes \I_p = P_K(\mC)$.
Based on this, we derive a variant of iD2A, referred to as MiD2A:
\eqe{ \label{MiD2A}
  \mx\+, \ml\+ &\approx \arg\min_{\mx \in \R^d}\max_{\ml \in \R^{np}} \cT_{P_K(C)}(\mz^k, \mx, \ml), \\
  \mw\+ &= \mz^k + \frac{1}{L_{\fr}}P_K(\mC)\ml\+, \\
  \mz\+ &= \mw\+ + \beta_k\pa{\mw\+-\mw^k}.
}
Since $\kappa_{P_K(C)} \ll \kappa_{C}$, the outer iteration complexity of MiD2A is significantly lower than that of iD2A,making MiD2A a preferable choice in certain scenarios, which we will discuss in detail later.
However, a potential issue of MiD2A is that $P_K(\mC)\ml$ can only be implemented in a decentralized manner using AcceleratedGossip, which requires $\lfloor \sqrt{\kappa_C} \rfloor$ rounds of communication. The formal version of MiD2A is presented in \cref{alg:MiD2A}, and its decentralized implementation is detailed in \cref{appendix_decen_iD2A}.

\section{Convergence Analysis: Convergence Rates} \label{converge_rate}
In this section, we analyze the convergence of iD2A and MiD2A under two cases: (1) the general convex case ($\fr$ is convex), and (2) the strongly convex case ($\fr$ is strongly convex).

Recall that MiD2A is essentially a specific instance of iD2A, where the gossip matrix $C$ is replaced with $P_K(C)$. Notably, $P_K(C)$ retains the last three properties of $C$ as stated in \cref{G}. In fact, only these last three properties of $C$ in \cref{G} are utilized in the convergence analysis of iD2A. This implies that the convergence results of iD2A can be directly applied to MiD2A by substituting $C$ with $P_K(C)$. For simplicity, we will primarily focus on the analysis of iD2A. Readers should assume that the results obtained also apply to MiD2A unless stated otherwise.

\subsection{Convergence Under the General Convex Case}
For the general convex case, we analyze the convergence of iD2A and MiD2A under the following condition:
\begin{condition} \label{0129}
  $\fr$ is convex and $L_{\fr}$-smooth with $L_{\fr} > 0$.
\end{condition}

According to \cref{2355}, \cref{0129} trivially holds when \cref{convex,G} are satisfied and $\rho > 0$. Nevertheless, there are additional conditions that can also guarantee \cref{0129} while allowing for $\rho = 0$, which we will discuss later.

We first establish the error conditions that $\mx\+$ and $\ml\+$ should satisfy. Let $\mathbf{S}^*_{k+1} \subseteq \R^d \times \R^{np}$ be the solution set of \cref{inner_pro}, and define define $\mX^*_{k+1} = \lt\{\mx \in \R^d | (\mx, \ml) \in \mathbf{S}^*_{k+1}\rt\}$ and $\bm{\Lambda}^*_{k+1} = \lt\{\ml \in \R^{np} | (\mx, \ml) \in \mathbf{S}^*_{k+1}\rt\}$. Based on the initialization and the update rules of iD2A, we can easily verify that $\mz^k \in \Range{\mC}$ for $k \geq 0$. Since $\Range{\mC} = \Range{\sqrt{\mC}}$, it follows that $\mz^k \in \Range{\sqrt{\mC}}$ for $k \geq 0$. According to \cref{23551}, $\mathbf{S}^*_{k+1}$ must be nonempty for $k \geq 0$.
Recall that $f$ is strongly convex, then $\mX^*_{k+1}$ must be a singleton, denoted as $\mx^*_{k+1}$. However, $\bm{\Lambda}^*_{k+1}$ may contain multiple points. For $\mx\+$, we can directly apply the naive error condition:
\eqe{ \label{x_error_condition}
  \norm{\mx\+ - \mx^*_{k+1}}^2 \leq e_{\mx,k+1}.
}
Different from $\mx\+$, there exists a more natural error condition for $\ml\+$. According to \cref{0129}, we know that (1) $\sqrt{\mC}\ml$ is unique, and (2) $\nabla \fr(\my) = \sqrt{\mC}\ml^*_{k+1}$ for any $\ml^*_{k+1} \in \bm{\Lambda}^*_{k+1}$, even though $\bm{\Lambda}^*_{k+1}$ may contain multiple elements. Therefore, we adopt the following error condition for $\ml\+$:
\eqe{ \label{1551}
  \norm{\sqrt{\mC}\ml\+ - \sqrt{\mC}\ml^*_{k+1}} \leq \varepsilon_{k+1},
}
where $\ml^*_{k+1}$ is an arbitrary point in $\bm{\Lambda}^*_{k+1}$. According to \cref{23551}, the left-hand term in \cref{1551} naturally represents the error between the exact gradient of $\fr$ at $\my^k$ and the inexact gradient utilized in iD2A.

Instead of directly analyzing iD2A, we focus on its equivalent form---\cref{analyze_form}, which is more amenable to analysis. Due to the equivalence between iD2A and \cref{analyze_form}, the convergence of the latter ensures that of iD2A. Let $\my^*$ be any solution of \cref{final_pro}, and let $(\mx^*, \ml^*)$ be any solution of \cref{saddle_pro} with $\mz = \sqrt{\mC}\my^*$. According to \cref{relation_solutions}, $\mx^*$ is the solution of \cref{main_pro}. The convergence of iD2A is given in the following theorem.
\begin{theorem} \label{1717}
  Assume that \cref{strong_duality,G,convex,0129} hold, $\beta_k = \frac{k}{k+3}$, $\mx\+$ and $\ml\+$ meet the conditions \cref{x_error_condition} and \cref{1551} respectively, and
  \eqe{ \label{1644}
    \varepsilon_k = \bO{\frac{1}{k^{2+\delta}}}, \ \lim_{k \rightarrow \infty} e_{\mx,k} = 0,
  }
  where $\delta > 0$. Then, $\my^k$ and $\mx\+$ generated by iD2A satisfy that $\fr(\my^k) - \fr(\my^*) = \bO{\frac{1}{k^2}}$ and $\lim_{k \rightarrow \infty} \norm{\mx\+-\mx^*} = 0$.
\end{theorem}

\begin{proof}
  See \cref{appendix_1717}.
\end{proof}
Note that \cref{x_error_condition} is not strictly required theoretically before the final iteration; however, we recommend incorporating it into every iteration for improved implementation performance.

As mentioned before, $\rho > 0$ is sufficient to guarantee \cref{0129} when \cref{convex,G} hold. However, there are other sufficient conditions, as demonstrated in \cref{F_rho_smooth}. This implies that we can also set $\rho=0$ if the other conditions in \cref{F_rho_smooth} are satisfied, we will discuss the advantages of setting $\rho = 0$ later.
\begin{lemma} \label{F_rho_smooth}
  Assume that \cref{G,convex} hold, \cref{0129} holds if at least one of the following conditions is met:
  \begin{enumerate}
    \item $\rho > 0$; in this case $L_{\fr} = \frac{1}{\rho}$.
    \item $h^*$ is $\mu_{h^*}$-strongly convex with $\mu_{h^*} > 0$; in this case $L_{\fr} = \frac{1}{\max\pa{\rho, \frac{\mu_{h^*}}{n\ove{C}}}}$.
    \item $g_i = 0$, and $A_i$ has full row rank, $i \in \sI$; in this case $L_{\fr} = \frac{1}{\max\pa{\rho, \frac{1}{\ove{C}}\min_{i \in \sI}\frac{\mins^2(A_i)}{L_i}}}$.
  \end{enumerate}
\end{lemma}

\begin{proof}
  See \cref{appendix_F_rho_smooth}.
\end{proof}

\begin{remark}
  We now examine the conditions under which $h^*$ is strongly convex. We have shown in the proof of \cref{1552} that $h^*$ is closed proper convex under \cref{convex}, hence $\dom{h^*}$ is convex. Recall the definition of the strong convexity for extended-valued functions in \cref{convex_analysis}, then $h^*$ is $\mu_{h^*}$-strongly convex if $h^*$ is $\mu_{h^*}$-strongly convex on $\dom{h^*}$. According to \cref{conjugate_strong_convex}, we can obtain a sufficient condition for $h^*$ to be strongly convex on $\dom{h^*}$: $h$ is $L$-smooth with $L > 0$, and $\dom{\partial h^*} = \dom{h^*}$, where the latter means that $\partial h^*(\lambda) \neq \emptyset$ for any $\lambda \in \dom{h^*}$. This condition is naturally satisfied in many practical problems where $h$ is smooth, such as linear regression, Huber regression, logistic regression, and Poisson regression. For further discussions on these regression problems, see \cref{appendix_applications}.3.
\end{remark}

\subsection{Convergence Rates Under the Strongly Convex Case}
For the strongly convex case, we analyze the convergence of iD2A and MiD2A under the following condition:
\begin{condition} \label{F_rho_convex_smooth}
  $\fr$ is $\mu_{\fr}$-strongly convex on $\Range{\sqrt{\mC}}$ and $L_{\fr}$-smooth for $L_{\fr} \geq \mu_{\fr} > 0$.
\end{condition}

In this subsection, we introduce two additional assumptions:
\begin{assumption} \label{hs_smooth}
  $h^*$ is $L_{h^*}$-smooth for $L_{h^*} \geq 0$.
\end{assumption}

\begin{assumption} \label{1643}
  At least one of the following conditions holds:
  \begin{enumerate}
    \item $h^*$ is $\mu_{h^*}$-strongly convex with $\mu_{h^*} > 0$;
    \item $g_i = 0$ and $A_i$ has full row rank, $i \in \sI$;
    \item $g_i = 0$ and $A = [A_1, \cdots, A_n]$ has full row rank, and $\rho > 0$.
  \end{enumerate}
\end{assumption}

\begin{remark}
  Note that \cref{convex} guarantees that $h$ is closed proper convex, then there are two typical cases that satisfy \cref{hs_smooth}: (2) $h$ is $\mu_h$-strongly convex for $\mu_h > 0$; in this case $L_{h^*} = \frac{1}{\mu_h} > 0$. (2) $h$ is the indicator function of a singleton set; in this case $h^*$ is a linear function and $L_{h^*} = 0$.
\end{remark}

The following lemma shows that \cref{F_rho_convex_smooth} holds under \cref{G,convex,hs_smooth,1643}.
\begin{lemma} \label{F_rho_convex}
  Assume that \cref{G,convex,hs_smooth,1643} hold. Then, $H$ is $\mu_H$-strongly convex and $L_H$-smooth, and \cref{F_rho_convex_smooth} holds, where $L_H = \max_{i \in \sI}\frac{\os^2(A_i)}{\mu_i} + \rho\ove{C}+\frac{L_{h^*}}{n}$, $\mu_{\fr} = \frac{\ue{C}}{L_H}$, and $L_{\fr} = \frac{1}{\max\pa{\rho, \frac{\mu_H}{\ove{C}}}}$. Specifically, $\mu_H$ varies with different cases:
  \begin{enumerate}
    \item $h^*$ is $\mu_{h^*}$-strongly convex with $\mu_{h^*} > 0$; in this case $\mu_H = \frac{\mu_{h^*}}{n}>0$.
    \item $g_i = 0$ and $A_i$ has full row rank, $i \in \sI$; in this case
          $$\mu_H = \max\pa{\min_{i \in \sI}\frac{\mins^2(A_i)}{L_i}, \frac{\mine{\A\A\T + \rho L_f\mC}}{L_f}}>0.$$
    \item $g_i = 0$ and $A = [A_1, \cdots, A_n]$ has full row rank, and $\rho > 0$; in this case $\mu_H = \frac{\mine{\A\A\T + \rho L_f\mC}}{L_f}>0$.
  \end{enumerate}
\end{lemma}

\begin{proof}
  See \cref{appendix_F_rho_convex}.
\end{proof}

Based on the KKT conditions for \cref{final_pro}, as detailed in the proof of \cref{1552}, we can verify that a solution to \cref{final_pro} exists in $\Range{\sqrt{\mC}}$ by projecting any solution onto the orthogonal subspaces $\Range{\sqrt{\mC}}$ and $\Null{\sqrt{\mC}}$.
Furthermore, by \cref{F_rho_convex}, $\fr$ is strongly convex on $\Range{\sqrt{\mC}}$ under \cref{G,convex,hs_smooth,1643}. Thus, \cref{final_pro} has a unique solution in $\Range{\sqrt{\mC}}$, which we denote as $\my^*$. Additionally, since both $f$ and $H$ are strongly convex, the solution of \cref{saddle_pro} with $\mz = \sqrt{\mC}\my^*$ must also be unique, denoted as $(\mx^*, \ml^*)$. Since $H$ is strongly convex, $\bm{\Lambda}^*_{k+1}$ must also be a singleton, denoted as $\ml^*_{k+1}$. Then, \cref{1551} can be guaranteed by
\eqe{ \label{lambda_error_condition}
  \norm{\ml\+-\ml_{k+1}^*}^2 \leq \frac{\varepsilon_{k+1}^2}{\ove{C}} = e_{\ml,k+1}.
}
The following theorem demonstrates the linear convergence of $\mx\+$ generated by iD2A.
\begin{theorem} \label{outer_convergence_SC}
  Assume that \cref{strong_duality,G,convex,hs_smooth,1643} hold, $\mx\+$ and $\ml\+$ meet the conditions \cref{x_error_condition} and \cref{lambda_error_condition} respectively, and
  \eqe{ \label{2333}
    e_{\mx,k+1} = \theta_{\mx} e_{\mx,k}, \ e_{\ml,k+1} = \theta_{\ml} e_{\ml,k},
  }
  where $e_{\mx,0}, e_{\ml,0} > 0$, $\theta_{\mx}$, $\theta_{\ml} \in \pa{0, 1}$. Then $\mx\+$ generated by iD2A satisfies that $\norm{\mx\+-\mx^*}^2$ converges as
  \begin{enumerate}
    \item $\mO\lt(\max\lt(1-\frac{1}{\sqrt{\kappa_{\fr}}}, \theta_{\ml}, \theta_{\mx}\rt)^k\rt)$ if $\theta_{\ml} \neq 1-\frac{1}{\sqrt{\kappa_{\fr}}}$;
    \item $\mO\lt(\max\pa{k^2\lt(1-\frac{1}{\sqrt{\kappa_{\fr}}}\rt)^k, \theta_{\mx}^k}\rt)$ if $\theta_{\ml} = 1-\frac{1}{\sqrt{\kappa_{\fr}}}$.
  \end{enumerate}
\end{theorem}

\begin{proof}
  See \cref{appendix_outer_convergence_SC}.
\end{proof}

As demonstrated in \cref{outer_convergence_SC}, the convergence rate of iD2A is governed by $\kappa_{\fr}$, the condition number of $\fr$. Therefore, it is essential to determine the value of $\kappa_{\fr}$.
\begin{lemma} \label{kappa_F}
  Assume that \cref{G,convex,hs_smooth,1643} hold. Then:
  \begin{enumerate}
    \item if $h^*$ is $\mu_{h^*}$-strongly convex with $\mu_{h^*} > 0$, $\kappa_{\fr} = \frac{\max_{i \in \sI}\frac{\os^2(A_i)}{\mu_i} +\frac{L_{h^*}}{n}+ \rho\ove{C}}{\max\pa{\rho\ove{C}, \frac{\mu_{h^*}}{n}}}\frac{\ove{C}}{\ue{C}}$;
    \item if $g_i = 0$ and $A_i$ has full row rank, $i \in \sI$, $\kappa_{\fr} = \frac{\max_{i \in \sI}\frac{\os^2(A_i)}{\mu_i} +\frac{L_{h^*}}{n}+ \rho\ove{C}}{\max\pa{\rho\ove{C}, \min_{i \in \sI}\frac{\mins^2(A_i)}{L_i}, \frac{\mine{\A\A\T + \rho L_f\mC}}{L_f}}}\frac{\ove{C}}{\ue{C}}$;
    \item if $g_i = 0$ and $A = [A_1, \cdots, A_n]$ has full row rank\footnote{Note that we implicitly assume the existence of $A_i$ that does not have full row rank for problems categorized into Case 3); otherwise, the problem should be classified into Case 2).}, and $\rho > 0$,
          $$\kappa_{\fr} = \frac{\max_{i \in \sI}\frac{\os^2(A_i)}{\mu_i} +\frac{L_{h^*}}{n}+ \rho\ove{C}}{\max\pa{\rho\ove{C}, \frac{\mine{\A\A\T + \rho L_f\mC}}{L_f}}}\frac{\ove{C}}{\ue{C}}.$$
  \end{enumerate}
  Furthermore, for the first two cases, if $f_i$ and $h^*$ are twice differentiable, and $g_i = 0$, we can obtain tighter condition numbers:
  \begin{enumerate}
    \item $\kappa_{\fr} = \frac{\max_{i \in \sI}\frac{\os^2(A_i)}{\mu_i}+\frac{L_{h^*}}{n} + \rho\ue{C}}{\frac{\mu_{h^*}}{n} + \rho\ove{C}}\frac{\ove{C}}{\ue{C}}$ for the first case;
    \item $\kappa_{\fr} = \frac{\max_{i \in \sI}\frac{\os^2(A_i)}{\mu_i}+\frac{L_{h^*}}{n} + \rho\ue{C}}{\min_{i \in \sI}\frac{\mins^2(A_i)}{L_i} + \rho\ove{C}}\frac{\ove{C}}{\ue{C}}$ for the second case.
  \end{enumerate}
\end{lemma}

\begin{proof}
  See \cref{appendix_kappa_F}.
\end{proof}

\begin{remark}
  Note that constructing the stepsizes of iD2A requires knowledge of $\kappa_{\fr}$. According to \cref{kappa_F}, $\kappa_{\fr}$ involves $\mine{\A\A\T + \rho L_f\mC}$ for the last two cases, but calculating $\mine{\A\A\T + \rho L_f\mC}$ in a decentralized manner may be challenging. In practice, however, this calculation is unnecessary, detailed explanation is provided in \cref{remark_1}.
\end{remark}

\begin{remark}
  Recall that MiD2A can be viewed as a special case of iD2A with $C$ replaced by $P_K(C)$. Therefore, when analyzing the convergence of MiD2A, we can simply apply the convergence results of iD2A by substituting $C$ with $P_K(C)$. An extra issue for MiD2A is the decentralized computation of $\ove{P_K(C)}$ and $\ue{P_K(C)}$. Fortunately, there exists an alternative solution. According to the proof of Theorem 4 in \cite{scaman2017optimal}, we have $1-2\frac{c_1^K}{1+c_1^{2K}} \leq \ue{P_K(C)} \leq \ove{P_K(C)} \leq 1+2\frac{c_1^K}{1+c_1^{2K}}$, where $c_1 = \frac{\sqrt{\kappa_C}-1}{\sqrt{\kappa_C}+1}$. Thus, when $\ove{P_K(C)}$ and $\ue{P_K(C)}$ are difficult to compute in a decentralized manner, we can simply set $\ove{P_K(C)} = 1+2\frac{c_1^K}{1+c_1^{2K}}$ and $\ue{P_K(C)} = 1-2\frac{c_1^K}{1+c_1^{2K}}$.
\end{remark}

From \cref{kappa_F}, it is evident that the tunable parameter $\rho$ plays an important role in determining the value of $\kappa_{\fr}$. This raises two key questions regarding the selection of $\rho$:
\begin{enumerate}
  \item Should we choose $\rho=0$ or $\rho>0$?
  \item If we choose $\rho>0$, how large should it be?
\end{enumerate}
According to \cref{1643}, we can select either $\rho=0$ or $\rho>0$ for the first two cases, but $\rho>0$ is the only option for the third case. From \cref{kappa_F}, we observe that $\kappa_{\fr}$ decreases as $\rho$ increases, which is desirable. However, it is important to recognize that a nonzero or larger $\rho$ introduces additional challenges in solving the subproblem \cref{inner_pro} compared to the case where $\rho = 0$:
\begin{enumerate}
  \item A larger $\rho$ increases the condition number of $H$, making the subproblem numerically more challenging.
  \item When $\rho = 0$, \cref{inner_pro} can be decomposed to $n$ independent local problems, allowing each agent to solve its subproblem locally without communication. In contrast, for $\rho > 0$, such a decomposition is no longer possible, requiring agents to cooperatively solve \cref{inner_pro} and incurring additional communication overhead.
\end{enumerate}
Thus, selecting a nonzero or larger $\rho$ comes with trade-offs. We will provide a detailed discussion on the choice of $\rho$ in later sections.

Since the convergence rate of iD2A has been explicitly established in \cref{outer_convergence_SC}, its outer iteration complexity can be directly derived. However, to establish the total complexity of iD2A, which includes both communication and computational complexities, we still need to analyze how to solve the subproblem \cref{inner_pro} and determine the associated inner iteration (oracle) complexity.

\section{Convergence Analysis: Overall Complexities} \label{complexity}
Based on the results obtained in \cref{converge_rate}, we are ready to examine the complexities of iD2A. Similar to \cref{converge_rate}, we will focus on the analysis of iD2A, and the results can be directly applied to MiD2A by substituting $C$ with $P_K(C)$.

\subsection{Oracle Complexities of Solving the Subproblem}
We begin by analyzing the oracle complexities associated with solving the subproblem \cref{inner_pro} using first-order algorithms. It is important to note that \cref{inner_pro} represents a special case of the general saddle-point problem described below:
\begin{equation} \label{general_saddle_pro} \tag{SPP}
  \begin{aligned}
  \min_{x \in \R^{d_x}}\max_{y \in \R^{d_y}} &\cL(x, y) \\
  = &f_1(x) + f_2(x) + y\T Bx - g_1(y) - g_2(y),
  \end{aligned}
\end{equation}
where $f_1:\R^{d_x} \rightarrow \R$ and $g_1:\R^{d_y} \rightarrow \R$ are smooth convex functions, $f_2:\R^{d_x} \rightarrow \exs$ and $g_2:\R^{d_y} \rightarrow \exs$ are potentially nonsmooth convex functions, and $B \in \R^{d_y \times d_x}$. We assume that there exists at least one solution $(x^*, y^*)$ to \cref{general_saddle_pro}. Considering the specific form of the subproblem \cref{inner_pro} and the preceding assumptions related to \cref{main_pro}, we focus on \cref{general_saddle_pro} that satisfies the following assumption:
\begin{assumption} \label{2255}
  $f_1$, $f_2$, $g_1$, and $g_2$ satisfy
  \begin{enumerate}
    \item $f_1$ is $\mu_x$-strongly convex and $L_x$-smooth with $L_x \geq \mu_x > 0$;
    \item $g_1$ is convex and $L_y$-smooth with $L_y \geq 0$;
    \item $f_2$ and $g_2$ are proper convex, lower semicontinuous, and proximal-friendly.
  \end{enumerate}
\end{assumption}
\cref{general_saddle_pro} under \cref{2255} has been extensively studied and many efficient algorithms have been proposed, just to name a few, iDAPG \cite{li2025inexact}, LPD \cite{thekumparampil2022lifted}, and ABPD-PGS \cite{luo2024accelerated}.
It is evident that \cref{inner_pro} corresponds to various instantiations of \cref{general_saddle_pro} under different assumptions. We systematically characterize their mapping relationships in \cref{appendix_mapping}. Notably, the subproblem of Case 3 in \cref{tab:complexity_h} corresponds to \cref{general_saddle_pro} satisfying \cref{2255} and the following assumption:
\begin{assumption} \label{general_saddle_pro_assumption}
  $f_2 = 0$, $g_1(y) = g_3(y) + \frac{1}{2}y\T Py + y\T b$, where $g_3$ is a smooth convex function and $P \geq 0$. Additionally, $BB\T + cP > 0$ for any $c > 0$.
\end{assumption}

Once we identify the specific case of \cref{general_saddle_pro} that corresponds to \cref{inner_pro} through their mapping relationships, we can utilize existing algorithms to solve \cref{inner_pro}. To minimize the overall complexity of iD2A, we should select the most efficient algorithm, defined as the one with minimal oracle complexity. In \cref{appendix_oracle_complexity}, we summarize the SOTA oracle complexities for different cases of \cref{general_saddle_pro}, enabling us to conveniently choose the best algorithm and obtain its associated oracle complexity.

\begin{remark}
  The subproblem solver is crucial for enhancing the performance of iD2A. There are two levels of acceleration in iD2A: (1) iD2A itself is an accelerated algorithm for solving \cref{final_pro}; (2) the subproblem solver of iD2A can also serve as an accelerated algorithm for solving \cref{inner_pro}, with the second level of acceleration arising from the decoupled structure of iD2A.
\end{remark}

\begin{figure*}[tb]
  \begin{center}
    \includegraphics[scale=0.32]{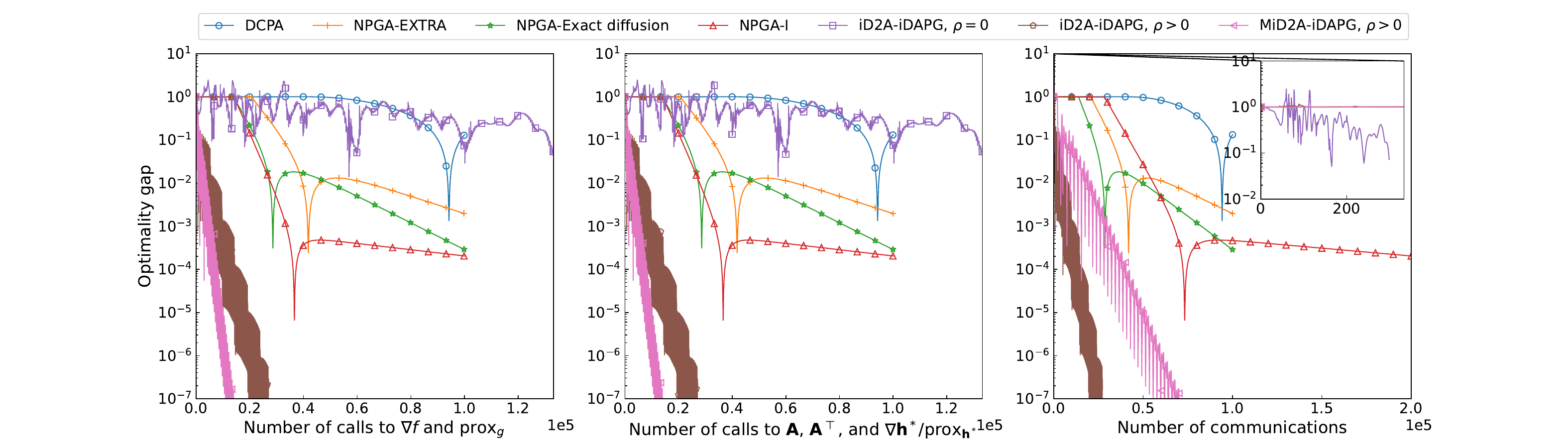}
    \caption {Results of Experiment I: Decentralized Elastic Net Regression ($n = 8, p = 20, d = 9, \kappa_C = 25, \kappa_f = 1, \kappa_{pd} = 98603$)\protect\footnotemark.}
    \label{expt2}
  \end{center}
\end{figure*}
\footnotetext{The optimality gap is defined as $\frac{\norm{\mx^k-\mx^*}}{\|\mx^0-\mx^*\|}$. Definitions of these condition numbers can be found in \cref{tab:complexity_h}. All condition numbers are either rounded to the nearest integer or expressed in scientific notation with three significant figures. The calls to $\A$ and $\A\T$ refer to the matrix-vector multiplications involving $\A$ and $\A\T$.}

\begin{figure*}[tb]
  \begin{center}
    \includegraphics[scale=0.32]{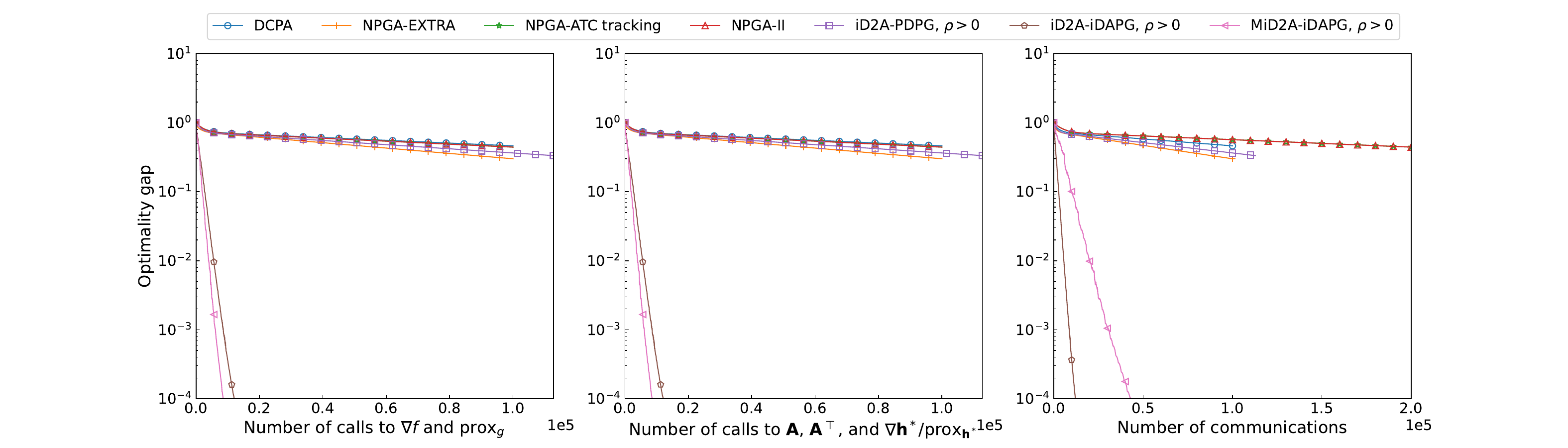}
    \caption {Results of Experiment II: Decentralized Constrained Linear Regression ($n = 8, p = 9, d = 9, \kappa_C = 25, \kappa_f = 1, \kappa_{\A_{\rho}} = 8.24 \times 10^{15}, \kappa_{\A'_{\rho}} = 4.89 \times 10^{16}$).}
    \label{expt3}
  \end{center}
\end{figure*}

When employing an algorithm to solve \cref{general_saddle_pro}, it is often practical to estimate the upper bounds of $\norm{y-y^*}$ and $\norm{x-x^*}$ using easily computable quantities. These estimates can then serve as effective stopping criteria. Here, $(x, y)$ represents the algorithm's output. To aid in this process, we provide the following useful lemma.

\begin{lemma} \label{1518}
  Assume that \cref{2255} holds, and $\varphi(y) = g_1(y) + (f_1+f_2)^*(-B\T y)$ is $\mu_{\varphi}$-strongly convex. Then, for any $(x, y) \in \R^{d_x} \times \R^{d_y}$ such that $\partial f_2(x)$ and $\partial g_2(y)$ are nonempty, we have
  \eqe{
    \norm{y-y^*} &\leq \frac{1}{\mu_{\varphi}}\dist{\0, \partial_y \cL(x, y)} + \frac{\os(B)}{\mu_x\mu_{\varphi}}\dist{\0, \partial_x \cL(x, y)}, \\
    \norm{x-x^*} &\leq \frac{\os(B)}{\mu_x\mu_{\varphi}}\dist{\0, \partial_y \cL(x, y)} \\
    &+ \pa{\frac{1}{\mu_x} + \frac{\os^2(B)}{\mu_x^2\mu_{\varphi}}}\dist{\0, \partial_x \cL(x, y)},
    \nonumber
  }
  where $(x^*, y^*)$ is the unique solution of \cref{general_saddle_pro}.
\end{lemma}

\begin{proof}
  See \cref{appendix_1518}.
\end{proof}

Note that the condition that $\varphi$ is strongly convex can be satisfied for all cases in \cref{tab:complexity_h}. For details, please refer to \cref{appendix_mapping} and \cite[Lemma 1]{li2025inexact}.

\subsection{Overall Complexities of iD2A}
Let us begin by addressing the problem of ensuring that the output of the subproblem solver of iD2A, $(\mx\+, \ml\+)$, satisfies the error conditions \cref{x_error_condition} and \cref{lambda_error_condition}. Since $(\mx_{k+1}^*, \ml_{k+1}^*)$ is unknown prior to solving the subproblem, \cref{x_error_condition,lambda_error_condition} cannot be directly enforced. However, by leveraging \cref{1518}, we can derive upper bounds for $\norm{\mx\+ - \mx_{k+1}^*}$ and $\norm{\ml\+ - \ml_{k+1}^*}$, which leads to the following lemma.
\begin{lemma} \label{2046}
  Assume that \cref{G,convex,1643} hold. Then, for $(\mx\+, \ml\+)$ such that $\partial g(\mx\+)$ and $\partial \mh^*(\ml\+)$ are nonempty, the conditions \cref{x_error_condition} and \cref{lambda_error_condition} are met if
  \eqe{
    &\frac{\os(\A)}{\mu_f\mu_H}\dist{\0, \partial_{\ml} \cT(\mz^k, \mx\+, \ml\+)} \\
    &+ \pa{\frac{1}{\mu_f} + \frac{\os^2(\A)}{\mu_f^2\mu_H}}\dist{\0, \partial_{\mx} \cT(\mz^k, \mx\+, \ml\+)} \leq \sqrt{e_{\mx,k+1}},
    \nonumber
  }
  \eqe{
    &\frac{1}{\mu_H}\dist{\0, \partial_{\ml} \cT(\mz^k, \mx\+, \ml\+)} \\
    &+ \frac{\os(\A)}{\mu_f\mu_H}\dist{\0, \partial_{\mx} \cT(\mz^k, \mx\+, \ml\+)} \leq \sqrt{e_{\ml,k+1}},
    \nonumber
  }
  where $\mu_H$ is defined in \cref{F_rho_convex}.
\end{lemma}

\begin{proof}
  See \cref{appendix_2046}.
\end{proof}

The following theorem establishes the outer and inner iteration complexities of iD2A.
\begin{theorem} \label{iD2A_complexity_convex}
  Under the same assumptions and conditions with \cref{outer_convergence_SC}, choose a constant $c > 1$ and set
  \eqe{ \label{1547}
  \theta_{\ml} &= \theta_{\mx} = 1-\frac{1}{c\sqrt{\kappa_{\fr}}}, \\
  e_{\ml,1} &= \frac{\mu_{\fr}}{\ove{C}}\pa{\sqrt{\theta_{\ml}}-\sqrt{1-\frac{1}{\sqrt{\kappa_{\fr}}}}}^2\pa{\fr(\my^0) - \fr(\my^*)}, \\
  e_{\mx,1} &= \frac{\os^2(\A)}{\mu_f^2}e_{\ml,1}.
  }
  Then, the outer iteration complexity of iD2A (to guarantee $\norm{\mx\+-\mx^*}^2 \leq \epsilon$) is given by
  \eqe{ \label{1545}
    \mO\lt(\sqrt{\kappa_{\fr}} \log \lt(\frac{\mathcal{C}\pa{\fr(\my^0) - \fr(\my^*)}}{\epsilon}\rt)\rt),
  }
  where $\mathcal{C} > 0$ is a constant.
  Suppose algorithm $\mathcal{A}$ can solve \cref{inner_pro} and produce an approximate solution satisfying both \cref{x_error_condition} and \cref{lambda_error_condition}, with an iteration complexity of
  \eqe{ \label{1634}
    \bO{\Delta \log \pa{\frac{\mathcal{C}_1\norm{\mx-\mx_{k+1}^*}^2 + \mathcal{C}_2\norm{\ml-\ml_{k+1}^*}^2}{\min\pa{e_{\mx,k+1}, e_{\ml,k+1}}}}},
  }
  where $\Delta > 0$ and $\mathcal{C}_1, \mathcal{C}_2 \geq 0$ are constants, and $(\mx, \ml)$ is the initialized solution for $\mathcal{A}$. If $\mathcal{A}$ is selected as the solver for \cref{inner_pro} and $(\mx^k, \ml^k)$ is used as the initialized solution for $\mathcal{A}$ at the $k$-th iteration, then the total inner iteration complexity of iD2A is given by
  \eqe{ \label{1546}
    \tilde{\mO}\pa{\Delta \sqrt{\kappa_{\fr}} \log \pa{\frac{\mathcal{C}\pa{\fr(\my^0) - \fr(\my^*)}}{\epsilon}}},
  }
  where $\tilde{\mO}$ hides a logarithmic factor dependent on $\mu_f$, $\mu_H$, $\mu_{\fr}$, $L_{\fr}$, $\os^2(\A)$, $\ove{C}$, $\mathcal{C}_1$, $\mathcal{C}_2$, and $c$.
\end{theorem}

\begin{proof}
  See \cref{appendix_iD2A_complexity_convex}.
\end{proof}

\begin{remark} \label{2303}
  Note that the choice of $e_{\ml,1}$ in \cref{iD2A_complexity_convex} is made solely to achieve a tighter logarithmic constant in \cref{1545}. However, due to the presence of $\fr(\my^*)$, this choice cannot be directly applied. The alternative choice is provided in \cref{remark_2}.
\end{remark}

\begin{remark} \label{0137}
  As demonstrated in \cref{iD2A_complexity_convex}, achieving the desired inner iteration complexity in \cref{1546} requires the subproblem solver of iD2A to satisfy \cref{1634}. Based on established convergence results, LPD \cite{thekumparampil2022lifted}, APDG \cite{kovalev2022accelerated}, the algorithm proposed in \cite{salim2022optimal}, and PDPG \cite{li2025inexact} meet this requirement, whereas ABPD-PGS \cite{luo2024accelerated} does not. Furthermore, as shown in \cite[Theorem 3]{li2025inexact}, iDAPG \cite{li2025inexact} also satisfies the requirement provided that $\Phi$ is smooth, which allows us to bound $\Phi(y^0)-\Phi(y^*)$ by $\norm{y^0-y^*}^2$. Notably, $\Phi$ corresponds to $H$ when solving \cref{saddle_pro}, and under the assumptions for \cref{iD2A_complexity_convex}, $H$ is indeed smooth by \cref{F_rho_convex}. Consequently, when iDAPG is employed as the subproblem solver, the inner iteration complexity of iD2A also adheres to \cref{1546}. To avoid potential confusion, we emphasize that \cref{1634} is intended solely for theoretical analysis. It does not imply that the subproblem solver requires knowledge of the terms $\norm{\mx-\mx_{k+1}^*}$ and $\norm{\ml-\ml_{k+1}^*}$; rather, it indicates that the solver's iteration complexity can be bounded by an expression involving these two quantities. In contrast, \cref{2046} is derived from a practical implementation perspective. It provides a concrete guideline for determining when the inexact solution provided by the subproblem solver satisfies \cref{x_error_condition} and \cref{lambda_error_condition}, thereby offering a clear stopping criterion for the inner loop.
\end{remark}

The outer iteration complexity of iD2A for different cases in \cref{kappa_F} follows directly from \cref{iD2A_complexity_convex} and \cref{kappa_F}. The inner iteration complexity, however, is determined by the iteration complexity of the selected subproblem solver. As established in \cref{appendix_oracle_complexity}, the iteration complexities of existing algorithms for solving \cref{general_saddle_pro}, which includes iD2A's subproblem as a special case, are well characterized. Through the explicit mapping between \cref{inner_pro} and \cref{general_saddle_pro} provided in \cref{appendix_mapping}, we can reformulate \cref{inner_pro} in terms of \cref{general_saddle_pro}, and consequently obtain precise inner iteration complexities for iD2A when using different subproblem solvers.

Once the outer and inner iteration complexities of iD2A are determined, we can evaluate its performance in comparison to existing algorithms. Decentralized algorithms are typically compared based on the communication and computation costs required to achieve a specified accuracy. The communication cost of a decentralized algorithm, regardless of the optimization problem type, is fully characterized by its communication complexity. In contrast, the metric for computation cost varies depending on the optimization problem. For \cref{main_pro}, the primary computational cost arises from evaluating $\nabla f$, $\prox_{g}$, $\nabla \mh^*$ or $\prox_{\mh^*}$, as well as performing matrix-vector multiplications involving $\A$ and $\A\T$. Let $\mathcal{A}$ represent the evaluation of $\nabla f$ and $\prox_{g}$, and let $\mathcal{B}$ denote the evaluation of $\nabla \mh^*$/$\prox_{\mh^*}$ and $\prox_{g_2}$, along with matrix-vector multiplications involving $\A$ and $\A\T$. Then, the computational cost of a decentralized algorithm for \cref{main_pro} can be fully characterized by the oracle complexities of $\mathcal{A}$ and $\mathcal{B}$.

According to the update rules of iD2A, the oracle complexities of $\mathcal{A}$ and $\mathcal{B}$ are solely determined by its inner iteration complexity, while the relationship between its communication complexity and the outer and inner iteration complexities depends on the value of $\rho$. When $\rho=0$, the communication complexity equals the outer iteration complexity; when $\rho>0$, its communication complexity is the sum of the outer and inner iteration complexities. \cref{2138} gives a detailed discussion on the optimal choice of $\rho$.

\begin{remark} \label{2138}
  We first discuss the optimal choice of $\rho$ when $\rho>0$. According to \cref{iD2A_complexity_convex}, \cref{kappa_F}, \cref{appendix_oracle_complexity}, and \cref{appendix_mapping}, the outer complexity of iD2A decreases as $\rho$ increases; while the inner complexity of iD2A is a convex function of $\rho$ on $(0, +\infty)$ that attains its minimum at $\rho^*(C) = \frac{\max_{i \in \sI}\frac{\os^2(A_i)}{\mu_i} +\frac{L_{h^*}}{n}}{\ove{C}}$, provided that any qualified subproblem solver discussed in \cref{0137} is used. Therefore, an approximately optimal choice is to set $\rho = \rho^*(C)$, which achieves the lowest oracle complexities of $\mathcal{A}$ and $\mathcal{B}$ and approximately the lowest communication complexity. For MiD2A, the corresponding choice is to set $\rho = \rho^*\pa{P_K(C)} = \frac{\max_{i \in \sI}\frac{\os^2(A_i)}{\mu_i} +\frac{L_{h^*}}{n}}{\ove{P_K(C)}}$. Based on this choice of $\rho$ and \cref{kappa_F}, we can verify that $\kappa_{\fr} = 2\kappa_C$ for iD2A and $\kappa_{\fr} = 2\kappa_{P_K(C)} \leq 8$ for MiD2A.
  The main advantage of choosing $\rho=0$ is that the communication complexity is solely determined by the outer iteration complexity, which may result in a slightly lower communication complexity compared to $\rho > 0$. However, choosing $\rho=0$ typically leads to significantly worse oracle complexities of $\mathcal{A}$ and $\mathcal{B}$, see \cref{appendix_full_table} for details.
\end{remark}

Based on the above analysis, the communication complexity and the oracle complexities of $\mathcal{A}$ and $\mathcal{B}$ for iD2A and MiD2A can be fully characterized once the subproblem solver and $\rho$ are specified. We present the complexities of iD2A and MiD2A with different subproblem solvers across various scenarios in \cref{appendix_full_table}. For conciseness, \cref{tab:complexity_h} focuses only on a specific version of MiD2A: MiD2A-iDAPG ($\rho>0$), which demonstrates significantly lower communication complexity and oracle complexities of $\mathcal{A}$ and $\mathcal{B}$ than SOTA algorithms in all three cases. It is important to note that other versions of iD2A and MiD2A may exhibit lower communication complexity or oracle complexity of $\mathcal{A}$ for Cases 1 and 2; see \cref{appendix_full_table}. We select MiD2A-iDAPG due to its balanced performance across different metrics, as other versions may have slightly lower communication complexity but substantially higher oracle complexity of $\mathcal{B}$.

\begin{remark}
  The communication and oracle complexities of several primal-dual algorithms listed in \cref{tab:complexity_h}, including DCDA \cite{alghunaim2019proximal}, DCPA \cite{alghunaim2021dual}, and NPGA \cite{li2024npga}, are not explicitly provided in the original literature; only their convergence rates are given. Instead, we have derived these complexities using a methodology similar to that employed for deriving the iteration complexity of PDPG in \cite{li2025inexact}.
\end{remark}

\subsection{Selection of Subproblem Solvers}
The selection of the subproblem solver plays a crucial role in the performance of iD2A and MiD2A. A convenient choice is iDAPG \cite{li2025inexact}, which results in MiD2A-iDAPG exhibiting lower complexities than SOTA algorithms, as shown in \cref{tab:complexity_h}. However, iDAPG may not always be the optimal choice in practice.

To choose the best subproblem solver, we must carefully analyze the characteristics of the network system on which our algorithm will be implemented and the specific problem we aim to solve. For instance, we should identify the bottleneck in the network system: is it communication or computation? If communication is the bottleneck, we should choose a subproblem solver that minimizes communication complexity; otherwise, we should focus on minimizing the oracle complexity of $\mathcal{A}$ and $\mathcal{B}$. Additionally, we need to consider the computational costs of $\mathcal{A}$ and $\mathcal{B}$ for the problem at hand. If the computation cost of $\mathcal{A}$ is significantly higher than that of $\mathcal{B}$, we should opt for a subproblem solver that reduces the oracle complexity of $\mathcal{A}$. Conversely, if $\mathcal{B}$ incurs higher costs, we should choose a solver that lowers the oracle complexity of $\mathcal{B}$.

Based on the above principles and the complexities of different versions of iD2A and MiD2A provided in \cref{appendix_full_table}, we can summarize the selection of subproblem solvers as follows:
\begin{itemize}
  \item \textbf{Case 1}: Choose LPD if the network system is computation-bottlenecked and the computation cost of $\mathcal{A}$ is significantly higher than that of $\mathcal{B}$; otherwise, choose iDAPG.
  \item \textbf{Case 2}: Choose APDG if the network system is computation-bottlenecked, the computation cost of $\mathcal{A}$ is significantly higher than that of $\mathcal{B}$, and $\kappa_f \geq \kappa_{\A}$; otherwise, choose iDAPG.
  \item \textbf{Case 3}: Always choose iDAPG.
\end{itemize}
Thus, iDAPG is approximately the best choice in most cases.

\section{Experiments} \label{experiments}
In this section, we perform numerical experiments to validate the theoretical convergence properties of iD2A and MiD2A, as well as to compare their practical performance against SOTA algorithms. Due to space limitations, we present only two experiments here: (1) decentralized elastic net regression and (2) decentralized constrained linear regression, both within the framework of decentralized vertical federated learning. Additional experiments can be found in \cref{appendix_experiments}.

\subsection{Experiment I: Decentralized Elastic Net Regression}
Consider a dataset with a raw feature matrix $X' \in \R^{p \times (d-1)}$ and a label vector $y \in \R^p$, let $X = [X', \1_p] \in \R^{p \times d}$. The elastic net regression problem can be formulated as
\eqe{ \label{elastic_net}
  \min_{\theta \in \R^d} \ \frac{1}{2p}\norm{X\theta-y}^2 + \alpha\rho\norm{\theta}_1 + \frac{\alpha(1-\rho)}{2}\norm{\theta}^2,
}
where $\alpha > 0$ and $\rho \in (0, 1)$ are tunable parameters.
In the setting of vertical federated learning, the global feature matrix $X$ is vertically partitioned as $X = [X_1, \cdots, X_n]$, where $X_i \in \R^{p \times d_i}$ (with $\sum_{i=1}^nd_i = d$) represents the local feature matrix of agent $i \in \sI$. Then, we can reformulate \cref{elastic_net} as \cref{main_pro} using the following definitions:
\eqe{ \label{2220}
  &f_i \triangleq \frac{\alpha(1-\rho)}{2}\norm{\cdot}^2, \ g_i \triangleq \alpha\rho\norm{\cdot}_1, \ A_i \triangleq X_i, i \in \sI; \\
  &h \triangleq \frac{1}{2p}\norm{\cdot-y}^2,
}
which clearly satisfy \cref{convex,strong_duality}. Also note that $h^*(w) = \frac{p}{2}\norm{w}^2 + y\T w$ is strongly convex and smooth, thus \cref{2220} matches Case 1 in \cref{tab:complexity_h}. Based on \cref{tab:complexity_h}, we select DCPA and NPGA as the baselines in this experiment. It is important to note that NPGA is a unified algorithmic framework with numerous variants. According to \cite{li2024npga}, all of its variants can solve \cref{2220} with a guarantee of linear convergence. For comparison, we have chosen several representative variants: NPGA-EXTRA, NPGA-Exact Diffusion, and NPGA-I.

We use the California housing prices dataset\footnote{The California housing prices dataset is available in \href{https://inria.github.io/scikit-learn-mooc/python_scripts/datasets_california_housing.html}{scikit-learn}.} to construct the problem. This dataset contains 8 features (excluding the housing price) and a total of 20,640 samples. We choose the first $20$ samples from the California housing prices dataset, resulting in the raw dataset denoted as $(X', y)$, where $X' \in \R^{20 \times 8}$ and $y \in \R^{20}$. We then define the global feature matrix as $X = [X', \1_{20}] \in \R^{20 \times 9}$. We consider a system of $8$ agents, and similar to Experiment I, we also use the Erdos-Renyi model with a connectivity probability of $0.1$ to generate the network topology. Based on the number of agents, we partition $X$ as $X = [X_1, \cdots, X_8]$: $X_i \in \R^{20 \times 1}, i \leq 7; X_8 \in \R^{20 \times 2}$. Additionally, we set $\alpha = 100$ and $\rho = 0.1$.

We use the Erdos-Renyi model \cite{erdos1960evolution} with a connectivity probability of $0.1$ to generate the network topology of agents. Both DCPA and NPGA require a symmetric and doubly stochastic mixing matrix $W \in \R^{n \times n}$, which is associated with the network topology. We generate $W'$ using the Laplacian method with $c = 1$ (where $c > 0$ is a tunable parameter in the Laplacian method; refer to \cite[Remark 2]{li2024npga} for details). Subsequently, we obtain $W$ by setting $W = \frac{1}{2}(\I+W')$. By defining $C = \frac{\I-W}{2}$, we can derive a gossip matrix $C$ that satisfies \cref{G} for iD2A. The parameters for DCPA and NPGA are set based on their theoretical convergence results to achieve optimal convergence rates. For iD2A\footnote{We do not discuss about MiD2A separately, as it is essentially a special case of iD2A: replacing the original gossip matrix $C$ with a new matrix $P_k(C)$.}, the parameters are configured according to \cref{iD2A_complexity_convex,2303}.

The experimental results are presented in \cref{expt2}. According to the figures, we observe that iD2A and MiD2A with $\rho>0$ exhibit significantly lower complexities than DCPA and NPGA across all metrics. Benefiting from Chebyshev acceleration, MiD2A has lower computational complexity (as indicated by the first two graphs) but higher communication complexity compared to iD2A, despite both having theoretical communication complexities of the same order. Therefore, in scenarios where the communication cost is much lower than the computation cost, MiD2A is preferable to iD2A. In contrast, iD2A with $\rho = 0$ demonstrates much higher computational complexity than iD2A with $\rho > 0$, but it converges very quickly w.r.t. the number of communications. This makes it a preferable choice when the communication cost dominates the total cost.

\subsection{Experiment II: Decentralized Constrained Linear Regression}
Unlike classical (unconstrained) regression, constrained regression incorporates additional constraints that arise from the specific application scenarios of the regression model \cite{boyd2018introduction}. For instance, in a housing price forecasting model, the predicted prices should be nonnegative. We consider the following constrained linear regression problem with an $\ell_2$ regularizer:
\eqe{ \label{con_linear_regre}
  \min_{\theta \in \R^d} \ &\frac{1}{2p}\norm{X\theta-y}^2 + \frac{\alpha}{2}\norm{\theta}^2 \\
  \text{s.t.} \ &X\theta \geq 0,
}
where $\alpha > 0$ is a tunable parameter. In this experiment, we assume that $X$ has full row rank. In the context of vertical federated learning, we can reformulate \cref{con_linear_regre} as \cref{main_pro} by introducing the following definitions:
\eqe{ \label{2221}
  &f_i \triangleq \frac{\alpha}{2}\norm{\cdot}^2, \ g_i \triangleq 0, \ A_i \triangleq X_i, i \in \sI; \\
  &h \triangleq \frac{1}{2p}\norm{\cdot-y}^2 + \iota_{\R^p_+},
}
which clearly satisfies \cref{convex}. Also note that $\dom{g} = \R^d$ and $A = X$ has full row rank, hence \cref{strong_duality} holds.
It is not difficult to derive that $h^*(w) = \sum_{i=1}^p h_i^*(w_i)$, where $w_i \in \R$ and $h_i^*(w_i) = \lt\{
  \begin{aligned}
    \frac{p}{2}w_i^2 + y_iw_i,      \  & w_i \geq -\frac{y_i}{p}, \\
    0, \                               & w_i < -\frac{y_i}{p}.
  \end{aligned}\rt.$
Clearly, $h^*$ is differentiable and $p$-smooth but not strongly convex. Consequently, \cref{2221} corresponds to Case 3 in \cref{tab:complexity_h}. According to \cref{tab:complexity_h}, we continue to use DCPA and NPGA as baseline algorithms in this experiments. However, unlike in Experiment II, not all variants of NPGA have established the theoretical guarantee for solving \cref{2221}. Therefore, in this experiment, we select a different set of NPGA variants: NPGA-EXTRA, NPGA-ATC tracking, and NPGA-II.

For this experiment, we also utilize the California housing prices dataset, selecting the first $9$ samples as the raw dataset $(X', y)$, where $X' \in \R^{9 \times 8}$ and $y \in \R^{9}$. We should mention that, the resulting global feature matrix $X = [X', \1_{9}] \in \R^{9 \times 9}$ has full row rank. We also consider a system of $8$ agents and use the same network topology as in Experiment II. Similarly, we partition $X$ as $X = [X_1, \cdots, X_8]$, where $X_i \in \R^{9 \times 1}, i \leq 7; X_8 \in \R^{9 \times 2}$. Additionally, we set $\alpha = 100$.

We adopt the same $W$ and $C$ as in Experiment II, and the parameters for DCPA, NPGA, and iD2A are determined using the same methods as in Experiment II. The parameters of PDPG are set according to \cite[Remark 1]{li2025inexact}.
The experimental results are presented in \cref{expt3}. From the figures, we observe that iD2A and MiD2A, which use iDAPG as the subproblem solver, exhibit significantly lower complexities than DCPA and NPGA across all metrics. Similar to Experiment II, MiD2A has lower computational complexity but higher communication complexity compared to iD2A. However, when PDPG is chosen as the subproblem solver, the performance of iD2A significantly degrades, even falling below that of NPGA-EXTRA. This aligns with the fact that iDAPG has lower theoretical complexity than PDPG, and also highlights the impact of the subproblem solver on the performance of iD2A.

\section{Conclusions} \label{conclusion}
This paper addresses a class of decentralized constraint-coupled optimization problems. Based on a novel dual$^2$ approach, we develop two accelerated algorithms: iD2A and MiD2A. Under certain assumptions, we prove the asymptotic and linear convergence of iD2A and MiD2A, and further demonstrate that they achieve significantly lower communication and computational complexities compared to existing algorithms. Numerical experiments validate these theoretical results, highlighting the superior performance of both algorithms in real-world scenarios. Future work may focus on extending the current convergence results to nonconvex objectives and directed graphs.

{\small
  \bibliographystyle{IEEEtran}
  \bibliography{../public/bib}
}

\begin{IEEEbiography}
  [{\includegraphics[width=1in,height=1.25in,clip,keepaspectratio]{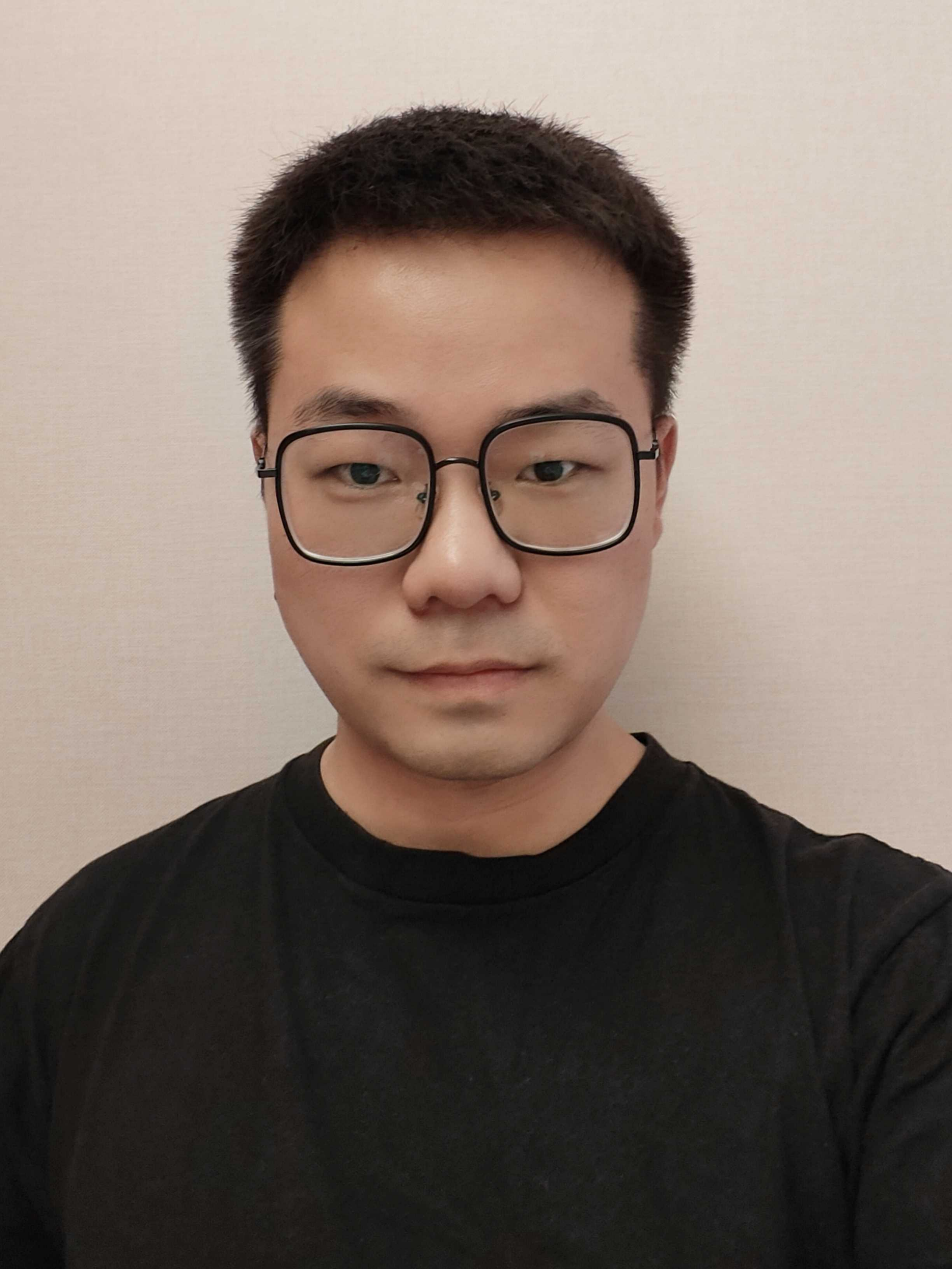}}]
  {Jingwang Li}
  received the B.Mgt. degree in Engineering Management from Huazhong Agricultural University (2019), and the M.Eng. degree in Control Science and Engineering from Huazhong University of Science and Technology (2022), both in Wuhan, China. He was working towards the Ph.D. degree from the Hong Kong University of Science and Technology. His research focuses on optimization and decentralized algorithms.
\end{IEEEbiography}

\begin{IEEEbiography}
  [{\includegraphics[width=1in,height=1.25in,clip,keepaspectratio]{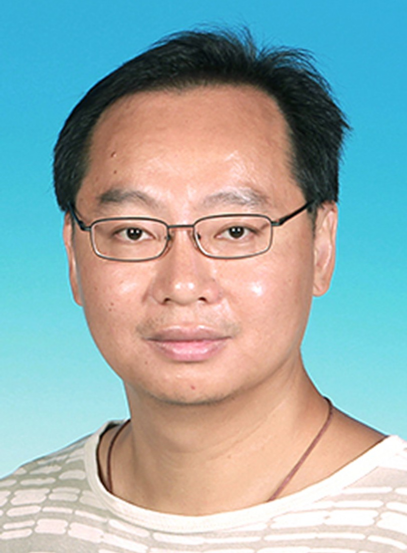}}]
  {Vincent K. N. Lau}
  (Fellow, IEEE) received the B.Eng. degree (Hons.) from The University of Hong Kong in 1992 and the Ph.D. degree from Cambridge University in 1997. He was with Bell Labs from 1997 to 2004 and the Department of Electronic and Computer Engineering (ECE), The Hong Kong University of Science and Technology (HKUST) in 2004. He is currently the Chair Professor and the Founding Director of the Huawei-HKUST Innovation Laboratory, HKUST. His current research interests include robust and delay-optimal cross layer optimization for MIMO/OFDM wireless systems, interference mitigation techniques for wireless networks, massive MIMO, M2M, and network control systems.
\end{IEEEbiography}

\appendices \label{appendix}
\crefalias{section}{appendix}
\section{AcceleratedGossip} \label{appendix_AcceleratedGossip}
\begin{algorithm}[H]
  \caption{\small{AcceleratedGossip \cite{scaman2017optimal}}}
  \label{alg:AccGossip}
  \begin{algorithmic}[1]
    \small
    \Require
    $x, C, K = \lfloor \sqrt{\kappa_C} \rfloor, c_2 = \frac{\kappa_C+1}{\kappa_C-1}, c_3 = \frac{2}{(1+1/\kappa_C)\ove{C}}$
    \Ensure
    $x - \frac{x^K}{a^K}$
    \State $a^0 = 1,a^1 = c_2$
    \State $x^0 = x, x^1 = c_2x(\I-c_3C)$
    \For {$k=0,\dots, K-1$}
    \State $a\+ = 2c_2a^k - a^{k-1}$
    \State $x\+ = 2c_2(\I-c_3C)x^k - x^{k-1}$
    \EndFor
  \end{algorithmic}
\end{algorithm}

\section{Proof of \cref{1552}} \label{appendix_1552}
\begin{appendix_proof}[\cref{1552}]
  Consider the Lagrangian dual function of \cref{main_pro}:
  \eqe{ \label{0024}
    \phi(\lambda) =& \inf_{x_1,\cdots, x_n; y}\bigg(\sum_{i=1}^n\pa{f_i(x_i) + g_i(x_i)} + h(y) \\
    &+ \dotprod{\lambda, \sum_{i=1}^{n}A_ix_i - y}\bigg) \\
    =& \inf_{x_1,\cdots, x_n}\pa{\sum_{i=1}^n\pa{f_i(x_i) + g_i(x_i) + \lambda\T A_ix_i}} \\
    &+ \inf_{y}\pa{h(y) - \lambda\T y} \\
    =& \sum_{i=1}^n\pa{\inf_{x_i}\pa{f_i(x_i) + g_i(x_i) + \lambda\T A_ix_i}} - h^*(\lambda) \\
    =& \sum_{i=1}^n\phi'_i(\lambda) - h^*(\lambda),
    \nonumber
  }
  where $\lambda \in \R^p$ is the dual variable. Then, the Lagrangian dual problem of \cref{main_pro} can be formulated as \cref{dual_pro_eq}, where $\phi_i = -\phi'_i$. According to \cref{convex,strong_duality}, the solution to \cref{dual_pro_eq} exists.

  Recall that \cref{G} ensures that $C$ is symmetric and positive semidefinite, allowing it to be decomposed as $C = \sqrt{C}\sqrt{C}$ via eigenvalue decomposition. Apparently, $\sqrt{C}$ is also symmetric, positive semidefinite, and satisfies $\Null{\sqrt{C}} = \Span{\1_n}$. Since both $C$ and $\sqrt{C}$ are symmetric and $\Null{C} = \Null{\sqrt{C}}$, it follows that $\Range{C} = \Range{\sqrt{C}}$.

  Let $\ml = \col{\lambda} \in \R^{np}$ and $\sqrt{\mC} = \sqrt{C} \otimes \I_p$. It is apparent that \cref{dual_pro_eq} is equivalent to \cref{dual_pro_consensus} because $\Null{\sqrt{C}} = \Span{\1_n}$, which implies that the solution to \cref{dual_pro_consensus} also exists.

  Consider the augmented Lagrangian of \cref{dual_pro_consensus}:
  \eqe{
    \cL_{\rho}(\ml, \my) = \mph(\ml) + \mh^*(\ml) + \my\T\sqrt{\mC}\ml + \frac{\rho}{2}\ml\T\mC\ml,
  }
  where $\mC = C \otimes \I_p$, $\my \in \R^{np}$ is the dual variable, and $\rho \geq 0$ is the augmented parameter. The augmented Lagrangian dual function is then defined as
  \eqe{
  F'_{\rho}(\my) &= \inf_{\ml} \cL_{\rho}(\ml, \my) = -H^*\pa{-\sqrt{\mC}\my},
  }
  where
  \eqe{
    H(\ml) = \mph(\ml) + \mh^*(\ml) + \frac{\rho}{2}\ml\T\mC\ml.
  }
  Then, the augmented Lagrangian dual problem of \cref{dual_pro_consensus} is given by \cref{final_pro}, where $\fr = -F'_{\rho}$.

  We now prove that the solution to \cref{final_pro} exists.
  According to \cref{convex}, $f_i+g_i$ is closed proper and $\mu_i$-strongly convex, and $h$ is closed proper convex. Consequently, $\pa{f_i+g_i}^*$ is convex and $\frac{1}{\mu_i}$-smooth by \cref{conjugate_subdiff}, and $h^*$ (and thus $\mh^*$) is closed proper convex by \cite[Theorem E.1.1.2]{hiriart2004fundamentals}.
  Note that
  \eqe{
    \phi_i(\lambda) &= -\inf_{x_i}\pa{f_i(x_i) + g_i(x_i) + \lambda\T A_ix_i} \\
    &= \sup_{x_i}\pa{-\lambda\T A_ix_i - \pa{f_i(x_i) + g_i(x_i)}} \\
    &= \pa{f_i+g_i}^*(-A_i\T\lambda),
  }
  hence $\phi_i$ is convex and continuously differentiable. Furthermore, we have
  \eqe{
    &\dotprod{\nabla \phi_i(\lambda)-\nabla \phi_i(\lambda'), \lambda-\lambda'} \\
    =& \dotprod{\nabla (f_i + g_i)^*(-A_i\T\lambda)-\nabla (f_i + g_i)^*(-A_i\T\lambda'), -A_i\T(\lambda-\lambda')} \\
    \leq& \frac{1}{\mu_i}\norm{A_i\T(\lambda-\lambda')}^2 \\
    \leq& \frac{\os^2(A_i)}{\mu_i}\norm{\lambda-\lambda'}^2, \ \forall \lambda, \lambda' \in \R^p,
    \nonumber
  }
  thus $\phi_i$ is also $\frac{\os^2(A_i)}{\mu_i}$-smooth, which implies that $\mph$ is convex and $\max_{i \in \sI}\frac{\os^2(A_i)}{\mu_i}$-smooth. Additionally, since $\mC \geq 0$ under \cref{G}, it follows that $q(\ml) = \frac{\rho}{2}\ml\T\mC\ml$ is convex and smooth. Given that $\mph$, $\mh^*$ and $q$ are all closed proper convex, and $\dom{\mph} \cap \dom{\mh^*} \cap \dom{q} = \dom{\mh^*} \neq \emptyset$, $H$ is also closed proper convex \cite[Proposition B.2.1.1]{hiriart2004fundamentals}.

  Notice that we can regard \cref{final_pro} as the standard Lagrangian dual problem of
  \eqe{ \label{1445}
    \min_{\ml \in \R^{np}} \ &\mph(\ml) + \mh^*(\ml) + \frac{\rho}{2}\ml\T\mC\ml \\
    \text{s.t.} \ &\sqrt{\mC}\ml = \0,
  }
  which can be rewritten as
  \eqe{ \label{14450}
    \min_{\ml \in \R^{np}} \ H(\ml) + \iota_{\set{0}}\pa{\sqrt{\mC}\ml},
  }
  where $\iota_{\set{0}}$ is also closed proper convex. By applying \cite[Theorem 31.3]{rockafellar1970convex} to \cref{14450}, we can conclude that $\my^* \in \R^{np}$ is a solution of \cref{final_pro} if there exists $\ml^* \in \R^{np}$ such that $(\ml^*, \my^*)$ satisfies the KKT conditions:
  \eqa{
    \0 &\in \partial H(\ml^*) + \sqrt{\mC}\my^* \nonumber \\
    &= \nabla \mph(\ml^*) + \partial\mh^*(\ml^*) + \rho\mC\ml^* + \sqrt{\mC}\my^*,  \label{1450} \\
    \0 &= \sqrt{\mC}\ml^*, \label{1507}
  }
  where the first equality follows from \cref{subdiff}.

  Notice that the solution of \cref{dual_pro_eq} exists under \cref{convex,strong_duality}. Let $\lambda^* \in \R^p$ be any solution of \cref{dual_pro_eq}. According to \cite[Theorem 27.1]{rockafellar1970convex} and \cref{subdiff}, we have
  \eqe{
    \0 \in \sum_{i=1}^n \nabla \phi_i(\lambda^*) + \partial h^*(\lambda^*).
  }
  Equivalently, there exists $S_{h^*}(\lambda^*) \in \partial h^*(\lambda^*)$ such that
  \eqe{ \label{1511}
    \0 = \sum_{i=1}^n \phi_i(\lambda^*) + S_{h^*}(\lambda^*).
  }
  Let $\ml' = \1_n \otimes \lambda^*$, we have
  \eqe{ \label{15110}
  \nabla \mph(\ml') &= \coll{\nabla \phi_1(\lambda^*), \cdots,\nabla \phi_n(\lambda^*)}, \\
  S_{\mh^*}(\ml') &= \coll{\frac{1}{n}S_{h^*}(\lambda^*), \cdots, \frac{1}{n}S_{h^*}(\lambda^*)} \in \partial\mh^*(\ml').
  }
  Clearly, $\ml'$ satisfies \cref{1507} under \cref{G}.
  Additionally, combining \cref{1511,15110} yields
  \eqe{
    \pa{\1_n\T \otimes \I_p}\pa{\nabla \mph(\ml') + S_{\mh^*}(\ml')} = \0,
  }
  which implies that $\nabla \mph(\ml') + S_{\mh^*}(\ml') \in \Null{\1_n\T \otimes \I_p} = \Range{\sqrt{\mC}}$, hence there exists $\mz \in \R^{np}$ such that
  \eqe{
    \nabla \mph(\ml') + S_{\mh^*}(\ml') = \sqrt{\mC}\mz.
  }
  Consequently, $(\ml', -\mz)$ satisfies \cref{1450,1507}, which implies that $-\mz$ is a solution of \cref{final_pro}.
\end{appendix_proof}

\section{Proof of \cref{2355}} \label{appendix_2355}
\begin{appendix_proof}[\cref{2355}]
  In the proof of \cref{1552}, we have shown that: (1) $\mph + \mh^*$ is closed proper convex under \cref{G,convex}, implying that \cref{dual_pro_consensus} can be seen as a special case of the linearly constrained optimization problem studied in \cite{li2025smooth}; and (2) the solution to the dual of \cref{1445} exists. Since the KKT conditions of \cref{dual_pro_consensus} and \cref{1445} are identical, the latter implies the solution to the dual of \cref{dual_pro_consensus} also exists. Therefore, we can immediately complete the proof by applying \cite[Theorem 1]{li2025smooth} to \cref{dual_pro_consensus}.
\end{appendix_proof}

\section{Proof of \cref{23551}} \label{appendix_23551}
\begin{appendix_proof}[\cref{23551}]
  Recall that $\mph + \mh^*$ is closed proper convex under \cref{G,convex}, and note that $\fr$ is differentiable everywhere implies $\dom{\fr} = \R^{np}$. Then, following the same approach used to prove the third conclusion of \cite[Theorem 1]{li2025smooth}, we can also show that, for any $\my \in \R^{np}$, there exists at least one solution $\ml^+ \in \R^{np}$ of $\min_{\ml \in \R^{np}} \cL_{\rho}(\ml, \my)$, and $\nabla \fr(\my) = -\sqrt{\mC}\ml^+$ holds for any $\ml^+$. According to \cite[Theorem 27.1]{rockafellar1970convex} and \cref{subdiff}, $\ml^+$ satisfies
  \eqe{ \label{2341}
    \0 \in \nabla \mph(\ml^+) + \partial \mh^*(\ml^+) + \rho\mC\ml^+ + \sqrt{\mC}\my.
  }

  Also note that $f+g$ is closed proper and $\mu_f$-strongly convex, then $(f+g)^*$ is convex and $\frac{1}{\mu_f}$-smooth by \cref{conjugate_smooth}. We also have
  \eqe{
    \mph(\ml) = \sup_{\mx} -\ml\T\A\mx - \pa{f(\mx) + g(\mx)} = \pa{f+g}^*(-\A\T\ml).
    \nonumber
  }
  then $\nabla \mph(\ml^+) = -\A\nabla\pa{f+g}^*(-\A\T\ml^+)$. Let $\mx^+ = \nabla\pa{f+g}^*(-\A\T\ml^+)$, substituting $\nabla \mph(\ml^+) = -\A\mx^+$ in \cref{2341} gives
  \eqe{ \label{2351}
    \0 &\in \A\mx^+ - \partial \mh^*(\ml^+) - \rho\mC\ml^+ - \sqrt{\mC}\my, \\
    \0 &\in \A\T\ml^+ + \nabla f(\mx^+) + \partial g(\mx^+),
  }
  where the last equality follows from \cref{conjugate_subdiff} and \cref{subdiff}.
  According to \cite[Proposition 19.20]{bauschke2017convex}, $(\mx^+, \ml^+)$ that satisfies \cref{2351} is necessarily a solution of \cref{saddle_pro} with $\mz = \sqrt{\mC}\my$.
\end{appendix_proof}

\section{Proof of \cref{relation_solutions}} \label{appendix_relation_solutions}
\begin{appendix_proof}[\cref{relation_solutions}]
  Since $\my^*$ is a solution of \cref{final_pro}, according to \cref{23551}, there exists a solution $(\mx^*, \ml^*) \in \R^d \times \R^{np}$ of \cref{saddle_pro} with $\mz = \sqrt{\mC}\my^*$ such that
  \eqe{ \label{1806}
    \0 = \nabla \fr(\my^*) = -\sqrt{\mC}\ml^*.
  }
  According to \cite[Proposition 19.20]{bauschke2017convex} and \cref{subdiff}, $(\mx^*, \ml^*)$ satisfies
  \eqa{
    \0 &\in \A\mx^* - \partial \mh^*(\ml^*) - \rho\mC\ml^* - \sqrt{\mC}\my, \label{23510} \\
    \0 &\in \A\T\ml^* + \nabla f(\mx^*) + \partial g(\mx^*), \label{23511}
  }
  Recall that $\Null{\sqrt{C}} = \Span{\1_n}$, thus \cref{1806} implies that $\ml^* = \1_n \otimes \lambda^*$ with $\lambda^* \in \R^p$.
  Substituting $\ml^* = \1_n \otimes \lambda^*$ in \cref{23510,23511} and left multiplying \cref{23510} with $\1_n\T \otimes \I_p$ gives
  \eqe{ \label{16330}
    \0 &\in A\T\lambda^* + \nabla f(\mx^*) + \partial g(\mx^*), \\
    \0 &\in A\mx^* - \partial h^*(\lambda^*).
  }
  According to \cite[Theorem 31.3]{rockafellar1970convex}, $\mx^*$ that satisfies \cref{16330} is necessarily a solution of \cref{main_pro}. Furthermore, $\mx^*$ is unique since $f$ is strongly convex.
\end{appendix_proof}

\section{Proof of \cref{1717}} \label{appendix_1717}
\begin{appendix_proof}[\cref{1717}]
  Under the error condition \cref{1551}, \cref{analyze_form} can be viewed as a special case of the inexact APG proposed in \cite{schmidt2011convergence}. Consequently, by applying \cite[Proposition 2]{schmidt2011convergence}, we derive the following lemma.
  \begin{lemma}* \label{1336}
    Assume that \cref{strong_duality,G,convex,0129} hold, $\beta_k = \frac{k}{k+3}$, and $\ml\+$ meets the condition \cref{1551}. Then, $\my^k$ generated by \cref{analyze_form} satisfies ($\forall k \geq 1$)
    \eqe{
      \fr(\my^k) - \fr(\my^*) \leq \frac{2L_{\fr}}{(k+1)^2}\pa{\norm{\my^0-\my^*} + 2\sum_{i=1}^k\frac{i\varepsilon_i}{L_{\fr}}}^2.
      \nonumber
    }
  \end{lemma}

  From \cref{1336}, it is evident that $\my^k$ generated by \cref{analyze_form} converges provided that the series $\set{k\epsilon_k}$ satisfies certain conditions, such as summability.

  According to \cref{1644} and $\delta > 0$, we have $\sum_{i=1}^ki\varepsilon_i < \infty$. It follows from \cref{1336} that $\fr(\my^k) - \fr(\my^*) = \bO{\frac{1}{k^2}}$. This implies that $\my^k$ converges to a solution of \cref{final_pro}, so does $\mv^k$. As a result, the conclusion follows from \cref{relation_solutions}, \cref{1644}, and the equivalence between \cref{analyze_form} and iD2A.
\end{appendix_proof}

\section{Proof of \cref{F_rho_smooth}} \label{appendix_F_rho_smooth}
\begin{appendix_proof}[\cref{F_rho_smooth}]
  The first case is proved in \cref{2355}.

  Recall that $\fr(\my) = H^*\pa{-\sqrt{\mC}\my}$. Assume that $H$ is $\mu_H$-strongly convex, we first prove that $\fr$ is $\frac{1}{\max\pa{\rho, \frac{\mu_H}{\ove{C}}}}$-smooth. We have shown in the proof of \cref{1552} that $H$ is closed proper convex, according to \cref{conjugate_smooth}, $H^*$ is $\frac{1}{\mu_H}$-smooth, then $\fr$ is continuously differentiable. Using the smoothness of $H^*$ gives
  \eqe{
    &\dotprod{-\nabla \fr(\my)+\nabla \fr(\my'), \my-\my'} \\
    =&\dotprod{\nabla H^*\pa{-\sqrt{\mC}\my} - \nabla H^*\pa{-\sqrt{\mC}\my'}, -\sqrt{\mC}\pa{\my-\my'}} \\
    \leq& \frac{1}{\mu_H}\norm{\sqrt{\mC}\pa{\my-\my'}}^2 \\
    \leq& \frac{\ove{C}}{\mu_H}\norm{\my-\my'}^2, \ \forall \my, \my' \in \R^{np},
  }
  hence $\fr$ is $\frac{\ove{C}}{\mu_H}$-smooth. Also note that $\fr$ is $\frac{1}{\rho}$-smooth when $\rho>0$, thus $\fr$ is $\frac{1}{\max\pa{\rho, \frac{\mu_H}{\ove{C}}}}$-smooth for $\rho \geq 0$. Therefore, to finish the proof, we only need to prove that $H$ is strongly convex for the last two cases.

  For the second case, $\mh^*$ is clearly $\frac{\mu_{h^*}}{n}$-strongly convex, so is $H$.

  For the third case, $f_i+g_i$ is $\mu_i$-strongly convex and $L_i$-smooth. According to \cref{conjugate_smooth,conjugate_strong_convex}, $(f_i+g_i)^*$ is $\frac{1}{L_i}$-strongly convex and $\frac{1}{\mu_i}$-smooth on $\R^p$, hence $\phi_i$ is continuously differentiable. Using the strong convexity of $(f_i + g_i)^*$ yields
  \eqe{
    &\dotprod{\nabla \phi_i(\lambda)-\nabla \phi_i(\lambda'), \lambda-\lambda'} \\
    =& \dotprod{\nabla (f_i + g_i)^*(-A_i\T\lambda)-\nabla (f_i + g_i)^*(-A_i\T\lambda'), -A_i\T(\lambda-\lambda')} \\
    \geq& \frac{1}{L_i}\norm{A_i\T(\lambda-\lambda')}^2 \\
    \geq& \frac{\mins^2(A_i)}{L_i}\norm{\lambda-\lambda'}^2, \ \forall \lambda, \lambda' \in \R^p,
  }
  where $\frac{\mins^2(A_i)}{L_i} > 0$ since $A_i$ has full row rank. Therefore, $\phi_i$ is $\frac{\mins^2(A_i)}{L_i}$-strongly convex, which implies that $\phi$ is $\min_{i \in \sI}\frac{\mins^2(A_i)}{L_i}$-strongly convex, so is $H$.
\end{appendix_proof}

\section{Proof of \cref{F_rho_convex}} \label{appendix_F_rho_convex}
\begin{appendix_proof}[\cref{F_rho_convex}]
  We have proved in \cref{1552} that $\mph$ is convex and $\max_{i \in \sI}\frac{\os^2(A_i)}{\mu_i}$-smooth.
  Note that \cref{hs_smooth} implies that $\mh^*$ is $\frac{L_{h^*}}{n}$-smooth, hence $H$ is $L_H$-smooth, where $L_H = \max_{i \in \sI}\frac{\os^2(A_i)}{\mu_i} + \rho\ove{C}+\frac{L_{h^*}}{n}$.

  In the proof of \cref{F_rho_smooth}, we have proved that $H$ is strongly convex for the first two cases and $\fr$ is smooth for all the three cases of \cref{F_rho_convex}. We first prove that $H$ is also strongly convex for the third case.
  For the third case, $f+g$ is $\mu_f$-strongly convex and $L_f$-smooth. According to \cref{conjugate_smooth,conjugate_strong_convex}, $(f+g)^*$ is $\frac{1}{L_f}$-strongly convex and $\frac{1}{\mu_f}$-smooth.
  Recall that
  \eqe{
    \mph(\ml) = (f+g)^*(-\A\T\ml),
  }
  hence $\mph$ is continuously differentiable. Let
  \eqe{
    \mph_{\rho}(\ml) = \mph(\ml) + \frac{\rho}{2}\ml\T\mC\ml,
  }
  using the strong convexity of $(f+g)^*$ yields
  \eqe{ \label{2126}
    &\dotprod{\nabla \mph_{\rho}(\ml) - \nabla \mph_{\rho}(\ml'), \ml-\ml'} \\
    =& \dotprod{\nabla (f+g)^*(-\A\T\ml)-\nabla (f+g)^*(-\A\T\ml'), -\A\T\pa{\ml-\ml'}} \\
    &+ \rho(\ml-\ml')\T\mC(\ml-\ml') \\
    \geq& \frac{1}{L_f}(\ml-\ml')\T\lt(\A\A\T+\rho L_f\mC\rt)(\ml-\ml') \\
    \geq& \frac{\mine{\A\A\T + \rho L_f\mC}}{L_f}\norm{\ml-\ml'}^2, \ \forall \ml, \ml' \in \R^{np}.
  }
  According to \cite[Lemma 4]{li2024npga}, $\A\A\T + \rho L_f\mC > 0$ if \cref{G} holds, $A$ has full row rank, and $\rho > 0$, hence $\frac{\mine{\A\A\T + \rho L_f\mC}}{L_f} > 0$, which implies that $\mph_{\rho}$ is $\frac{\mine{\A\A\T + \rho L_f\mC}}{L_f}$-strongly convex.
  Also note that $A_i$ has full row rank implies that $A$ have full row rank, hence \cref{2126} also holds for the second case, which implies that $\mph_{\rho}$ is $\max\pa{\min_{i \in \sI}\frac{\mins^2(A_i)}{L_i}, \frac{\mine{\A\A\T + \rho L_f\mC}}{L_f}}$-strongly convex.

  Therefore, $H$ is strongly convex and smooth under \cref{G,convex,hs_smooth,1643}, which implies that $H^*$ is $\frac{1}{L_H}$-strongly convex and $\frac{1}{\mu_H}$-smooth.
  Recall that $\fr(\my) = H^*\pa{-\sqrt{\mC}\my}$, using the strong convexity of $H^*$ gives
  \eqe{
    &\dotprod{\nabla \fr(\my)-\nabla \fr(\my'), \my-\my'} \\
    =&\dotprod{\nabla H^*\pa{-\sqrt{\mC}\my} - \nabla H^*\pa{-\sqrt{\mC}\my'}, -\sqrt{\mC}\pa{\my-\my'}} \\
    \geq& \frac{1}{L_H}\norm{\sqrt{\mC}\pa{\my-\my'}}^2 \\
    \geq& \frac{\ue{C}}{L_H}\norm{\my-\my'}^2, \ \forall \my, \my' \in \Range{\sqrt{\mC}},
  }
  where the last inequality follows from $\my-\my' \in \Range{\sqrt{\mC}}$. Therefore, $\fr$ is $\frac{\ue{C}}{L_H}$-strongly convex on $\Range{\sqrt{\mC}}$.
\end{appendix_proof}

\section{Proof of \cref{outer_convergence_SC}} \label{appendix_outer_convergence_SC}
\begin{appendix_proof}[\cref{outer_convergence_SC}]
  We first study the convergence of $\my^k$.
  \begin{lemma} \label{1135}
    Assume that \cref{strong_duality,G,convex,hs_smooth,1643} hold, $\ml\+$ meets the condition \cref{1551}, and $\mv^0 = \my^0 = \0$, then $\my^k$ generated by \cref{analyze_form} satisfies
    \eqe{
      \fr(\my^k) - \fr(\my^*) \leq \lt(1-\frac{1}{\sqrt{\kappa_{\fr}}}\rt)^k\bigg(\sqrt{2(\fr(\my^0) - \fr(\my^*))} \\
      + \sqrt{\frac{2}{\mu_{h_\rho}}}\mathcal{E}_k\bigg)^2, \ \forall k \geq 1,
      \nonumber
    }
    where $\kappa_{\fr} = \frac{L_{h_\rho}}{\mu_{h_\rho}}$ and $\mathcal{E}_k = \sum_{i=1}^k\lt(1-\sqrt{\frac{1}{\kappa_{\fr}}}\rt)^{-i/2}\varepsilon_i$.
  \end{lemma}

  \begin{proof}
    See \cref{appendix_1135}.
  \end{proof}

  Based on \cref{1135}, we can easily verify that $\fr(\my^k) - \fr(\my^*)$ will converge linearly to $0$ with a rate $1-\frac{1}{\sqrt{\kappa_{\fr}}}$ if $\varepsilon_k^2$ decreases linearly to $0$ with a rate $\theta_{\ml} < 1-\frac{1}{\sqrt{\kappa_{\fr}}}$. We now proceed to bound $\norm{\mx\+-\mx^*}^2$ in terms of $\fr(\my^k) - \fr(\my^*)$.

  \begin{lemma} \label{1136}
    Assume that \cref{strong_duality,G,convex,hs_smooth,1643} hold, $\mx\+$ meets the condition \cref{x_error_condition}, then $\mx\+$ generated by \cref{analyze_form} satisfies
    \eqe{
      \norm{\mx\+-\mx^*}^2 \leq \frac{4\os^2(\A)\ove{C}}{\mu_f^2\mu_{H}^2\mu_{\fr}}\bigg(\frac{2\sqrt{\kappa_{\fr}}}{\sqrt{\kappa_{\fr}}+1}\sqrt{\Gamma^k} \\
      + \frac{\sqrt{\kappa_{\fr}}-1}{\sqrt{\kappa_{\fr}}+1}\sqrt{\Gamma^{k-1}}\bigg)^2 + 2e_{\mx,k+1}, \ \forall k \geq 1,
    }
    where $\Gamma^k = \fr(\my^k) - \fr(\my^*)$.
  \end{lemma}

  \begin{proof}
    See \cref{appendix_1136}.
  \end{proof}

  With \cref{1135,1136} in hand, we are now ready to prove \cref{outer_convergence_SC}.
  According to \cref{2333}, we have
  \eqe{ \label{2203}
    \mathcal{E}_k &= \sum_{i=1}^k\lt(1-\frac{1}{\sqrt{\kappa_{\fr}}}\rt)^{-i/2}\epsilon_i = \frac{\varepsilon_1}{\sqrt{\theta_{\ml}}}\sum_{i=1}^k\lt(\frac{\theta_{\ml}}{1-\frac{1}{\sqrt{\kappa_{\fr}}}}\rt)^{i/2} \\
    &< \frac{\varepsilon_1}{\sqrt{1-\frac{1}{\sqrt{\kappa_{\fr}}}}-\sqrt{\theta_{\ml}}}, \ \text{if} \ \theta_{\ml} < 1-\frac{1}{\sqrt{\kappa_{\fr}}}; \\
    \mathcal{E}_k &= \frac{\pa{\lt(\frac{\theta_{\ml}}{1-\frac{1}{\sqrt{\kappa_{\fr}}}}\rt)^{k/2}-1}\varepsilon_1}{\sqrt{\theta_{\ml}}-\sqrt{1-\frac{1}{\sqrt{\kappa_{\fr}}}}} \\
    &< \frac{\varepsilon_1}{\sqrt{\theta_{\ml}}-\sqrt{1-\frac{1}{\sqrt{\kappa_{\fr}}}}}\lt(\frac{\theta_{\ml}}{1-\frac{1}{\sqrt{\kappa_{\fr}}}}\rt)^{k/2}, \ \text{if} \ \theta_{\ml} > 1-\frac{1}{\sqrt{\kappa_{\fr}}}; \\
    \mathcal{E}_k &= k\frac{\varepsilon_1}{\sqrt{\theta_{\ml}}}, \ \text{if} \ \theta_{\ml} = 1-\frac{1}{\sqrt{\kappa_{\fr}}}.
  }
  Combining \cref{1136}, \cref{lambda_error_condition}, \cref{1135}, \cref{2203}, and the equivalence between \cref{analyze_form,iD2A} completes the proof.
\end{appendix_proof}

\section{Proof of \cref{1135}} \label{appendix_1135}
\begin{appendix_proof}[\cref{1135}]
  Note that \cref{analyze_form} can be regarded an inexact AGD that applies to \cref{final_pro}. Therefore, we can apply \cite[Proposition 4]{schmidt2011convergence}, which analyzes the convergence of the general inexact AGD under the strong convexity assumption, to study the convergence of \cref{analyze_form}. The challenge is that \cite[Proposition 4]{schmidt2011convergence} assumes that the objective function is globally strongly convex, while $\fr$ is only strongly convex on $\Range{\sqrt{\mC}}$. However, the initialization $\mv^0 = \my^0 = \0$ guarantees that $\mv^k, \my^k \in \Range{\sqrt{\mC}}$ for any $k \geq 0$. Thus, the strong convexity of $\fr$ holds at $\mv^k$, $\my^k$, $\my^*$, and their linear combinations, implying that \cite[Proposition 4]{schmidt2011convergence} still holds for \cref{analyze_form}. Therefore, we can complete the proof by applying \cite[Proposition 4]{schmidt2011convergence} to \cref{analyze_form}.
\end{appendix_proof}

\section{Proof of \cref{1136}} \label{appendix_1136}
\begin{appendix_proof}[\cref{1136}]
  Recall that $H^*$ is $\frac{1}{L_H}$-strongly convex and $\frac{1}{\mu_H}$-smooth, $\my^*$ is the unique solution of \cref{final_pro} in $\Range{\sqrt{\mC}}$, $(\mx^*, \ml^*)$ and $(\mx_{k+1}^*, \ml_{k+1}^*)$ are the exact solutions of \cref{saddle_pro} with $\mz = \sqrt{\mC}\my^*$ and $\mz = \sqrt{\mC}\mv^k$, respectively.
  According to \cref{23551} and \cite[Lemma 36.2]{rockafellar1970convex}, we have
  \eqa{
    &\min_{\mx \in \R^d}\max_{\ml \in \R^{np}}\cT\pa{\sqrt{\mC}\my, \mx, \ml} \\
    =& \min_{\mx \in \R^d}\max_{\ml \in \R^{np}} f(\mx) + g(\mx) + \ml\T\A\mx \\
    &- \pa{\mh^*(\ml) + \frac{\rho}{2}\ml\T\mC\ml + \ml\T\sqrt{\mC}\my} \\
    =& -\min_{\ml \in \R^{np}} H(\ml) + \my\T\sqrt{\mC}\ml \label{dual2_pro},
  }
  hence $\ml^*$ and $\ml_{k+1}^*$ satisfy
  \eqe{
    -\sqrt{\mC}\my^* &= \nabla H(\ml^*), \\
    -\sqrt{\mC}\mv^k &= \nabla H(\ml_{k+1}^*).
  }
  According to \cref{conjugate_subdiff}, we have
  \eqe{
    \ml^* &= \nabla H^*(-\sqrt{\mC}\my^*), \\
    \ml_{k+1}^* &= \nabla H^*(-\sqrt{\mC}\mv^k).
  }
  Using the smoothness of $H^*$ gives
  \eqe{ \label{1528}
    \norm{\ml_{k+1}^*-\ml^*} =& \norm{\nabla H^*\lt(-\sqrt{\mC}\my^*\rt) - \nabla H^*\lt(-\sqrt{\mC}\mv^k\rt)} \\
    \leq& \frac{1}{\mu_H}\norm{\sqrt{\mC}(\mv^k-\my^*)} \\
    \leq& \frac{\ove{\sqrt{\mC}}}{\mu_H}\norm{\mv^k-\my^*} \\
    \overset{\cref{analyze_form}}{\leq}& \frac{\ove{\sqrt{\mC}}}{\mu_H}\bigg(\lt(1+\frac{\sqrt{\kappa_{\fr}}-1}{\sqrt{\kappa_{\fr}}+1}\rt)\norm{\my^k-\my^*} \\
    &+ \frac{\sqrt{\kappa_{\fr}}-1}{\sqrt{\kappa_{\fr}}+1}\norm{\my^{k-1}-\my^*}\bigg).
  }

  According to the definitions of $\mx^*$ and $\mx_{k+1}^*$, we have
  \eqe{ \label{2249}
    -\A\T\ml^* &\in \nabla f(\mx^*) + \partial g(\mx^*), \\
    -\A\T\ml_{k+1}^* &\in \nabla f(\mx_{k+1}^*) + \partial g(\mx_{k+1}^*).
  }
  Recall that $f$ is $\mu_f$-strongly convex, so is $f+g$. By \cite[Lemma 3]{zhou2018fenchel}, we can obtain
  \eqe{ \label{1527}
    \norm{\mx_{k+1}^* - \mx^*}  &\overset{\cref{2249}}{\leq} \frac{1}{\mu_f}\norm{\A\T\lt(\ml_{k+1}^*-\ml^*\rt)} \\
    &\leq \frac{\os(\A)}{\mu_f}\norm{\ml_{k+1}^*-\ml^*}.
  }
  Recall that $\fr$ is $\mu_{\fr}$-strongly convex on $\Range{\sqrt{\mC}}$, and $\my^k, \my^* \in \Range{\sqrt{\mC}}$, we have
  \eqe{ \label{2348}
    \norm{\my^k-\my^*}^2 \leq \frac{2}{\mu_{\fr}}\pa{\fr(\my^k) - \fr(\my^*)}.
  }
  Also note that
  \eqe{ \label{2300}
    \norm{\mx\+-\mx^*}^2 \leq 2\norm{\mx\+-\mx_{k+1}^*}^2 + 2\norm{\mx_{k+1}^*-\mx^*}^2,
  }
  then the proof can be completed by combining \cref{2300}, \cref{x_error_condition}, \cref{1527}, \cref{1528} and \cref{2348}.
\end{appendix_proof}

\section{Proof of \cref{kappa_F}} \label{appendix_kappa_F}
\begin{appendix_proof}[\cref{kappa_F}]
  The first half of \cref{kappa_F} can be easily derived from \cref{F_rho_convex}, hence we only focus on the second half part.

  Inspired by \cite{arjevani2020ideal}, we can also obtain $\mu_{\fr}$ and $L_{\fr}$ by examining the Hessian of $\fr$, provided that it is twice differentiable.
  Consider the standard Lagrangian dual function of \cref{dual_pro_consensus}:
  \eqe{
    F'(\my) = \inf_{\ml} \mph(\ml) + \mh^*(\ml) + \my\T\sqrt{\mC}\ml,
  }
  the Moreau envelope of $F = -F'$ is given by
  \eqe{
    M_{\rho F}(\my) = \inf_{\mw} F(\mw) + \frac{1}{2\rho}\norm{\mw-\my}^2.
  }
  Similar to \cite[Proposition 7]{arjevani2020ideal}, we can obtain the explicit form of the Hessian of $M_{\rho F}$.
  \begin{lemma}* \label{Hessian}
    If $\pa{\mph+\mh^*}^*$ is twice differentiable, then $M_{\rho F}$ is also twice differentiable, and
    \eqe{
      &\nabla^2 M_{\rho F}(\my)=\frac{1}{\rho} \I-\frac{1}{\rho^2}\bigg[\frac{1}{\rho} \I \\
      &+\sqrt{\mC} \nabla^2 \pa{\mph+\mh^*}^*\pa{-\sqrt{\mC} \prox_{\rho F}(\my)} \sqrt{\mC}\bigg]^{-1}, \ \forall \my \in \R^{np}.
      \nonumber
    }
  \end{lemma}

  Recall from the proof of \cref{1552} that $\mph+\mh^*$ is closed proper convex, we then have $\fr = M_{\rho F}$ by \cite[Lemma 5]{li2025smooth}.
  With \cref{Hessian} in hand, we can obtain $\mu_{\fr}$ and $L_{\fr}$ by analyzing the lower and upper bounds of $\nabla^2 \fr$.
  \begin{lemma} \label{convex_Moreau}
    Assume that \cref{convex,G} hold, $f_i$ and $h^*$ are twice differentiable, $g_i = 0$, and $\mph+\mh^*$ is $\mu'$-strongly convex and $L'$-smooth, then $\fr$ is $\mu_{\fr}$-strongly convex on $\Range{\sqrt{\mC}}$ and $L_{\fr}$-smooth, where $\mu_{\fr} = \frac{\ue{C}}{L' + \rho\ue{C}}$ and $\ L_{\fr} = \frac{\ove{C}}{\mu' + \rho\ove{C}}$.
  \end{lemma}
  \begin{proof}
    According to \cite[Corollary X.4.2.10]{hiriart1996convex}, if a function $\zeta: \R^d \rightarrow \R$ is strongly convex and twice differentiable, then $\zeta^*$ is also twice differentiable. Since $f_i$ is smooth, strongly convex and twice differentiable, and $g_i = 0$, $(f_i+g_i)^*$ is strongly convex, smooth and twice differentiable, so is $(f+g)^*$. Recall that $\mph(\ml) = (f+g)^*(-\A\T\ml)$, hence $\phi$ is twice differentiable. Also note that $h^*$ is twice differentiable and $\mph+\mh^*$ is $\mu'$-strongly convex, so $\pa{\mph+\mh^*}^*$ is also twice differentiable, which implies that \cref{Hessian} is applicable.

    Since $\mph+\mh^*$ is $\mu'$-strongly convex and $L'$-smooth, $\pa{\mph+\mh^*}^*$ is $\frac{1}{L'}$-strongly convex and $\frac{1}{\mu'}$-smooth, then we have
    \eqe{
      \frac{1}{L'}\I \leq \nabla^2 \pa{\mph+\mh^*}^*(x) \leq \frac{1}{\mu'}\I, \ \forall x \in \R^{np}.
    }
    Combining the above inequality with \cref{Hessian} completes the proof.
  \end{proof}

  According to the proofs of \cref{F_rho_smooth,F_rho_convex}, we have
  \begin{enumerate}
    \item $\mph+\mh^*$ is $\pa{\max_{i \in \sI}\frac{\os^2(A_i)}{\mu_i}+\frac{L_{h^*}}{n}}$-smooth;
    \item $\mh^*$ is $\frac{\mu_{h^*}}{n}$-strongly convex for the first case;
    \item $\mph$ is $\min_{i \in \sI}\frac{\mins^2(A_i)}{L_i}$-strongly convex for the second case.
  \end{enumerate}
  Therefore, we can complete the proof by applying \cref{convex_Moreau}.
\end{appendix_proof}

\section{Proof of \cref{1518}} \label{appendix_1518}
\begin{appendix_proof}[\cref{1518}]
  Since the solution to \cref{general_saddle_pro} (i.e., the saddle-point) exists, by \cite[Lemma 36.2]{rockafellar1970convex}, we have
  \eqa{
    &\min_{x \in \R^{d_x}}\max_{y \in \R^{d_y}} f_1(x) + f_2(x) + y\T Bx - g_1(y) -g_2(y) \\
    =&\max_{y \in \R^{d_y}}\min_{x \in \R^{d_x}} f_1(x) + f_2(x) + y\T Bx - g_1(y) -g_2(y) \\
    =&-\min_{y \in \R^{d_y}} \pa{\varphi(y) + g_2(y)}. \label{2158}
  }
  Define
  \eqe{
    x^*(y) = \arg\min_{x \in \R^{d_x}} f_1(x) + f_2(x) + y\T Bx.
  }
  Since $f_1+f_2$ is closed proper and $\mu_x$-strongly convex, it follows from \cref{conjugate_smooth} that $(f_1+f_2)^*$ is $\frac{1}{\mu_x}$-smooth and $x^*(y) = \nabla (f_1+f_2)^*(-B\T y)$.
  The strong convexity of $f_1+f_2$ further ensures, according to \cref{conjugate_smooth} and \cite[Theorem 27.1]{rockafellar1970convex}, the existence and uniqueness of of $x^*(y)$ for any $y \in \R^{d_y}$. By the optimality condition of $x^*(y)$, we have
  \eqe{
    \0 \in \nabla f_1\pa{x^*(y)}+B\T y + \partial f_2\pa{x^*(y)} = \partial_x \cL(x^*(y), y).
  }
  Using the $\mu_x$-strong convexity of $\cL$ w.r.t. $x$, we obtain
  \eqe{ \label{1015}
  \norm{x-x^*(y)} &\leq \min_{S_{f_2}(x) \in \partial f_2(x)} \frac{1}{\mu_x}\norm{\nabla f_1(x)+S_{f_2}(x)+B\T y} \\
  &= \frac{1}{\mu_x}\dist{\0, \partial_x \cL(x, y)}.
  }
  Since $f_1$ and $\varphi$ are both strongly convex, \cref{general_saddle_pro} admits a unique solution $(x^*, y^*)$, where $y^*$ is the unique solution to \cref{2158}. Thus, the optimality condition for $y^*$ yields
  \eqe{ \label{0040}
    \0 \in \nabla g_1(y^*) + S_{g_2}(y^*) - Bx^*(y^*).
  }
  The $\mu_{\varphi}$-strong convexity of $\varphi$ implies that $\varphi + g_2$ is also $\mu_{\varphi}$-strongly convex. Combining this with \cref{0040}, we obtain
  \eqe{ \label{1509}
    &\norm{y-y^*} \\
    \leq& \min_{S_{g_2}(y) \in \partial g_2(y)} \frac{1}{\mu_{\varphi}}\norm{\nabla g_1(y) + S_{g_2}(y) - Bx^*(y)} \\
    \leq& \frac{1}{\mu_{\varphi}}\bigg(\min_{S_{g_2}(y) \in \partial g_2(y)}\pa{\norm{\nabla g_1(y) + S_{g_2}(y) - Bx}} \\
    &+\os(B)\norm{x-x^*(y)}\bigg) \\
    \overset{\cref{1015}}\leq& \frac{1}{\mu_{\varphi}}\pa{\dist{\0, \partial_y \cL(x, y)} + \frac{\os(B)}{\mu_x}\dist{\0, \partial_x \cL(x, y)}}.
  }
  Furthermore, the $\frac{1}{\mu_x}$-smoothness of $(f_1+f_2)^*$ gives
  \eqe{
    &\norm{x^*(y) - x^*(y^*)} \\
    &= \norm{\nabla (f_1+f_2)^*(-B\T y) - \nabla (f_1+f_2)^*(-B\T y^*)} \\
    &\leq \frac{\os(B)}{\mu_x}\norm{y-y^*}.
  }
  Recognizing $x^* = x^*(y^*)$, we have
  \eqe{ \label{1510}
    \norm{x-x^*} &\leq \norm{x-x^*(y)}+\norm{x^*(y)-x^*(y^*)} \\
    &\leq \frac{1}{\mu_x}\dist{\0, \partial_x \cL(x, y)} + \frac{\os(B)}{\mu_x}\norm{y-y^*}.
  }
  Combining \cref{1509,1510} completes the proof.
\end{appendix_proof}

\section{Proof of \cref{2046}} \label{appendix_2046}
\begin{appendix_proof}[\cref{2046}]
  Let $\zeta (\ml) = \frac{\rho}{2}\ml\T\mC\ml + \ml\T\mz^k$. We can regard \cref{inner_pro} as a specific instance of \cref{general_saddle_pro} with the following setting:
  \eqe{
    f_1 \triangleq f, \ f_2 \triangleq g, \ B \triangleq \A, \ g_1 \triangleq \zeta, \ g_2 \triangleq \mh^*.
  }
  Under \cref{G,convex,1643}, we can easily verify that \cref{2255} holds, $f_1$ is $\mu_f$-strongly convex, and $\varphi$ is $\mu_H$-strongly convex, where $\mu_H$ is defined in \cref{F_rho_convex}. The proof can then be completed by applying \cref{1518}.
\end{appendix_proof}

\section{Proof of \cref{iD2A_complexity_convex}} \label{appendix_iD2A_complexity_convex}
\begin{appendix_proof}[\cref{iD2A_complexity_convex}]
  According to \cref{1528}, \cref{1135}, and \cref{2203}, we have
  \eqe{ \label{2318}
  &\norm{\ml_{k+1}^*-\ml^*}^2 \\
  <& C_1\bigg(\sqrt{2(\fr(\my^0) - \fr(\my^*))} \\
  &+ \frac{\sqrt{\frac{2}{\mu_{\fr}}}\sqrt{e_{\ml,1}\ove{C}}}{\sqrt{\theta_{\ml}}-\sqrt{1-\frac{1}{\sqrt{\kappa_{\fr}}}}}\lt(\frac{\theta_{\ml}}{1-\frac{1}{\sqrt{\kappa_{\fr}}}}\rt)^{\frac{k}{2}}\bigg)^2\lt(1-\frac{1}{\sqrt{\kappa_{\fr}}}\rt)^k \\
  \overset{\cref{1547}}\leq& C_2e_{\ml,1}\theta_{\ml}^k,
  }
  where $C_1 = \frac{2\ove{C}}{\mu^2_H\mu_{\fr}}\lt(\frac{\sqrt{\kappa_{\fr}}\pa{2+\sqrt{1-\frac{1}{\sqrt{\kappa_{\fr}}}}}}{\sqrt{\kappa_{\fr}}+1}\rt)^2$ and $C_2 = \frac{8C_1\ove{C}}{\mu_{\fr}\pa{\sqrt{\theta_{\ml}}-\sqrt{1-\frac{1}{\sqrt{\kappa_{\fr}}}}}^2}$.
  Combining \cref{lambda_error_condition,2318} yields
  \eqe{ \label{2251}
  \norm{\ml\+-\ml^*}^2 &< 2C_2e_{\ml,1}\theta_{\ml}^k + 2e_{\ml,k+1} \\
  &= C_3e_{\ml,1}\theta_{\ml}^k,
  }
  where $C_3 = 2C_2+2$.
  According to \cref{1527}, \cref{2318}, and \cref{1547}, we have
  \eqe{
  \norm{\mx_{k+1}^*-\mx^*}^2 < C_2e_{\mx,1}\theta_{\ml}^k,
  }
  it follows that
  \eqe{ \label{1520}
  \norm{\mx\+-\mx^*}^2 < 2C_2e_{\mx,1}\theta_{\ml}^k + 2e_{\mx,k+1} = C_3e_{\mx,1}\theta_{\ml}^k.
  }
  Then, we can obtain \cref{1545} by \cref{1520} and \cref{1547}, where
  \eqe{
    \mathcal{C}=\frac{\os^2(\A)}{\mu_f^2}\bigg(\frac{32\ove{C}\kappa_{\fr}\pa{2+\sqrt{1-\frac{1}{\sqrt{\kappa_{\fr}}}}}^2}{\mu^2_H\mu_{\fr}\pa{\sqrt{\kappa_{\fr}}+1}^2} \\
    +\frac{2\mu_{\fr}}{\ove{C}}\pa{\sqrt{\theta_{\ml}}-\sqrt{1-\frac{1}{\sqrt{\kappa_{\fr}}}}}^2\bigg).
    \nonumber
  }

  To prove \cref{1546}, it suffices to show that $\frac{\mathcal{C}_1\norm{\mx^k-\mx_{k+1}^*}^2 + \mathcal{C}_2\norm{\ml^k-\ml_{k+1}^*}^2}{\min\pa{e_{\mx,k+1}, e_{\ml,k+1}}}$ is bounded by a constant independent of $k$. By \cref{2318,2251}, we have
  \eqe{
    \frac{\norm{\ml^k-\ml_{k+1}^*}^2}{e_{\ml,k+1}} &\leq \frac{2}{{e_{\ml,k+1}}}\pa{\norm{\ml^k-\ml^*}^2 + \norm{\ml_{k+1}^*-\ml^*}^2} \\
    &< 2\pa{\theta_{\ml} C_2 + C_3}\frac{e_{\ml,k}}{{e_{\ml,k+1}}} \\
    &\leq \frac{2\pa{\theta_{\ml} C_2 + C_3}}{\theta_{\ml}}.
  }
  Similarly, we can obtain
  \eqe{
    \frac{\norm{\mx^k-\mx_{k+1}^*}^2}{e_{\mx,k+1}} < \frac{2\pa{\theta_{\ml} C_2 + C_3}}{\theta_{\ml}}.
  }
  It follows that
  \eqe{
    &\frac{\mathcal{C}_1\norm{\mx^k-\mx_{k+1}^*}^2 + \mathcal{C}_2\norm{\ml^k-\ml_{k+1}^*}^2}{\min\pa{e_{\mx,k+1}, e_{\ml,k+1}}} \\
    &\leq \mathcal{C}_1\max\pa{\frac{\norm{\mx^k-\mx_{k+1}^*}^2}{e_{\mx,k+1}}, \frac{\norm{\mx^k-\mx_{k+1}^*}^2}{e_{\ml,k+1}}} \\
    &+ \mathcal{C}_2\max\pa{\frac{\norm{\ml^k-\ml_{k+1}^*}^2}{e_{\ml,k+1}}, \frac{\norm{\ml^k-\ml_{k+1}^*}^2}{e_{\mx,k+1}}} \\
    &< \frac{2\pa{\theta_{\ml} C_2 + C_3}}{\theta_{\ml}}\bigg(\mathcal{C}_1\max\pa{1,\frac{\mu_f^2}{\os^2(\A)}} \\
    &+ \mathcal{C}_2\max\pa{1, \frac{\os^2(\A)}{\mu_f^2}}\bigg),
  }
  which completes the proof.
\end{appendix_proof}

\section{Applications of \cref{main_pro}} \label{appendix_applications}
\subsubsection{Decentralized Resource Allocation}
Consider a system comprising of $n_p$ energy-production nodes and $n_r$ energy-reserving nodes. A specific class of decentralized resource allocation problem (DRAP), known as the decentralized energy resource management problem, aims to optimize energy costs \cite{doostmohammadian2023distributed}. This can be formulated as
\eqe{ \label{16570}
  \min_{\substack{x_i \in \R, i = 1, \cdots, n_p; \\ y_j \in \R, j = 1, \cdots, n_r}} \  &\sum_{i=1}^{n_p} f^p_i(x_i) + \sum_{j=1}^{n_r} f^r_j(y_j) \\
  \text{s.t.} \                            &\sum_{i=1}^{n_p}x_i = b + \sum_{j=1}^{n_r}y_j,  \\
  &x_i \in \sX_i, i = 1, \cdots, n_p, \\
  &y_j \in \mathcal{Y}_j, j = 1, \cdots, n_r,
}
where $x_i$ and $f^p_i: \R \rightarrow \R$ denote the power state and the local cost function of the energy-production node $i$, respectively; $y_j$ and $f^r_j: \R \rightarrow \R$ are those of energy-consumption node $j$; and $\sX_i$ and $\mathcal{Y}_j$ are interval constraints.  The equality constraint ensures that the total generated power equals the sum of the power demand $b \in \R$ and the reserved power. The set constraints specify the ranges within which the power states must fall.
Introducing $n = n_p+n_r$ and
\eqe{
&f_i \triangleq f^p_i, \ g_i \triangleq \iota_{\sX_i}, \ A_i \triangleq [1], i = 1, \cdots, n_p; f_i \triangleq f^r_{i-n_p}, \\
&g_i \triangleq \iota_{\mathcal{Y}_{i-n_p}}, \ A_i \triangleq [-1], i = n_p+1, \cdots, n; \ h \triangleq \iota_{\set{b}},
\nonumber
}
we can reformulate \cref{16570} as \cref{main_pro}.

\subsubsection{Decentralized Model Predictive Control}
Consider a class of decentralized linear model predictive control problems given in \cite{yfantis2024distributed}:
\eqe{ \label{0319}
&\min_{\substack{x_i^0, \cdots, x_i^K \in \R^{d_{x_i}}; \\ u_i^0, \cdots, u_i^{K-1} \in \R^{d_{u_i}}, \\ i \in \sI}} \ \sum_{i=1}^n\pa{J_i^f(x_i^K) + \sum_{k=0}^{K-1}J_i(x_i^k, u_i^k)} \\
\text{s.t.} \ &x_i\+ = A_ix_i^k + B_iu_i^k, \ k = 0, \cdots, K-1, i \in \sI, \\
&x_i^0 = \tx_i(t_0), \ i \in \sI, \\
&x_i^k \in \sX_i \subseteq \R^{d_{x_i}},  \ k = 1, \cdots, K, i \in \sI, \\
&u_i^k \in \mathcal{U}_i \subseteq \R^{d_{u_i}},  \ k = 0, \cdots, K-1, i \in \sI, \\
&\sum_{i=1}^{n}R_iu_i^k \leq r^k, \ k = 0, \cdots, K-1,
}
where $x_i$ and $u_i$ are the states and control inputs of agent $i$ respectively, $J_i^f(x_i^K)$ and $J_i(x_i^k, u_i^k)$ denote the terminal and stage costs of agent $i$ respectively, $A_i \in \R^{d_{x_i} \times d_{x_i}}$ and $B_i \in \R^{d_{x_i} \times d_{u_i}}$ are the state and input matrices of agent $i$ respectively, $\tx_i(t_0)$ denotes the measured states of agent $i$ at the current time $t_0$, $R_i \in \R^{d_r \times d_{u_i}}$ represents the map from the control inputs $u_i$ of agent $i$ to the corresponding resource consumption or production, and $r^k \in \R^{d_r}$ is the maximum availability of resources at time $k$.

Problem \cref{0319} is a special case of \cref{main_pro}. To demonstrate this, we define
\eqe{
  y_i =& \coll{x_i^0, \cdots, x_i^K, u_i^0, \cdots, u_i^{K-1}} \in \R^{d_i}, \\
  d_i =& K(d_{x_i}+d_{u_i})+d_{x_i}, \mathcal{S}_i = \set{\tx_i(t_0)} \times \pa{\sX_i}^{K-1} \times \pa{\mathcal{U}_i}^{K-1}, \\
  Q_i =& \mat{ \0_{(Kd_{x_i}) \times d_{x_i}}, \I_{Kd_{x_i}}, \0_{(Kd_{x_i}) \times (Kd_{u_i})}} \\
  &- \mat{\I_K \otimes A_i, \0_{(Kd_{x_i}) \times d_{x_i}}, \I_K \otimes B_i},  \\
  [P_1, &\cdots, P_n] = \diag{Q_1, \cdots, Q_n}, \ P_i \in \R^{(K\sum_{i=1}^{n}d_{x_i}) \times d_{y_i}}, \\
  G_i =& \mat{\I_K \otimes R_i}\mat{\0_{(Kd_{u_i}) \times ((K+1)d_{x_i})}, \I_{Kd_{u_i}}}, \\
  \mathcal{Z} =& \set{\0_{K\sum_{i=1}^{n}d_{x_i}}} \times \prod_{k=0}^{K-1}\set{x \in \R^{d_r} | x \leq r^k}, i \in \sI.
  \nonumber
}
Using the above notations, we can reformulate \cref{0319} as \cref{main_pro} with the following definitions:
\eqe{
  &f_i(y_i) = J_i^f(x_i^K) + \sum_{k=0}^{K-1}J_i(x_i^k, u_i^k), \ g_i = \iota_{\mathcal{S}_i}, \\
  &A_i = \begin{bmatrix}
    P_i \\
    G_i
  \end{bmatrix}, i \in \sI; \ h = \iota_{\mathcal{Z}}.
}

\subsubsection{Decentralized Learning} \label{appendix_decen_learning}
Consider a dataset with a raw feature matrix $X' \in \R^{p \times (d-1)}$ and a label vector $y \in \R^p$, let $X = [X', \1_p] \in \R^{p \times d}$. When employing generalized linear models (GLMs) \cite{hardin2018generalized}, such as linear regression, logistic regression, Poisson regression, or constrained GLMs \cite{mcdonald1990fitting} like nonnegative least squares, to fit the dataset, the associated regularized empirical risk minimization (ERM) problem can be formulated as
\eqe{ \label{1628}
  \min_{\theta \in \R^d} f(\theta) + g(\theta) + \ell(X\theta),
}
where $\theta \in \R^d$ is the model parameter, $\ell: \R^p \rightarrow \exs$ is a convex but possibly nonsmooth loss function, $f: \R^d \rightarrow \R$ is a convex and smooth function, and $g: \R^d \rightarrow \exs$ is a convex but possibly nonsmooth function. For example, $f$ could represent an $\ell_2$ regularizer, while $g$ might be an $\ell_1$ regularizer, an indicator function $\iota_{\R^d_+}$, or a combination of these. Additionally, $\ell$ varies with different GLMs:
\begin{itemize}
  \item Linear regression: $\ell(z) = \frac{1}{2p}\norm{z-y}^2$;
  \item Huber regression: $\ell(z) = \frac{1}{p}L_{\delta}(z-y)$, where $L_{\delta}$ is the Huber loss function defined as
        $$L_{\delta}(x) = \lt\{
          \begin{aligned}
            \frac{1}{2}\norm{x}^2,      \             & \norm{x} \leq \delta, \\
            \delta\pa{\norm{x}-\frac{1}{2}\delta}, \  & \text{otherwise}.
          \end{aligned}\right.$$
  \item Logistic regression: $\ell(z) = \frac{1}{p}\sum_{j=1}^p\log\pa{1+e^{-y_jz_j}}$ ($y_j \in \set{1,-1}$);
  \item Poisson regression: $\ell(z) = \frac{1}{p}\sum_{j=1}^p\pa{e^{z_j} - y_jz_j + \log(y_j!)}$ ($y_j \geq 0$ is an integer).
\end{itemize}

In the context of vertical federated learning \cite{liu2024vertical}, each agent (party) holds only partial features of the global dataset. Specifically, the feature matrix is vertically partitioned as $X = [X_1, \cdots, X_n]$, where $X_i \in \R^{p \times d_i}$ (with $\sum_{i=1}^nd_i = d$) represents the local feature matrix of agent $i \in \sI$. As the private data of agent $i$, $X_i$ is not allowed to be transited to other agents, which is the key characteristic of federated learning.
It is important to note that $f$ and $g$ in the ERM problem \cref{1628} can typically be decomposed into a sum of local functions. For instance, when $\theta = [\theta_1, \cdots, \theta_n]$ with $\theta_i \in \R^{d_i}$, we have
\begin{itemize}
  \item $\ell_2$ regularizer: $\frac{1}{2}\norm{\theta}^2 = \sum_{i=1}^n\frac{1}{2}\norm{\theta_i}^2$;
  \item $\ell_1$ regularizer: $\norm{\theta}_1 = \sum_{i=1}^n\norm{\theta_i}_1$;
  \item Nonnegative indicator function: $\iota_{\R^d_+}(\theta) = \sum_{i=1}^n\iota_{\R^{d_i}_+}(\theta_i)$.
\end{itemize}
Considering the above setting and the decomposability of $f$ and $g$, we can reformulate the ERM problem \cref{1628} into the following decentralized form:
\eqe{ \label{1630}
  \min_{\theta_i \in \R^{d_i}, i \in \sI; z \in \R^p} \ &\sum_{i=1}^n\pa{f_i(\theta_i) + g_i(\theta_i)} + \ell(z) \\
  \text{s.t} \ &\sum_{i=1}^nX_i\theta_i = z,
}
which is clearly a special case of \cref{main_pro}.

\section{Additional Experiments} \label{appendix_experiments}
\begin{figure*}[tb]
  \begin{center}
    \includegraphics[scale=0.3]{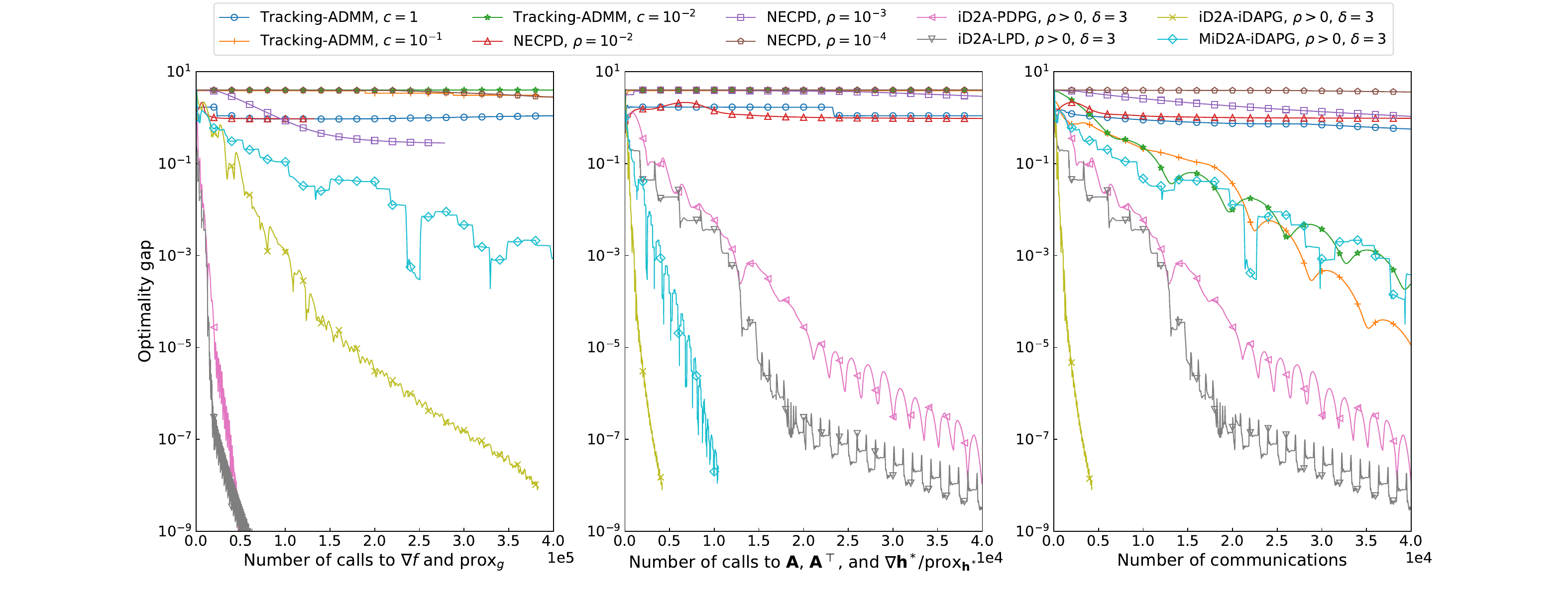}
    \caption {Results of Experiment III: Decentralized Resource Allocation ($n = 20, p = 10, d = 40, \kappa_C = 99, \kappa_f = 1000, \kappa_{A} = 8)$\protect\footnotemark.}
    \label{expt1}
  \end{center}
\end{figure*}
\footnotetext{$\kappa_{A} = \frac{\os^2(A)}{\us^2(A)}$, definitions of other condition numbers can be found in \cref{tab:complexity_h}. Since Tracking-ADMM and NECPD do not utilize $\nabla \mh^*$/$\prox_{\mh^*}$, the abscissa of the middle figure actually represents ``Number of calls to $\A$ and $\A\T$'' for these methods.}

\subsection{Experiment III: Decentralized Resource Allocation}
Consider the following decentralized resource allocation problem:
\eqe{ \label{0244}
  \min_{x_i \in \R^{d_i}, i \in \sI} \ &\sum_{i=1}^n \pa{\frac{1}{2}x_i\T P_ix_i + q_i\T x_i}, \\
  \text{s.t.} \ &\sum_{i=1}^{n}B_ix_i \leq b, \\
  &l_i \leq x_i \leq u_i, \ i \in \sI,
}
where $P_i \in \R^{d_i \times d_i}$ is a positive definite matrix, $q_i \in \R^{d_i}$, $B_i \in \R^{p \times d_i}$, $b \in \R^p$, $l_i, u_i \in \R^{d_i}$, $i \in \sI$. We can reformulate \cref{0244} as \cref{main_pro} by introducing the following definitions:
\eqe{ \label{0134}
  &f_i(x_i) \triangleq \frac{1}{2}x_i\T P_ix_i + q_i\T x_i, \ g_i \triangleq \iota_{\set{x_i \in \R^{d_i} | l_i \leq x_i \leq u_i}}, \\
  &A_i \triangleq B_i, i \in \sI; \ h \triangleq \iota_{\set{y \in \R^p | y \leq b}},
}
which obviously satisfy \cref{convex}. According to \cref{2024}, only iD2A and MiD2A can deal with the function $h$ defined in \cref{0134}. However, by introducing an auxiliary variable $y \in \R^p$, we can transform \cref{0244} into an equivalent problem that can be solved by other algorithms listed in \cref{2024}. Specifically, it can be easily verified that \cref{0244} is equivalent to
\eqe{ \label{02441}
  \min_{x_i \in \R^{d_i}, i \in \sI; y \in \R^p} \ &\sum_{i=1}^{n} \pa{\frac{1}{2}x_i\T P_ix_i + q_i\T x_i}, \\
  \text{s.t.} \ &\sum_{i=1}^{n}B_ix_i + y = b, \\
  &l_i \leq x_i \leq u_i, \ i \in \sI, \\
  &\0 \leq y \leq u_y,
}
if $u_y \in \R^p$ is sufficiently large \footnote{The upper bound of $y$ is useless for ensuring the equivalence between \cref{0244,02441}; it is introduced to make $y$ lie in a bounded set, because Tracking-ADMM and NECPD can only deal with those problems with bounded constrained sets.}. Similarly, we can reformulate \cref{02441} as \cref{main_pro} using the same definitions as in \cref{0134}, but with different $f_n$, $g_n$, $A_n$, and $h$:
\eqe{ \label{01341}
  &f_n(z) \triangleq \frac{1}{2}x_n\T P_nx_n + q_n\T x_n, \ g_n \triangleq \iota_{\set{z \in \R^{d_n+p} | l_z \leq z \leq u_z}}, \\
  &A_n \triangleq [B_n, \I], \ h \triangleq \iota_{\set{b}},
}
where $z = \coll{x_n, y}$, $\ l_z = \coll{l_n, \0_p}$, and $\ u = \coll{u_n, u_y}$. We should notice that $f_n$ in \cref{01341} is not strongly convex, which is significantly different from \cref{0134}.
According to \cref{2024}, Tracking-ADMM, NECPD, and Proj-IDEA are capable of solving \cref{main_pro} with the above definitions. However, since Proj-IDEA is a continuous-time algorithm and thus challenging to compare directly, we focus on the first two algorithms as the primary baselines for this experiment. Specifically, iD2A and MiD2A are applied to solve \cref{0244} directly, while Tracking-ADMM and NECPD are employed to address its equivalent problem \cref{02441}.

In this experiment, we randomly generate the data, namely $P_i$, $q_i$, $B_i$, $b$, $l_i$, $u_i$, and $u_y$, with dimensions set to $n = 20$, $p = 10$, and $d_i = 2$ for $i \in \sI$. Specifically, we construct $P_i$ by using $P_i = Q_i\Lambda_i Q_i\T$, where $Q_i \in \R^{d_i \times d_i}$ is an orthogonal matrix and $\Lambda_i \in \R^{d_i \times d_i}$ is diagonal matrix. The matrix $Q_i$ is generated by applying QR decomposition to a random matrix $M_i \in \R^{d_i \times d_i}$, whose elements are sampled from the standard normal distribution $\N(0, 1)$. The diagonal elements of $\Lambda_i$ are uniformly sampled from $[1, 1000]$. The elements of $B_i$ and $q_i$ are also sampled from $\N(0, 1)$. For $u_i$ and $b$, their elements are generated by taking the absolute values of numbers sampled from $\N(0, 1)$. We will implement multiple times if necessary (but usually not) to guarantee that $b > \0$. Finally, we set $l_i = \0$ and $u_y = 10^5\cdot b$. We guarantee that $l_i = \0$ and $b > \0$. This implies that: (1) the problem \cref{0244} is always feasible because $\0$ is a trivial feasible solution; and (2) \cref{strong_duality} always holds for \cref{0244}, since there always exists a point $\mx \in \ri{\dom{g}}$ sufficiently close to $\0$ such that $A\mx < b$.

Similar to DCPA and NPGA in Experiment I, Tracking-ADMM and NECPD also require a symmetric and doubly stochastic mixing matrix $W \in \R^{n \times n}$. We generate $W$ for Tracking-ADMM and NECPD, and $C$ for iD2A using the same method as in Experiment I. Tracking-ADMM has a single parameter $c>0$ to tune, while NECPD requires two parameters: $\rho>0$ and $0 < \eta \leq \rho$. Since there are no theoretical results for Tracking-ADMM and NECPD that allow us to determine the optimal parameters, we must manually tune them. Most parameters for iD2A are set according to \cref{1717}, but we still need to manually tune $\delta > 0$. For Tracking-ADMM and iD2A, we experiment with many different values of $c$ and $\delta$, and for NECPD, we also test various $\rho$ while setting $\eta = \rho$. However, to ensure clarity in the convergence plots, we only present the experimental results corresponding to the parameter configurations with better performance.
Additionally, at each iteration of Tracking-ADMM and NECPD, a constrained optimization problem needs to be solved exactly. This problem can be formulated as:
\eqe{ \label{0341}
  \min_{x \in \sX} \ \tf(x),
}
where $\sX \subseteq \R^m$ is a closed and compact set, and $\tf: \R^m \rightarrow \R$ is a convex function for Tracking-ADMM and a strongly convex function for NECPD. Under the setting of this experiment, $\tf$ is also smooth and $\iota_{\sX}$ is proximal-friendly. Thus, we can reformulate \cref{0341} as
\eqe{ \label{03411}
  \min_{x \in \R^m} \ \tg(x) = \tf(x) + \iota_{\sX},
}
which can be solved by APG with optimal complexity, provided that only first-order information ($\nabla \tf$ and $\prox_{\iota_{\sX}}$) is available. Therefore, we employ APG to solve the subproblem at each iteration of Tracking-ADMM and NECPD, using different step schemes for the convex case (Tracking-ADMM) and the strongly convex case (NECPD). It is important to note that Tracking-ADMM and NECPD utilize the exact solution of the subproblem, which is impractical in our case. Instead, we replace the exact solution $x^*$ with an approximate solution $\hat{x}$ that satisfies $\dist{\0, \partial \tf(\hat{x})} \leq 10^{-10}$. This ensures that $\norm{\hat{x}-x^*}$ remains sufficiently small. Besides, we employ a warm start strategy by using the approximate solution from the previous iteration as the initial solution for the current iteration's subproblem. This approach significantly reduces the total number of iterations required by APG.

The experimental results are presented in \cref{expt1}. From the figures, we observe that iD2A and MiD2A converge significantly faster than Tracking-ADMM and NECPD, w.r.t. the number of calls to $\nabla f$ and $\prox_g$, as well as the number of calls to $\A$, $\A\T$, and $\nabla \mh^*$/$\prox_{\mh^*}$. Although the communication complexity of Tracking-ADMM is slightly lower than that of MiD2A, iD2A still demonstrates the lowest communication complexity among all algorithms. Additionally, it is evident that the subproblem solver plays a crucial role in the performance of iD2A. If the communication cost or the cost associated with calls to $\A$, $\A\T$, and $\nabla \mh^*$/$\prox_{\mh^*}$ is higher than that to calls of $\nabla f$ and $\prox_g$, then iDAPG is the better choice. Conversely, if the situation is reversed, we should opt for PDPG or LPD.

\section{Additional Remarks} \label{appendix_remarks}
\begin{remark} \label{remark_1}
  We first consider Case 2), where $\A$ has full row rank. Let $M = \max\pa{\rho\ove{C}, \min_{i \in \sI}\frac{\mins^2(A_i)}{L_i}, \frac{\mine{\A\A\T + \rho L_f\mC}}{L_f}}$. When $\rho=0$, $\frac{\mine{\A\A\T + \rho L_f\mC}}{L_f} = \frac{\mins^2(\A)}{L_f} \leq \min_{i \in \sI}\frac{\mins^2(A_i)}{L_i}$, hence $M = \min_{i \in \sI}\frac{\mins^2(A_i)}{L_i}$. When $\rho>0$, it is shown in \cref{complexity} that an approximately optimal choice of $\rho$ is $\rho^*(C) = \frac{\max_{i \in \sI}\frac{\os^2(A_i)}{\mu_i} +\frac{L_{h^*}}{n}}{\ove{C}}$. Under this choice, $\rho\ove{C} = \max_{i \in \sI}\frac{\os^2(A_i)}{\mu_i} +\frac{L_{h^*}}{n}$. By Weyl's inequality \cite{horn2012matrix}, we have $\frac{\mine{\A\A\T + \rho L_f\mC}}{L_f} \leq \frac{1}{L_f}(\rho L_f\mine{\mC} + \ove{\A\A\T}) = \frac{\os^2(\A)}{L_f} < \rho\ove{C}$, hence $M = \max\pa{\rho\ove{C}, \min_{i \in \sI}\frac{\mins^2(A_i)}{L_i}} = \rho\ove{C}$. For Case 3), let $M' = \max\pa{\rho\ove{C}, \frac{\mine{\A\A\T + \rho L_f\mC}}{L_f}}$. By Weyl's inequality, we have $\frac{\mine{\A\A\T + \rho L_f\mC}}{L_f} \leq \frac{1}{L_f}(\mine{\A\A\T} + \rho L_f\ove{\mC}) \leq \rho\ove{C}$ for any $\rho \geq 0$, hence $M' = \rho\ove{C}$. Therefore, for both Cases 2) and 3), $\kappa_{\fr}$ is independent of $\mine{\A\A\T + \rho L_f\mC}$ in practice.
\end{remark}

\begin{remark} \label{remark_2}
  Alternatively, we can set $e_{\ml,1}$ to an upper bound of $\frac{\mu_{\fr}}{\ove{C}}\pa{\sqrt{\theta_{\ml}}-\sqrt{1-\frac{1}{\sqrt{\kappa_{\fr}}}}}^2\pa{\fr(\my^0) - \fr(\my^*)}$.
  Since $\fr$ is $\mu_{\fr}$-strongly convex on $\Range{\sqrt{\mC}}$ and differentiable, and $\my^0, \my^* \in \Range{\sqrt{\mC}}$, we have
  \eqe{
    &\fr(\my^0) - \fr(\my^*) \\
    \leq& \frac{1}{2\mu_{\fr}}\norm{\nabla \fr(\my^0)}^2 \\
    \leq& \frac{1}{2\mu_{\fr}}\pa{\norm{\sqrt{\mC}\tilde{\ml}}^2 + \ove{C}\norm{\tilde{\ml}-\ml^*(\my^0)}^2} \\
    \leq& \frac{1}{2\mu_{\fr}}\bigg(\norm{\sqrt{\mC}\tilde{\ml}}^2 + \frac{\ove{C}}{\mu_H^2}\Big(\norm{\nabla_{\ml} \cT(\mz^0, \tilde{\mx}, \tilde{\ml})} \\
    &+ \frac{\os(\A)}{\mu_f}\dist{\0, \partial_{\mx} \cT(\mz^0, \tilde{\mx}, \tilde{\ml})}\Big)^2\bigg) = C,
  }
  where $(\tilde{\mx}, \tilde{\ml})$ is an approximate solution to \cref{saddle_pro} with $\mz^0 = \sqrt{\mC}\my^0$, and the last inequality follows from the proof of \cref{2046}.
  Therefore, in practice, we can set $e_{\ml,1} = \frac{\mu_{\fr}}{\ove{C}}\pa{\sqrt{\theta_{\ml}}-\sqrt{1-\frac{1}{\sqrt{\kappa_{\fr}}}}}^2C$.
\end{remark}

\newpage
\section{Oracle Complexities of SOTA Algorithms for Solving Different Cases of Problem \cref{general_saddle_pro}} \label{appendix_oracle_complexity}
\begin{table}[h]
  \caption{\cite{li2025inexact}\protect\footnotemark The oracle complexities of SOTA first-order algorithms to solve different cases of Problem \cref{general_saddle_pro}, along with the corresponding lower bounds (if available).}
  \label{1229}
  \scriptsize
  \renewcommand{\arraystretch}{0.5}
  \centering
    \begin{threeparttable}[b]
      \begin{tabularx}{\textwidth}{ccC}
        \toprule
                                                             & Additional assumptions & Oracle complexity\tnote{1}                                                                                                                                                                                        \\
        \midrule
        \multicolumn{3}{c}{\textbf{Strongly-Convex-Concave Case: \cref{2255}}}                                                                                                                                                                                                                            \\
        \midrule
        LPD \cite{thekumparampil2022lifted}                  &                        & $\bO{\max\pa{\sqrt{L_x/\epsilon}, \frac{\os(B)}{\sqrt{\mu_x\epsilon}}, \sqrt{L_y/\epsilon}}}$                                                                                                     \\
        \midrule
        ABPD-PGS \cite{luo2024accelerated}                   &                        & $\bO{\max\pa{\sqrt{\kappa_x} \log \pa{\frac{1}{\epsilon}}, \frac{\os(B)}{\sqrt{\mu_x\epsilon}}, \sqrt{L_y/\epsilon}}}$                                                                           \\
        \midrule
        Lower bound\tnote{2} \cite{thekumparampil2022lifted} & $f_2 = 0$, $g_2 = 0$   & $\low{\max\pa{\sqrt{\kappa_x} \log \pa{\frac{1}{\epsilon}}, \frac{\os(B)}{\sqrt{\mu_x\epsilon}}, \sqrt{L_y/\epsilon}}}$                                                                                \\
        \midrule
        \multicolumn{3}{c}{\textbf{Strongly-Convex-Strongly-Concave Case: \cref{2255}, $g_1$ is $\mu_y$-strongly convex}}                                                                                                                                                                                 \\
        \midrule
        LPD \cite{thekumparampil2022lifted}, ABPD-PGS \cite{luo2024accelerated}                  &                        & $\bO{\max\pa{\sqrt{\kappa_x}, \sqrt{\kappa_{xy}}, \sqrt{\kappa_y}} \log \pa{\frac{1}{\epsilon}}}$                                                                                 \\
        \midrule
        APDG \cite{kovalev2022accelerated}                   & $f_2 = 0$, $g_2 = 0$   & $\bO{\max\pa{\sqrt{\kappa_x}, \sqrt{\kappa_{xy}}, \sqrt{\kappa_y}} \log \pa{\frac{1}{\epsilon}}}$ \\
        \midrule
        iDAPG \cite{li2025inexact}                                          &                        & \makecell{$\mathcal{A}$: $\tilde{\mO}\pa{\sqrt{\kappa_x}\max\pa{\sqrt{\kappa_{xy}}, \sqrt{\kappa_y}} \log \pa{\frac{1}{\epsilon}}}$ \tnote{3}                                               \\ $\mathcal{B}$: $\bO{\max\pa{\sqrt{\kappa_{xy}}, \sqrt{\kappa_y}} \log \pa{\frac{1}{\epsilon}}}$}                        \\
        \midrule
        Lower bound \cite{zhang2022lower}                    & $f_2 = 0$, $g_2 = 0$   & $\low{\max\pa{\sqrt{\kappa_x}, \sqrt{\kappa_{xy}}, \sqrt{\kappa_y}} \log \pa{\frac{1}{\epsilon}}}$                                                                                      \\
        \midrule
        \multicolumn{3}{c}{\textbf{Strongly-Convex-Full-Rank Case: \cref{2255}, $f_2 = 0$, $B$ has full row rank}}                                                                                                                                                                                        \\
        \midrule
        APDG \cite{kovalev2022accelerated}                   & $g_2 = 0$              & $\bO{\max\pa{\sqrt{\kappa_{xy'}}, \sqrt{\kappa_x\kappa_B}, \kappa_B} \log \pa{\frac{1}{\epsilon}}}$                                                         \\
        \midrule
        iDAPG \cite{li2025inexact}                                          &                        & \makecell{$\mathcal{A}$: $\tilde{\mO}\pa{\sqrt{\kappa_x}\max\pa{\sqrt{\kappa_{xy'}}, \sqrt{\kappa_x\kappa_B}} \log \pa{\frac{1}{\epsilon}}}$                          \\ $\mathcal{B}$: $\bO{\max\pa{\sqrt{\kappa_{xy'}}, \sqrt{\kappa_x\kappa_B}} \log \pa{\frac{1}{\epsilon}}}$}                        \\
        \midrule
        \multicolumn{3}{c}{\textbf{Strongly-Convex-Linear Case: \cref{2255}, $f_2 = 0$, $g_2 = 0$, $g_1$ is linear}}                                                                                                                                                                           \\
        \midrule
        Algorithm 1 \cite{salim2022optimal}               &                        & \makecell{$\mathcal{A}$: $\bO{\sqrt{\kappa_x} \log \pa{\frac{1}{\epsilon}}}$, $\mathcal{B}$: $\bO{\sqrt{\kappa_x\kappa_{B'}} \log \pa{\frac{1}{\epsilon}}}$}                          \\
        \midrule
        APDG \cite{kovalev2022accelerated}                   &                        & $\bO{\sqrt{\kappa_x\kappa_{B'}} \log \pa{\frac{1}{\epsilon}}}$
        \\
        \midrule
        iDAPG \cite{li2025inexact}                                          &                        & \makecell{$\mathcal{A}$: $\tilde{\mO}\pa{\kappa_x\sqrt{\kappa_{B'}} \log \pa{\frac{1}{\epsilon}}}$, $\mathcal{B}$: $\bO{\sqrt{\kappa_x\kappa_{B'}} \log \pa{\frac{1}{\epsilon}}}$} \\
        \midrule
        Lower bound \cite{salim2022optimal}                  &                        & $\mathcal{A}$: $\low{\sqrt{\kappa_x} \log \pa{\frac{1}{\epsilon}}}$, $\mathcal{B}$: $\low{\sqrt{\kappa_x\kappa_{B'}} \log \pa{\frac{1}{\epsilon}}}$                                   \\
        \midrule
        \multicolumn{3}{c}{\textbf{The Case that Satisfies \cref{2255,general_saddle_pro_assumption}}}                                                                                                                                                                                \\
        \midrule
        PDPG \cite{li2025inexact}        & $g_3 = 0$              & $\bO{\max\pa{\kappa_{xy^2}, \kappa_x\kappa_{xy^3}} \log \pa{\frac{1}{\epsilon}}}$                                                                           \\
        \midrule
        iDAPG \cite{li2025inexact} &                        & \makecell{$\mathcal{A}$: $\tilde{\mO}\pa{\sqrt{\kappa_x}\max\pa{\sqrt{\kappa_{xy^2}}, \sqrt{\kappa_x\kappa_{xy^3}}} \log \pa{\frac{1}{\epsilon}}}$ \\ $\mathcal{B}$: $\bO{\max\pa{\sqrt{\kappa_{xy^2}}, \sqrt{\kappa_x\kappa_{xy^3}}} \log \pa{\frac{1}{\epsilon}}}$}                        \\
        \midrule
        \multicolumn{3}{c}{\textbf{Dual-Strongly-Convex Case: \cref{2255}, $\varphi$ is $\mu_{\varphi}$-strongly convex}}                                                                                                                                     \\
        \midrule
        iDAPG \cite{li2025inexact} &                        & \makecell{$\mathcal{A}$: $\tilde{\mO}\pa{\sqrt{\kappa_x}\max\pa{\sqrt{\frac{L_y}{\mu_{\varphi}}}, \frac{\os(B)}{\sqrt{\mu_x\mu_{\varphi}}}} \log \pa{\frac{1}{\epsilon}}}$                            \\ $\mathcal{B}$: $\bO{\max\pa{\sqrt{\frac{L_y}{\mu_{\varphi}}}, \frac{\os(B)}{\sqrt{\mu_x\mu_{\varphi}}}} \log \pa{\frac{1}{\epsilon}}}$}                        \\        
        \bottomrule
      \end{tabularx}
      \begin{tablenotes}
        \item $\kappa_x = \frac{L_x}{\mu_x}$, $\kappa_y = \frac{L_y}{\mu_y}$, $\kappa_{B} = \frac{\os^2(B)}{\mins^2(B)}$, $\kappa_{B'} = \frac{\os^2(B)}{\us^2(B)}$, $\kappa_{xy} = \frac{\os^2(B)}{\mu_x\mu_y}$, $\kappa_{xy'} = \frac{L_xL_y}{\mins^2(B)}$, $\kappa_{xy^2} = \frac{L_xL_y}{\mine{BB\T + L_xP}}$, $\kappa_{xy^3} = \frac{\os^2(B)}{\mine{BB\T + L_xP}}$.
        \item [1] $\mathcal{A}$: $\nabla f_1$ and $\prox_{f_2}$; $\mathcal{B}$: $B$, $B\T$, $\nabla g_1$, and $\prox_{g_2}$. If only one complexity is provided, it suggests that the oracle complexities of $\mathcal{A}$ and $\mathcal{B}$ are the same.
        \item [2] When only one lower bound is provided, it suggests that only the lower bound of the maximum of the oracle complexities of $\mathcal{A}$ and $\mathcal{B}$ is available.
        \item [3] $\tilde{\mO}$ hides a logarithmic factor that depends on the problem parameters; refer to \cite{li2025inexact} for further details.
      \end{tablenotes}
    \end{threeparttable}
\end{table}
\footnotetext{Most of the results in this table are adapted from \cite[Table I]{li2025inexact}, except for the strongly-convex-concave case.}

\newpage
\section{Mapping Relations Between \cref{inner_pro} and \cref{general_saddle_pro}} \label{appendix_mapping}
\begin{table*}[h]
  \caption{The mapping relations between \cref{inner_pro} and \cref{general_saddle_pro} under \cref{strong_duality,G,convex,0129}.}
  \label{0154}
  \scriptsize
  \centering
  \begin{threeparttable}[b]
    \begin{tabularx}{\textwidth}{cC}
      \toprule
      Different cases of \cref{inner_pro}  & Corresponding cases of \cref{general_saddle_pro} \\
      \midrule
      $\rho > 0$                           & \makecell{Strongly-Convex-Concave Case            \\ (\textcircled{1}: $g_1 \triangleq  \zeta, \ g_2 \triangleq \mh^*$; \textcircled{3}: $g_1 \triangleq \mh^* + \zeta, \ g_2 \triangleq 0$)}                           \\
      \midrule
      $h^*$ is $\mu_{h^*}$-strongly convex & \makecell{Strongly-Convex-Strongly-Concave Case  \\ (\textcircled{1}: $g_1 \triangleq  \frac{\mu_{h^*}}{2n}\norm{\cdot}^2 + \zeta, \ g_2 \triangleq \mh^*-\frac{\mu_{h^*}}{2n}\norm{\cdot}^2$\tnote{1}; \textcircled{3}: $g_1 \triangleq  \mh^* + \zeta, \ g_2 \triangleq 0$)}                           \\
      \midrule
      $g_i = 0$, $A_i$ has full row rank   & \makecell{Strongly-Convex-Full-Rank Case         \\ (\textcircled{1}: $g_1 \triangleq  \zeta, \ g_2 \triangleq \mh^*$; \textcircled{3}: $g_1 \triangleq \mh^* + \zeta, \ g_2 \triangleq 0$)}                           \\
      \bottomrule
    \end{tabularx}
    \begin{tablenotes}
      \item We consider two possible cases of $h^*$: \textcircled{1} $h^*$ is proximal-friendly; \textcircled{3} $h^*$ is $L_h$-smooth. For all cases, it holds that $f_1 \triangleq f, \ f_2 \triangleq g, \ B \triangleq \A$.  The function $\zeta$ is defined by $\zeta (\ml) = \frac{\rho}{2}\ml\T\mC\ml + \ml\T\mz^k$.
      \item [1] We need to notice two facts here: (1) $\mh^*-\frac{\mu_{h^*}}{2n}\norm{\cdot}^2$ is convex if $h^*$ is $\mu_{h^*}$-strongly convex; (2) $\mh^*-\frac{\mu_{h^*}}{2n}\norm{\cdot}^2$ is proximal-friendly if $h^*$ is proximal-friendly.
    \end{tablenotes}
  \end{threeparttable}
\end{table*}

\begin{table*}[h]
  \caption{The mapping relations between \cref{inner_pro} and \cref{general_saddle_pro} under \cref{strong_duality,G,convex,hs_smooth,1643}.}
  \label{1823}
  \footnotesize
  \centering
  \begin{threeparttable}[b]
    \begin{tabularx}{\textwidth}{cC}

      \toprule
      Different cases of \cref{inner_pro}               & Corresponding cases of \cref{general_saddle_pro}                        \\
      \midrule
      Case 1\tnote{1}: $h^*$ is $\mu_{h^*}$-strongly convex              & \makecell{Strongly-Convex-Strongly-Concave Case                         \\ ($g_1 \triangleq \mh^* + \zeta, \ g_2 \triangleq 0$; If \textcircled{1}: $g_1 \triangleq \frac{\mu_{h^*}}{2n}\norm{\cdot}^2 + \zeta, \ g_2 \triangleq \mh^*-\frac{\mu_{h^*}}{2n}\norm{\cdot}^2$)}                           \\
      \midrule
      Case 2: $g_i = 0$, $A_i$ has full row rank                & \makecell{Strongly-Convex-Full-Rank Case                                \\ ($g_1 \triangleq \mh^* + \zeta, \ g_2 \triangleq 0$; If \textcircled{1}: $g_1 \triangleq \zeta, \ g_2 \triangleq \mh^*$)}                        \\
      \midrule
      Case 2.1: $g_i = 0$, $A_i$ has full row rank, $h^*$ is linear, $\rho = 0$ & \makecell{Strongly-Convex-Linear Case ($g_1 \triangleq \mh^* + \zeta$) \\ or Strongly-Convex-Full-Rank Case ($g_1 \triangleq 0, \ g_2 \triangleq \mh^* + \zeta$)} \\
      \midrule
      Case 3: $g_i = 0$, $A$ has full row rank, $\rho > 0$      & \makecell{The Case that Satisfies \cref{2255,general_saddle_pro_assumption}                     \\ ($g_1 \triangleq \mh^* + \zeta, \ g_2 \triangleq 0$; If \textcircled{1}: $g_1 \triangleq \zeta, \ g_2 \triangleq \mh^*$)}                        \\
      \bottomrule
    \end{tabularx}
    \begin{tablenotes}
      \item For all cases, it holds that $f_1 \triangleq f, \ f_2 \triangleq g, \ B \triangleq \A$. \textcircled{1} and $\zeta$ are the same as \cref{0154}.
      \item [1] These cases are identical to those in \cref{tab:complexity_h}.
    \end{tablenotes}
  \end{threeparttable}
\end{table*}

\newpage
\section{Decentralized Implementations of iD2A and MiD2A} \label{appendix_decen_iD2A}
\begin{algorithm}[h]
  \caption{Decentralized Implementation of iD2A}
  \label{alg:decen_iD2A}
  \begin{algorithmic}[1]
    \small
    \Require
    $K>0$, $\rho \geq 0$, $C$, $L_{\fr}$, $\mu_{\fr}$
    \Ensure
    $x_1^K, \cdots, x_n^K$
    \Statex \hspace{-\algorithmicindent} For agent $i = 1, \cdots, n$, implement:
    \State $x_i^0 = \lambda_i^0 = \0$, $z_i^0 = w_i^0 = \0$
    \State Set $\beta_k = \frac{\sqrt{\kappa_{\fr}}-1}{\sqrt{\kappa_{\fr}}+1}$ if $\mu_{\fr} > 0$, where $\kappa_{\fr} = \frac{L_{\fr}}{\mu_{\fr}}$; otherwise set $\beta_k = \frac{k}{k+3}$.
    \For {$k=0,\dots, K-1$}
    \If{$\rho = 0$}
    \State Solve
    \eqe{
      \min_{x_i \in \R^{d_i}}\max_{\lambda_i \in \R^p} f_i(x_i) + g_i(x_i) + \lambda_i\T A_ix_i - \pa{\frac{1}{n}h^*(\lambda_i) + \lambda_i\T z_i^k}
    }
    to obtain an inexact solution $\pa{x_i\+, \lambda_i\+}$.
    \Else
    \State Collaborate with neighbour agents to solve
    \eqe{
      \min_{\mx \in \R^d}\max_{\ml \in \R^{np}} f(\mx) + g(\mx) + \ml\T\A\mx - \pa{\mh^*(\ml) + \frac{\rho}{2}\ml\T\mC\ml + \ml\T\mz^k}
    }
    to obtain an inexact local solution $\pa{x_i\+, \lambda_i\+}$.
    \EndIf
    \State Send $\lambda_i\+$ to neighbour agents and receive $\lambda_j\+$ from them.
    \State $w_i\+ = z_i^k + \frac{1}{L_{\fr}}\sum_{j=1}^n c_{ij}\lambda_j\+$
    \State $z_i\+ = w_i\+ + \beta_k\lt(w_i\+-w_i^k\rt)$
    \EndFor
  \end{algorithmic}
\end{algorithm}

\begin{algorithm}[H]
  \caption{Subproblem Solving Procedure of iD2A-iDAPG for Agent $i$ When $\rho > 0$\protect\footnotemark}
  \begin{algorithmic}[1]
    \small
    \Require
    $L_{\varphi}$, $\mu_{\varphi}$, $z_i$, $x_i^0$, $\lambda_i^0$, stopping criterion
    \Ensure
    $x_i^k$, $\lambda_i^k$
    \State $v_i^0 = \lambda_i^0$
    \State Set $\beta_k = \frac{\sqrt{\kappa_{\varphi}}-1}{\sqrt{\kappa_{\varphi}}+1}$ if $\mu_{\varphi} > 0$, where $\kappa_{\varphi} = \frac{L_{\varphi}}{\mu_{\varphi}}$; otherwise set $\beta_k = \frac{k}{k+3}$.
    \State $k = 0$
    \While{$\pa{x_i^k, \lambda_i^k}$ does not satisfiy the stopping criterion}
    \State Solve
    \eqe{
      \min_{x_i \in \R^{d_i}} f_i(x_i) + g_i(x_i) + \dotprod{A_i\T v_i^k, x_i}
    }
    to obtain an inexact solution $x_i\+$.
    \State Send $v_i^k$ to neighbour agents and receive $v_j^k$ from them.
    \State $\lambda_i\+ = \prox_{\frac{1}{nL_{\varphi}} h^*}\pa{v_i^k - \frac{1}{L_{\varphi}}\pa{\rho\sum_{j=1}^n c_{ij}v_j^k + z_i - A_ix_i\+}}$
    \State $v_i\+ = \lambda_i\+ + \beta_k\pa{\lambda_i\+-\lambda_i^k}$
    \State $k = k+1$
    \EndWhile
  \end{algorithmic}
\end{algorithm}
\footnotetext{Here, we assume that the mapping relation between \cref{inner_pro} and \cref{general_saddle_pro} is given by $f_1 \triangleq f, \ f_2 \triangleq g, \ B \triangleq \A, \ g_1 \triangleq \zeta$, and $g_2 \triangleq \mh^*$, where $\zeta (\ml) = \frac{\rho}{2}\ml\T\mC\ml + \ml\T\mz^k$. The same applies to MiD2A.}

\newpage
\begin{algorithm}[h]
  \caption{Decentralized Implementation of MiD2A}
  \label{alg:decen_MiD2A}
  \begin{algorithmic}[1]
    \small
    \Require
    $T>0$, $K>0$, $\rho \geq 0$, $C$, $L_{\fr}$, $\mu_{\fr}$
    \Ensure
    $x_1^T, \cdots, x_n^T$
    \Statex \hspace{-\algorithmicindent} For agent $i = 1, \cdots, n$, implement:
    \State $x_i^0 = \lambda_i^0 = \0$, $z_i^0 = w_i^0 = \0$
    \State Set $\beta_k = \frac{\sqrt{\kappa_{\fr}}-1}{\sqrt{\kappa_{\fr}}+1}$ if $\mu_{\fr} > 0$, where $\kappa_{\fr} = \frac{L_{\fr}}{\mu_{\fr}}$; otherwise set $\beta_k = \frac{k}{k+3}$.
    \For {$k=0,\dots, T-1$}
    \If{$\rho = 0$}
    \State Solve
    \eqe{
      \min_{x_i \in \R^{d_i}}\max_{\lambda_i \in \R^p} f_i(x_i) + g_i(x_i) + \lambda_i\T A_ix_i - \pa{\frac{1}{n}h^*(\lambda_i) + \lambda_i\T z_i^k}
    }
    to obtain an inexact solution $\pa{x_i\+, \lambda_i\+}$.
    \Else
    \State Collaborate with neighbour agents to solve
    \eqe{
      \min_{\mx \in \R^d}\max_{\ml \in \R^{np}} f(\mx) + g(\mx) + \ml\T\A\mx - \pa{\mh^*(\ml) + \ml\T\mz^k + \frac{\rho}{2}\dotprod{\ml, \text{AcceleratedGossip}(\ml, \mC, K)}}
    }
    to obtain an inexact local solution $\pa{x_i\+, \lambda_i\+}$.
    \EndIf
    \State Collaborate with neighbour agents to execute $\text{AcceleratedGossip}(\ml\+, \mC, K)$ and obtain $\hat{\lambda}_i\+$.
    \State $w_i\+ = z_i^k + \frac{1}{L_{\fr}}\hat{\lambda}_i\+$
    \State $z_i\+ = w_i\+ + \beta_k\lt(w_i\+-w_i^k\rt)$
    \EndFor
  \end{algorithmic}
\end{algorithm}

\begin{algorithm}[H]
  \caption{Subproblem Solving Procedure of MiD2A-iDAPG for Agent $i$ When $\rho > 0$}
  \begin{algorithmic}[1]
    \small
    \Require
    $L_{\varphi}$, $\mu_{\varphi}$, $z_i$, $x_i^0$, $\lambda_i^0$, stopping criterion
    \Ensure
    $x_i^k$, $\lambda_i^k$
    \State $v_i^0 = \lambda_i^0$
    \State Set $\beta_k = \frac{\sqrt{\kappa_{\varphi}}-1}{\sqrt{\kappa_{\varphi}}+1}$ if $\mu_{\varphi} > 0$, where $\kappa_{\varphi} = \frac{L_{\varphi}}{\mu_{\varphi}}$; otherwise set $\beta_k = \frac{k}{k+3}$.
    \State $k = 0$
    \While{$\pa{x_i^k, \lambda_i^k}$ does not satisfiy the stopping criterion}
    \State Solve
    \eqe{
      \min_{x_i \in \R^{d_i}} f_i(x_i) + g_i(x_i) + \dotprod{A_i\T v_i^k, x_i}
    }
    to obtain an inexact solution $x_i\+$.
    \State Collaborate with neighbour agents to execute $\text{AcceleratedGossip}(\mv^k, \mC, K)$ and obtain $\hat{v}_i^k$.
    \State $\lambda_i\+ = \prox_{\frac{1}{nL_{\varphi}} h^*}\pa{v_i^k - \frac{1}{L_{\varphi}}\pa{\rho\hat{v}_i^k + z_i - A_ix_i\+}}$
    \State $v_i\+ = \lambda_i\+ + \beta_k\pa{\lambda_i\+-\lambda_i^k}$\
    \State $k = k+1$
    \EndWhile
  \end{algorithmic}
\end{algorithm}

\newpage
\section{Full Version of \cref{tab:complexity_h}} \label{appendix_full_table}
\begin{table*}[h]
  \caption{(Full version of \cref{tab:complexity_h}) Communication and oracle complexities of various algorithms across different scenarios.}
  \label{tab:complexity_h_full}
  \renewcommand{\arraystretch}{0.5}
  \centering
  \tabcolsep=-1.5pt
  \scriptsize
  \begin{threeparttable}[b]
    \begin{tabularx}{\textwidth}{>{\centering\arraybackslash}p{3cm}ccC}
      \toprule
                                                            & \makecell{Additional                                                                                                                                                                                                                                                                                     \\ assumptions}                                                                                                       & Communication complexity                                                                                                         & Oracle complexity\tnote{1} \\
      \midrule
      \multicolumn{4}{c}{\textbf{Case 1: \cref{G,convex,strong_duality,hs_smooth}, $h^*$ is $\mu_{h^*}$-strongly convex}}                                                                                                                                                                                                                                              \\
      \midrule
      DCPA \cite{alghunaim2021dual}, NPGA \cite{li2024npga} & \textcircled{1}      & $\bO{\max\pa{\kappa_f, \kappa_{pd}\kappa_C} \log \pa{\frac{1}{\epsilon}}}$                                                      & $\bO{\max\pa{\kappa_f, \kappa_{pd}\kappa_C} \log \pa{\frac{1}{\epsilon}}}$                                                                      \\
      \midrule
      \rowcolor{bgcolor}
      iD2A-LPD\tnote{2}, $\rho=0$                                 & \textcircled{1}      & $\bO{\sqrt{c_1\kappa_{pd}\kappa_C} \log \pa{\frac{1}{\epsilon}}}$                                                               & $\tilde{\mO}\pa{\sqrt{c_1\kappa_{pd}\kappa_C}\max\pa{\sqrt{\kappa_f}, \sqrt{\kappa_{pd}}} \log \pa{\frac{1}{\epsilon}}}$\tnote{3}                        \\
      \midrule
      \rowcolor{bgcolor}
      MiD2A-LPD, $\rho=0$                       & \textcircled{1}      & $\bO{\sqrt{c_1\kappa_{pd}\kappa_C} \log \pa{\frac{1}{\epsilon}}}$                                                               & $\tilde{\mO}\pa{\sqrt{c_1\kappa_{pd}}\max\pa{\sqrt{\kappa_f}, \sqrt{\kappa_{pd}}} \log \pa{\frac{1}{\epsilon}}}$                                \\
      \midrule
      \rowcolor{bgcolor}
      iD2A-LPD, $\rho>0$\tnote{4}                                 & \textcircled{1}      & $\tilde{\mO}\pa{\sqrt{\kappa_C}\max\pa{\sqrt{\kappa_f}, \sqrt{\kappa_{pd}}} \log \pa{\frac{1}{\epsilon}}}$                      & $\tilde{\mO}\pa{\sqrt{\kappa_C}\max\pa{\sqrt{\kappa_f}, \sqrt{\kappa_{pd}}} \log \pa{\frac{1}{\epsilon}}}$                                      \\
      \midrule
      \rowcolor{bgcolor}
      MiD2A-LPD, $\rho>0$                       & \textcircled{1}      & $\tilde{\mO}\pa{\sqrt{\kappa_C}\max\pa{\sqrt{\kappa_f}, \sqrt{\kappa_{pd}}} \log \pa{\frac{1}{\epsilon}}}$                      & $\tilde{\mO}\pa{\max\pa{\sqrt{\kappa_f}, \sqrt{\kappa_{pd}}} \log \pa{\frac{1}{\epsilon}}}$                                                     \\
      \midrule
      \rowcolor{bgcolor}
      iD2A-iDAPG, $\rho=0$                         & \textcircled{1}      & $\bO{\sqrt{c_1\kappa_{pd}\kappa_C} \log \pa{\frac{1}{\epsilon}}}$                                                               & $\mathcal{A}$: $\tilde{\mO}\pa{\sqrt{\kappa_f\kappa_C}\kappa_{pd} \log \pa{\frac{1}{\epsilon}}}$, $\mathcal{B}$: $\tilde{\mO}\pa{\sqrt{\kappa_C}\kappa_{pd} \log \pa{\frac{1}{\epsilon}}}$ \\
      \midrule
      \rowcolor{bgcolor}
      MiD2A-iDAPG, $\rho=0$                        & \textcircled{1}      & $\bO{\sqrt{c_1\kappa_{pd}\kappa_C} \log \pa{\frac{1}{\epsilon}}}$                                                               & $\mathcal{A}$: $\tilde{\mO}\pa{\sqrt{\kappa_f}\kappa_{pd} \log \pa{\frac{1}{\epsilon}}}$, $\mathcal{B}$: $\tilde{\mO}\pa{\kappa_{pd} \log \pa{\frac{1}{\epsilon}}}$ \\
      \midrule
      \rowcolor{bgcolor}
      iD2A-iDAPG, $\rho>0$                         & \textcircled{1}      & $\tilde{\mO}\pa{\sqrt{\kappa_{pd}\kappa_C} \log \pa{\frac{1}{\epsilon}}}$                                                       & $\mathcal{A}$: $\tilde{\mO}\pa{\sqrt{\kappa_f\kappa_{pd}\kappa_C} \log \pa{\frac{1}{\epsilon}}}$, $\mathcal{B}$:  $\tilde{\mO}\pa{\sqrt{\kappa_{pd}\kappa_C} \log \pa{\frac{1}{\epsilon}}}$ \\
      \midrule
      \rowcolor{bgcolor}
      MiD2A-iDAPG, $\rho>0$                        & \textcircled{1}      & $\tilde{\mO}\pa{\sqrt{\kappa_{pd}\kappa_C} \log \pa{\frac{1}{\epsilon}}}$                                                       & $\mathcal{A}$: $\tilde{\mO}\pa{\sqrt{\kappa_f\kappa_{pd}} \log \pa{\frac{1}{\epsilon}}}$, $\mathcal{B}$:  $\tilde{\mO}\pa{\sqrt{\kappa_{pd}} \log \pa{\frac{1}{\epsilon}}}$ \\
      \midrule
      \rowcolor{bgcolor}
      iD2A-LPD, $\rho = 0$                               &                      & $\bO{\sqrt{\kappa_C}\max\pa{\sqrt{c_1\kappa_{pd}}, \sqrt{\kappa_{h^*}}} \log \pa{\frac{1}{\epsilon}}}$                          & \makecell[c]{$\tilde{\mO}\big(\sqrt{\kappa_C}\max\pa{\sqrt{c_1\kappa_{pd}}, \sqrt{\kappa_{h^*}}}\cdot$                                         \\$\max\pa{\sqrt{\kappa_f}, \sqrt{\kappa_{pd}}, \sqrt{\kappa_{h^*}}} \log \pa{\frac{1}{\epsilon}}\big)$}           \\
      \midrule
      \rowcolor{bgcolor}
      MiD2A-LPD, $\rho = 0$                              &                      & $\bO{\sqrt{\kappa_C}\max\pa{\sqrt{c_1\kappa_{pd}}, \sqrt{\kappa_{h^*}}} \log \pa{\frac{1}{\epsilon}}}$                          & \makecell[c]{$\tilde{\mO}\big(\max\pa{\sqrt{c_1\kappa_{pd}}, \sqrt{\kappa_{h^*}}}\cdot$                                                        \\$\max\pa{\sqrt{\kappa_f}, \sqrt{\kappa_{pd}}, \sqrt{\kappa_{h^*}}} \log \pa{\frac{1}{\epsilon}}\big)$} \\
      \midrule
      \rowcolor{bgcolor}
      iD2A-LPD, $\rho > 0$                               &                      & $\tilde{\mO}\pa{\sqrt{\kappa_C}\max\pa{\sqrt{\kappa_f}, \sqrt{\kappa_{pd}}, \sqrt{\kappa_{h^*}}} \log \pa{\frac{1}{\epsilon}}}$ & $\tilde{\mO}\pa{\sqrt{\kappa_C}\max\pa{\sqrt{\kappa_f}, \sqrt{\kappa_{pd}}, \sqrt{\kappa_{h^*}}} \log \pa{\frac{1}{\epsilon}}}$                 \\
      \midrule
      \rowcolor{bgcolor}
      MiD2A-LPD, $\rho>0$                                &                      & $\tilde{\mO}\pa{\sqrt{\kappa_C}\max\pa{\sqrt{\kappa_f}, \sqrt{\kappa_{pd}}, \sqrt{\kappa_{h^*}}} \log \pa{\frac{1}{\epsilon}}}$ & $\tilde{\mO}\pa{\max\pa{\sqrt{\kappa_f}, \sqrt{\kappa_{pd}}, \sqrt{\kappa_{h^*}}} \log \pa{\frac{1}{\epsilon}}}$                                \\
      \midrule
      \rowcolor{bgcolor}
      iD2A-iDAPG, $\rho=0$                         &                      & $\bO{\sqrt{\kappa_C}\max\pa{\sqrt{c_1\kappa_{pd}}, \sqrt{\kappa_{h^*}}} \log \pa{\frac{1}{\epsilon}}}$                          & \makecell{$\mathcal{A}$: $\tilde{\mO}\pa{\sqrt{\kappa_f\kappa_C}\max\pa{\kappa_{pd}, \kappa_{h^*}} \log \pa{\frac{1}{\epsilon}}}$               \\ $\mathcal{B}$: $\tilde{\mO}\pa{\sqrt{\kappa_C}\max\pa{\kappa_{pd}, \kappa_{h^*}} \log \pa{\frac{1}{\epsilon}}}$} \\
      \midrule
      \rowcolor{bgcolor}
      MiD2A-iDAPG, $\rho=0$                        &                      & $\bO{\sqrt{\kappa_C}\max\pa{\sqrt{c_1\kappa_{pd}}, \sqrt{\kappa_{h^*}}} \log \pa{\frac{1}{\epsilon}}}$                          & \makecell{$\mathcal{A}$: $\tilde{\mO}\pa{\sqrt{\kappa_f}\max\pa{\kappa_{pd}, \kappa_{h^*}} \log \pa{\frac{1}{\epsilon}}}$                       \\ $\mathcal{B}$: $\tilde{\mO}\pa{\max\pa{\kappa_{pd}, \kappa_{h^*}} \log \pa{\frac{1}{\epsilon}}}$} \\
      \midrule
      \rowcolor{bgcolor}
      iD2A-iDAPG, $\rho>0$                         &                      & $\tilde{\mO}\pa{\sqrt{\kappa_C}\max\pa{\sqrt{\kappa_{pd}}, \sqrt{\kappa_{h^*}}} \log \pa{\frac{1}{\epsilon}}}$                  & \makecell{$\mathcal{A}$: $\tilde{\mO}\pa{\sqrt{\kappa_f\kappa_C}\max\pa{\sqrt{\kappa_{pd}}, \sqrt{\kappa_{h^*}}} \log \pa{\frac{1}{\epsilon}}}$ \\ $\mathcal{B}$:  $\tilde{\mO}\pa{\sqrt{\kappa_C}\max\pa{\sqrt{\kappa_{pd}}, \sqrt{\kappa_{h^*}}} \log \pa{\frac{1}{\epsilon}}}$} \\
      \midrule
      \rowcolor{bgcolor}
      MiD2A-iDAPG, $\rho>0$                        &                      & $\tilde{\mO}\pa{\sqrt{\kappa_C}\max\pa{\sqrt{\kappa_{pd}}, \sqrt{\kappa_{h^*}}} \log \pa{\frac{1}{\epsilon}}}$                  & \makecell{$\mathcal{A}$: $\tilde{\mO}\pa{\sqrt{\kappa_f}\max\pa{\sqrt{\kappa_{pd}}, \sqrt{\kappa_{h^*}}} \log \pa{\frac{1}{\epsilon}}}$         \\ $\mathcal{B}$:  $\tilde{\mO}\pa{\max\pa{\sqrt{\kappa_{pd}}, \sqrt{\kappa_{h^*}}} \log \pa{\frac{1}{\epsilon}}}$} \\
      \bottomrule
    \end{tabularx}
    \begin{tablenotes}
      \item $\mathcal{A}$: $\nabla f$ and $\prox_g$; $\mathcal{B}$: $\A$, $\A\T$, and $\nabla \mh^*$/$\prox_{\mh^*}$ ($\prox_{\mh^*}$/$\nabla \mh^*$ denotes $\prox_{\mh^*}$ if $h^*$ is proximal-friendly; otherwise, it denotes $\nabla \mh^*$). The oracle complexity of $\A$ and $\A\T$ represents the number of matrix-vector multiplications involving $\A$ and $\A\T$. \textcircled{1}: $h^*$ is proximal-friendly, \textcircled{2}: $h = \iota_{\set{b}}$, i.e., $h^*$ is a linear function. $c_1 = \frac{\max_{i \in \sI}\frac{\os^2(A_i)}{\mu_i}}{\os^2(\A)/\mu_f} \leq 1$, $c_2 = \frac{\min_{i \in \sI}\frac{\mins^2(A_i)}{L_i}}{\mins^2(\A)/L_f} \geq 1$, $\kappa_C = \frac{\ove{C}}{\ue{C}}$, $\kappa_f = \frac{L_f}{\mu_f}$, $\kappa_{pd} = \frac{\os^2(\A)}{\mu_f\mu_{h^*}/n}$, $\kappa_{pd'} = \frac{L_fL_{h^*}/n}{\mins^2(\A)}$, $\kappa_{pd^2} = \frac{L_fL_{h^*}/n}{\mine{\A\A\T + \rho L_f\mC}}$, $\kappa_{pd^3}
        = \frac{L_fL_{h^*}/n}{\mine{\A\A\T + \rho L_fP_K(\mC)}}$, $\kappa_{\A}  = \frac{\os^2(\A)}{\mins^2(\A)}$, $\kappa_{\A_{\rho}} = \frac{\os^2(\A)}{\mine{\A\A\T + \rho L_f\mC}}$, $\kappa_{\A'_{\rho}} = \frac{\os^2(\A)}{\mine{\A\A\T + \rho L_fP_K(\mC)}}$.
        \item [1] When only one oracle complexity is provided, it indicates that the oracle complexities of $\mathcal{A}$ and $\mathcal{B}$ are the same.
        \item [2] \textbf{``Algorithm A-Algorithm B'' indicates that Algorithm A utilizes Algorithm B as its subproblem solver.}
        \item [3] $\tilde{\mO}$ hides a logarithmic factor that depends on the problem parameters; for further details, please refer to \cref{iD2A_complexity_convex}.
      \item [4] When $\rho>0$, it indicates that $\rho = \rho^*(C) = \frac{\max_{i \in \sI}\frac{\os^2(A_i)}{\mu_i} +\frac{L_{h^*}}{n}}{\ove{C}}$ for iD2A and $\rho = \rho^*\pa{P_K(C)} = \frac{\max_{i \in \sI}\frac{\os^2(A_i)}{\mu_i} +\frac{L_{h^*}}{n}}{\ove{P_K(C)}}$ for MiD2A.
    \end{tablenotes}
  \end{threeparttable}
\end{table*}

\begin{table*}[h]
  \caption{Continued from \cref{tab:complexity_h_full}.}
  \renewcommand{\arraystretch}{0.5}
  \centering
  \tabcolsep=-1.5pt
  \scriptsize
  \begin{threeparttable}[b]
    \begin{tabularx}{\textwidth}{>{\centering\arraybackslash}p{3cm}ccC}
      \toprule
                                                                            & \makecell{Additional                                                                                                                                                                                                                                                                                                                                 \\ assumptions}                                                                                                       & Communication complexity                                                                                                         & Oracle complexity\tnote{1} \\
      \midrule
      \multicolumn{4}{c}{\textbf{Case 2: \cref{G,convex,strong_duality,hs_smooth}, $g_i = 0$ and $A_i$ has full row rank, $i \in  \sI$}}                                                                                                                                                                                                                                                                                           \\
      \midrule
      (80) in \cite{nedic2018improved}                                      & $A_i = \I$, \textcircled{2}                                                         & $\bO{\sqrt{\kappa_f\kappa_C}\log \pa{\frac{1}{\epsilon}}}$                                                                            & $\bO{\sqrt{\kappa_f\kappa_C}\log \pa{\frac{1}{\epsilon}}}$                                                             \\
      \midrule
      DCDA \cite{alghunaim2019proximal}                                     & \textcircled{2}                                                                     & $\mO\lt(\max\lt(\kappa_f\kappa_{\A}, \kappa_C\rt) \log \pa{\frac{1}{\epsilon}}\rt)$                                                   & $\mO\lt(\max\lt(\kappa_f\kappa_{\A}, \kappa_C\rt) \log \pa{\frac{1}{\epsilon}}\rt)$                                    \\
      \midrule
      \rowcolor{bgcolor}
      \makecell[c]{MiD2A-Algorithm 1 \cite{salim2022optimal}, \\ $\rho = 0$} & \textcircled{2}                  & $\bO{\sqrt{\frac{c_1}{c_2}}\sqrt{\kappa_f\kappa_{\A}\kappa_C} \log \pa{\frac{1}{\epsilon}}}$ & $\mathcal{A}$: $\tilde{\mO}\pa{\sqrt{\frac{c_1}{c_2}}\kappa_f\sqrt{\kappa_{\A}}\log \pa{\frac{1}{\epsilon}}}$, $\mathcal{B}$:  $\tilde{\mO}\pa{\sqrt{\frac{c_1}{c_2}}\kappa_f\kappa_{\A}\log \pa{\frac{1}{\epsilon}}}$ \\
      \midrule
      \rowcolor{bgcolor}
      MiD2A-APDG, $\rho>0$                                               & \textcircled{1}                                                                     & $\tilde{\mO}\pa{\sqrt{\kappa_C}\max\pa{\sqrt{\kappa_{f}\kappa_{\A}}, \kappa_{\A}} \log \pa{\frac{1}{\epsilon}}}$                      & $\tilde{\mO}\pa{\max\pa{\sqrt{\kappa_{f}\kappa_{\A}}, \kappa_{\A}} \log \pa{\frac{1}{\epsilon}}}$                      \\
      \midrule
      \rowcolor{bgcolor}
      MiD2A-iDAPG, $\rho = 0$ &       \textcircled{1}           & $\bO{\sqrt{\frac{c_1}{c_2}}\sqrt{\kappa_{f}\kappa_{\A}\kappa_C} \log \pa{\frac{1}{\epsilon}}}$ & $\mathcal{A}$: $\tilde{\mO}\pa{\kappa_{f}^{1.5}\kappa_{\A} \log \pa{\frac{1}{\epsilon}}}$, $\mathcal{B}$: $\tilde{\mO}\pa{\kappa_{f}\kappa_{\A} \log \pa{\frac{1}{\epsilon}}}$ \\      
      \midrule
      \rowcolor{bgcolor}
      MiD2A-iDAPG, $\rho>0$ & \textcircled{1}                  & $\tilde{\mO}\pa{\sqrt{\kappa_{f}\kappa_{\A}\kappa_C} \log \pa{\frac{1}{\epsilon}}}$ & $\mathcal{A}$: $\tilde{\mO}\pa{\kappa_{f}\sqrt{\kappa_{\A}} \log \pa{\frac{1}{\epsilon}}}$, $\mathcal{B}$: $\tilde{\mO}\pa{\sqrt{\kappa_{f}\kappa_{\A}} \log \pa{\frac{1}{\epsilon}}}$ \\

      \midrule
      \rowcolor{bgcolor}
      MiD2A-APDG, $\rho=0$                                               &                                                                                     & $\bO{\sqrt{\frac{1}{c_2}}\sqrt{\kappa_C}\max\pa{\sqrt{\kappa_{pd'}}, \sqrt{c_1\kappa_f\kappa_{\A}}} \log \pa{\frac{1}{\epsilon}}}$    & \makecell[c]{$\tilde{\mO}\Big(\sqrt{\frac{1}{c_2}}\max\pa{\sqrt{\kappa_{pd'}}, \sqrt{c_1\kappa_f\kappa_{\A}}}\cdot$ \\ $\max\pa{\sqrt{\kappa_{pd'}}, \sqrt{\kappa_{f}\kappa_{\A}}, \kappa_{\A}}\Big)$} \\
      \midrule
      \rowcolor{bgcolor}
      MiD2A-APDG, $\rho>0$                                               &                                                                                     & $\tilde{\mO}\pa{\sqrt{\kappa_C}\max\pa{\sqrt{\kappa_{pd'}}, \sqrt{\kappa_{f}\kappa_{\A}}, \kappa_{\A}} \log \pa{\frac{1}{\epsilon}}}$ & $\tilde{\mO}\pa{\max\pa{\sqrt{\kappa_{pd'}}, \sqrt{\kappa_{f}\kappa_{\A}}, \kappa_{\A}} \log \pa{\frac{1}{\epsilon}}}$ \\
      \midrule
      \rowcolor{bgcolor}
      MiD2A-iDAPG, $\rho = 0$ &                  & $\bO{\sqrt{\frac{1}{c_2}}\sqrt{\kappa_C}\max\pa{\sqrt{\kappa_{pd'}}, \sqrt{c_1\kappa_f\kappa_{\A}}} \log \pa{\frac{1}{\epsilon}}}$ & \makecell{$\mathcal{A}$: $\tilde{\mO}\pa{\sqrt{\kappa_f}\max\pa{\kappa_{pd'}, \kappa_{f}\kappa_{\A}} \log \pa{\frac{1}{\epsilon}}}$ \\ $\mathcal{B}$: $\tilde{\mO}\pa{\max\pa{\kappa_{pd'}, \kappa_{f}\kappa_{\A}} \log \pa{\frac{1}{\epsilon}}}$} \\
      \midrule
      \rowcolor{bgcolor}
      MiD2A-iDAPG, $\rho>0$ &                  & $\tilde{\mO}\pa{\sqrt{\kappa_C}\max\pa{\sqrt{\kappa_{pd'}}, \sqrt{\kappa_{f}\kappa_{\A}}} \log \pa{\frac{1}{\epsilon}}}$ & \makecell{$\mathcal{A}$: $\tilde{\mO}\pa{\sqrt{\kappa_f}\max\pa{\sqrt{\kappa_{pd'}}, \sqrt{\kappa_{f}\kappa_{\A}}} \log \pa{\frac{1}{\epsilon}}}$ \\ $\mathcal{B}$: $\tilde{\mO}\pa{\max\pa{\sqrt{\kappa_{pd'}}, \sqrt{\kappa_{f}\kappa_{\A}}} \log \pa{\frac{1}{\epsilon}}}$} \\
      \midrule
      \multicolumn{4}{c}{\textbf{Case 3: \cref{G,convex,strong_duality,hs_smooth}, $g_i = 0$ and $A = [A_1, \cdots, A_n]$ has full row rank}}                                                                                                                                                                                                                                                                                      \\
      \midrule
      DCPA \cite{alghunaim2021dual}                                         & \textcircled{1}                                                                     & $\mO\lt(\kappa_f\kappa_{\A_{\rho}}\kappa_C \log \pa{\frac{1}{\epsilon}}\rt)$                                                          & $\mO\lt(\kappa_f\kappa_{\A_{\rho}}\kappa_C \log \pa{\frac{1}{\epsilon}}\rt)$                                           \\
      \midrule
      NPGA \cite{li2024npga}                                                & \textcircled{1}                                                                     & $\mO\lt(\max\lt(\kappa_f\kappa_{\A_{\rho}}, \kappa_C\rt) \log \pa{\frac{1}{\epsilon}}\rt)$                                            & $\mO\lt(\max\lt(\kappa_f\kappa_{\A_{\rho}}, \kappa_C\rt) \log \pa{\frac{1}{\epsilon}}\rt)$                             \\
      \midrule
      \rowcolor{bgcolor}
      iD2A-iDAPG, $\rho>0$ & \textcircled{1}                 & $\tilde{\mO}\pa{\sqrt{\kappa_f\kappa_{\A_{\rho}\kappa_C}} \log \pa{\frac{1}{\epsilon}}}$ & $\mathcal{A}$: $\tilde{\mO}\pa{\kappa_f\sqrt{\kappa_{\A_{\rho}}\kappa_C} \log \pa{\frac{1}{\epsilon}}}$, $\mathcal{B}$: $\tilde{\mO}\pa{\sqrt{\kappa_f\kappa_{\A_{\rho}}\kappa_C} \log \pa{\frac{1}{\epsilon}}}$ \\
      \midrule
      \rowcolor{bgcolor}
      MiD2A-iDAPG, $\rho>0$ & \textcircled{1}                 & $\tilde{\mO}\pa{\sqrt{\kappa_f\kappa_{\A'_{\rho}}\kappa_C} \log \pa{\frac{1}{\epsilon}}}$ & $\mathcal{A}$: $\tilde{\mO}\pa{\kappa_f\sqrt{\kappa_{\A'_{\rho}}} \log \pa{\frac{1}{\epsilon}}}$, $\mathcal{B}$: $\tilde{\mO}\pa{\sqrt{\kappa_f\kappa_{\A'_{\rho}}} \log \pa{\frac{1}{\epsilon}}}$ \\
      \midrule
      \rowcolor{bgcolor}
      iD2A-iDAPG, $\rho>0$ &                  & $\tilde{\mO}\pa{\sqrt{\kappa_C}\max\pa{\sqrt{\kappa_{pd^2}}, \sqrt{\kappa_f\kappa_{\A_{\rho}}}} \log \pa{\frac{1}{\epsilon}}}$ & \makecell{$\mathcal{A}$: $\tilde{\mO}\pa{\sqrt{\kappa_f\kappa_C}\max\pa{\sqrt{\kappa_{pd^2}}, \sqrt{\kappa_f\kappa_{\A_{\rho}}}} \log \pa{\frac{1}{\epsilon}}}$ \\ $\mathcal{B}$: $\tilde{\mO}\pa{\sqrt{\kappa_C}\max\pa{\sqrt{\kappa_{pd^2}}, \sqrt{\kappa_f\kappa_{\A_{\rho}}}} \log \pa{\frac{1}{\epsilon}}}$} \\
      \midrule
      \rowcolor{bgcolor}
      MiD2A-iDAPG, $\rho>0$ &                  & $\tilde{\mO}\pa{\sqrt{\kappa_C}\max\pa{\sqrt{\kappa_{pd^3}
      }, \sqrt{\kappa_f\kappa_{\A'_{\rho}}}} \log \pa{\frac{1}{\epsilon}}}$ & \makecell{$\mathcal{A}$: $\tilde{\mO}\pa{\sqrt{\kappa_f}\max\pa{\sqrt{\kappa_{pd^3}
      }, \sqrt{\kappa_f\kappa_{\A'_{\rho}}}} \log \pa{\frac{1}{\epsilon}}}$                                                                                                                                                                                                                                                                                                                                                        \\ $\mathcal{B}$: $\tilde{\mO}\pa{\max\pa{\sqrt{\kappa_{pd^3}
      }, \sqrt{\kappa_f\kappa_{\A'_{\rho}}}} \log \pa{\frac{1}{\epsilon}}}$}                                                                                                                                                                                                                                                                                                                                                       \\
      \bottomrule
    \end{tabularx}
    \begin{tablenotes}
      \item See notes in \cref{tab:complexity_h_full}.
    \end{tablenotes}
  \end{threeparttable}
\end{table*}

\end{document}